\documentclass{article}

\usepackage{graphicx}
\usepackage{natbib,amssymb}
\usepackage{amsmath}
\usepackage{amsthm}
\usepackage{epsfig}
\usepackage{epstopdf}
\usepackage{fancybox}
\usepackage{float}
\usepackage{fancyhdr}
\usepackage{ulem}
\usepackage{url}
\usepackage{color}
\usepackage{placeins}
\usepackage{multirow} 
\usepackage[inner=2.5cm,outer=2cm,bottom=2cm]{geometry}

\def \RR {\mathbb{R}}

\def \NN {\mathbb{N}}

\def \N {\mathcal{N}}

\def \EE {\mathbb{E}}
\def \indun {\mathbf{1}}

\def \Rt {\tilde{R}}
\def \Rb {\bar{R}}
\def \Kt {\tilde{K}}
\def \rt {\tilde{r}}
\def \barf {\bar{f}}
\def \hatty {\tilde{\hat{y}}}

\newtheorem{prop}{Proposition}[section]
\newtheorem{lem}[prop]{Lemma}
\newtheorem{theo}[prop]{Theorem}
\newtheorem{cond}[prop]{Condition}
\newtheorem{defi}[prop]{Definition}

\newtheorem{rem}[prop]{Remark}
\newtheorem{cor}[prop]{Corollary}

\newcommand{\argmin}{\mathop{\mathrm{argmin}}}
\newcommand{\diag}{\mathop{\mathrm{diag}}}
\newcommand{\var}{\mathop{\mathrm{var}}}

\newcommand{\Tr}{\mathop{\mathrm{Tr}}}

\begin{document}

\title{Asymptotic analysis of covariance parameter estimation for Gaussian processes in the misspecified case}

\author{Fran\c cois Bachoc \\
Institut de Math\'ematiques de Toulouse \\
Universit\'e Paul Sabatier, \\
118 Route de Narbonne, \\
31400 Toulouse \\
francois.bachoc@math.univ-toulouse.fr
}

\maketitle

\begin{abstract}
In parametric estimation of covariance function of Gaussian processes, it is often the case that the true covariance function does not belong to the parametric set used for estimation. This situation is called the misspecified case. In this case, it has been shown that, for irregular spatial sampling of observation points, Cross Validation can yield smaller prediction errors than Maximum Likelihood. Motivated by this observation,
we provide a general asymptotic analysis of the misspecified case, for independent and uniformly distributed observation points.
We prove that the Maximum Likelihood estimator asymptotically minimizes a Kullback-Leibler divergence, within the misspecified parametric set, while Cross Validation asymptotically minimizes the integrated square prediction error. In a Monte Carlo simulation, we show that the covariance parameters estimated by Maximum Likelihood and Cross Validation, and the corresponding Kullback-Leibler divergences and integrated square prediction errors, can be strongly contrasting. On a more technical level, we provide new increasing-domain asymptotic results for independent and uniformly distributed observation points.

\end{abstract}



\section{Introduction} \label{section:introduction}

Kriging models \citep{stein99interpolation,rasmussen06gaussian} consist in inferring the values of a Gaussian random field given observations at a finite set of observation points.
They have become a popular method for a
large range of applications, such as numerical code approximation \citep{sacks89design,santner03design} and calibration \citep{paulo12calibration} or global optimization \citep{jones98efficient}.

One of the main issues regarding Kriging is the choice of the covariance function
for the Gaussian process.
Indeed, a Kriging model yields an unbiased predictor with minimal variance and a correct predictive variance only if the correct covariance
function is used. The most common practice is to statistically estimate the covariance function, from a set
of observations of the Gaussian process, and to plug \citep[Ch.6.8]{stein99interpolation} the estimate in the Kriging equations.
Usually, it is assumed that the covariance function belongs to a given parametric family (see \cite{abrahamsen97review} for a review of classical families).
In this case, the estimation boils down to estimating the corresponding covariance parameters. For covariance parameter estimation, Maximum Likelihood (ML) is the most studied and used method, while Cross Validation (CV) \citep{sundararajan01predictive,zhang10kriging} is an alternative technique.

Consider first the case where the true covariance function of the Gaussian process belongs to the parametric family of covariance functions used for estimation, which we call the well-specified case. Then, it is shown in several references that ML should be preferred over CV. It is proved in \cite{stein90comparison} that for the estimation of a signal-to-noise ratio parameter of a Brownian motion, CV has twice the asymptotic variance of ML. 
In the situations treated by \cite{bachoc14asymptotic}, the asymptotic variance is also larger for CV than for ML.
Several numerical results, showing an advantage for ML over CV as well, are available, coming either from Monte Carlo studies as in \cite[Ch.3]{santner03design} or deterministic studies as in \cite{martin04use}.
The settings of both the above studies can arguably be classified in the well-specified case, since the interpolated functions are smooth, and the covariance structures are adapted, being Gaussian in \cite{martin04use} and having a free smoothness parameter in \cite{santner03design}.
Finally, in situations similar to the well-specified case, ML-type methods have been shown to be preferable over CV-type methods in \cite{Stein93spline} for estimation and prediction.

Consider now the case where the true covariance function of the Gaussian process does not belong to the parametric family of covariance functions used for estimation, which we call the misspecified case.
This can occur in many situations, given for example that it is frequent to enforce the smoothness parameter in the Mat\'ern model to an arbitrary value (e.g. $3/2$ in \cite{chevalier14fast}), which de facto makes the covariance model misspecified if the Gaussian process has a different order of smoothness.
In the misspecified case, \cite{bachoc13cross} shows that, provided the spatial sampling of observation points is not too regular,
CV can yield a smaller integrated square prediction error than ML.
In a context of spline approximation methods, \cite{Stein93spline} and \cite{Kou03efficiency} also suggest that CV-type methods can provide smaller prediction errors than ML-type methods under misspecification.

In this paper, we primarily aim at showing, in agreement with the preceding discussion, that CV can provide asymptotically optimal integrated square prediction errors under misspecification.
In this regard, the two most studied asymptotic frameworks in the Kriging literature are the increasing-domain and fixed-domain asymptotics \citep[p.62]{stein99interpolation}.
In increasing-domain asymptotics, the average density of observation points is bounded, so that the infinite sequence of observation points is unbounded. In fixed-domain asymptotics, this sequence is dense in a bounded domain.

In fixed-domain asymptotics, significant results are available concerning the estimation of the covariance function, and its influence on Kriging predictions and confidence intervals. In this asymptotic framework, two types of covariance parameters can be distinguished:
microergodic and non-microergodic covariance parameters. Following the definition
in \cite{stein99interpolation}, a covariance parameter is microergodic if two covariance functions are orthogonal whenever they differ for it (as in \cite{stein99interpolation}, we say that
two covariance functions are orthogonal if the two underlying Gaussian measures are orthogonal). Non-microergodic covariance parameters cannot be consistently estimated, but have no asymptotic influence on Kriging predictions and confidence intervals \citep{stein88asymptotically,stein90bounds,stein90uniform,zhang04inconsistent}. On the contrary,
there is a
fair amount of literature on consistent estimation of microergodic covariance parameters \citep{ying91asymptotic,ying93maximum,zhang04inconsistent,loh05fixed,anderes10consistent}.
Consistent estimation of microergodic parameters is shown, in some cases, to entail asymptotically optimal predictions and confidence intervals \citep{putter01effect}.

Nevertheless, a downside of fixed-domain asymptotics is that the results currently under reach, despite their significant insights, are restricted in terms of covariance model. For example, \cite{ying93maximum} addresses ML for the tensorized exponential model only and \cite{loh05fixed} addresses ML for the Mat\'ern $3/2$ covariance model only. 

Hence, in this paper, we work under increasing-domain asymptotics, in which case results can be proved for fairly general covariance models \citep{mardia84maximum,cressie93asymptotic,cressie96asymptotics,bachoc14asymptotic}.
In fact, generally speaking, under increasing-domain asymptotics, all (identifiable) covariance parameters
have a strong asymptotic influence on predictions \citep{bachoc14asymptotic} and can be consistently estimated with asymptotic normality \citep{mardia84maximum,bachoc14asymptotic}.
This is because increasing-domain asymptotics is characterized by a vanishing dependence between
observations from distant observation points, so that a large sample size gives more and more information about the covariance structure. Note that, beside Kriging, increasing-domain asymptotics is largely considered in spatial statistics \citep{lahiri16central,hallin09local}

The increasing-domain asymptotic setting we consider in this paper consists of $n$ independent observation points with uniform distribution on $[0,n^{1/d}]^d$, for $d \in \NN^*$. In Theorem \ref{theorem:oracle:CV}, we prove that CV asymptotically minimizes the integrated square prediction error, within the misspecified set of covariance functions used for estimation. On the other hand, we prove in Theorem \ref{theorem:oracle:ML} that ML asymptotically minimizes, the Kullback-Leibler divergence from the true covariance function, defined at the observation vector. This latter finding does not provide information on the prediction errors of the Gaussian process at new points, stemming from ML.
Thus, an asymptotic confirmation is given to the empirical finding of \cite{bachoc13cross}, that when the spatial sampling is not too regular, CV can provide smaller integrated square prediction errors than ML in the misspecified case. 

On a more technical level, we provide increasing-domain asymptotic results for matrix-form estimation criteria with independent and uniformly distributed observation points. To the best of our knowledge, this type of situation has not been addressed in the existing literature.

We conclude this paper by a Monte Carlo simulation, illustrating Theorems \ref{theorem:oracle:ML} and \ref{theorem:oracle:CV}. The simulation highlights that the ML and CV estimators can estimate radically different covariance parameters, and that their subsequent performances for the Kullback-Leibler divergence and the integrated square prediction error can be strongly contrasting. 

The rest of the paper is organized as follows. We present the context on parametric covariance function estimation in the misspecified case and on the spatial sampling in Section \ref{section:context}. We give the asymptotic optimality results for ML and CV in Section \ref{section:asymptotic:optimality}. We discuss the simulation results in Section \ref{section:MC}. All the proofs are given in the appendix.

Finally, note that one should be cautious about inferring from this paper that CV is preferable over ML in the misspecified case. Indeed, there exist other prediction scores than the integrated square prediction error (see \cite{gneiting07strictly,gneiting11making}) some of them also assessing the coverage of the confidence intervals obtained from the Kriging model.
The main contribution of this paper is to provide rigorous results for CV, relatively to the integrated square prediction error only, which is nonetheless a largely considered criterion for comparing predictors.

\section{Context} \label{section:context}

\subsection{Presentation and notation for the covariance model}

We consider a stationary Gaussian process $Y$ on $\RR^d$ with zero mean function and covariance function $K_{0}$. 
Noisy observations of $Y$ are obtained at the random points $X_1,...,X_n \in \RR^d$, for $n \in \NN^*$. That is, for $i=1,...,n$, we observe $y_i = Y(X_i) + \epsilon_i$, where $\epsilon = (\epsilon_1,...,\epsilon_n)^t$, $Y$ and $(X_1,...,X_n)$ are mutually independent and $\epsilon$ follows a $\N(0, \delta_0 I_n)$ distribution, with $\delta_0 \geq 0$ and $I_n$ the identity matrix of size $n$. The distribution of $(X_1,...,X_n)$ is specified and discussed in Condition \ref{cond:iid_obs_points} below.

The case where $Y$ is observed exactly is treated by this framework by letting $\delta_0 =0$. Otherwise, letting $\delta_0 >0$ can correspond for instance to measure errors \citep{bachoc14calibration} or to Monte Carlo computer experiments \citep{legratiet14regularity}. Note also that the case of a Gaussian process with discontinuous covariance function at $0$ (nugget effect) is mathematically equivalent to this framework if the observation points $X_1,...,X_n$ are two by two distinct. [This is the case in this paper, in an almost sure sense, see Condition \ref{cond:iid_obs_points}.]

Let $p \in \NN^*$ and let $\Theta$ be the compact subset $[\theta_{inf},\theta_{sup}]^p$ with $- \infty < \theta_{inf} < \theta_{sup} < + \infty$.
We consider a parametric model attempting to approximate the covariance function $K_0$ and the noise variance $\delta_0$,
$\{ ( K_{\theta} , \delta_{\theta} ), \theta \in \Theta \}$, with $K_{\theta}$ a stationary covariance function and $\delta_{\theta} >0$. We call the case where there exists $\theta_0 \in \Theta$ so that $(K_0,\delta_0) = (K_{\theta_0},\delta_{\theta_0})$ the well-specified case. The converse case, where $(K_{0},\delta_{0}) \neq (K_{\theta},\delta_{\theta})  $ for all $\theta \in \Theta$ is called the misspecified case.

The well-specified case has been extensively studied in the Gaussian process literature, see the references given in Section \ref{section:introduction}. 
Nevertheless, the misspecified case can occur in many practical applications. Indeed, even if we assume $\delta_{\theta} = \delta_{0}$ for all $\theta$, the standard covariance models $\{ K_{\theta} , \theta \in \Theta \}$ are often driven by a limited number of parameters and thus restricted in some ways. For instance, an existing practice (e.g. \cite{martin04use,conti10bayesian}) is to use the Gaussian covariance model, where $p=d+1$, $\Theta \subset (0,\infty)^p$, $\theta=(\sigma^2,\ell_1,...,\ell_d)$ and $K_{\theta}(t) = \sigma^2 \exp( - \sum_{i=1}^d  t_i^2/\ell_i^2)$. With the Gaussian covariance model, all the covariance functions $K_{\theta}$ generate Gaussian process realizations that are almost surely infinitely differentiable. Thus, the Gaussian model is de facto misspecified if the realizations of $Y$ have only a finite order of differentiability. [Note that the use of the Gaussian covariance model is dis-advised in several references, see \cite{stein99interpolation}.] In theory, the Mat\'ern model considered in Section \ref{section:MC} provides more flexibility by incorporating a tunable smoothness parameter $\nu > 0$. However, it is also common practice to enforce a priori this parameter $\nu$ to a fixed value (e.g. $3/2$ in \cite{chevalier14fast}).

In this paper, we are primarily interested in analyzing the misspecified case although the asymptotic results that are given in Section \ref{section:asymptotic:optimality} are valid for both the well-specified and misspecified cases.

We let $X = (X_1,...,X_n)$ be the random $n$-tuple of the $n$ observation points.
For $\theta \in \Theta$, we define the $n \times n$ random matrix $R_{\theta}$ by $\left(R_{\theta}\right)_{i,j} = K_{\theta} \left( X_i - X_j \right) + \delta_{\theta} \indun_{i=j}$. We define the $n \times n$ random matrix $R_{0}$ by $\left(R_{0}\right)_{i,j} = K_{0} \left( X_i - X_j \right) + \delta_{0} \indun_{i=j}$.
We define the random vector $y = (y_1,...,y_n)^t$ of size $n$ by $y_i = Y\left(X_i \right) + \epsilon_i$.
Then, conditionally to $X$, $y$ follows a $\N(0,R_0)$ distribution and is assumed to follow a $\N(0,R_{\theta})$ distribution under the covariance parameter $\theta$. 

\subsection{Maximum Likelihood and Cross Validation estimators}

The Maximum Likelihood (ML) estimator is defined by $\hat{\theta}_{ML} \in \argmin_{\theta} L_{\theta}$, where
\begin{equation}  \label{eq:MLtheta}
L_{\theta} := \frac{1}{n} \log{\left( \det{\left( R_{\theta} \right)} \right) } + \frac{1}{n} y^t R_{\theta}^{-1} y
\end{equation}
is the modified opposite log-likelihood.

\begin{rem}
For concision, we do not write explicitly the dependence of $R_{\theta}$, $R_0$, $y$ and $L_{\theta}$ on $X$, $n$, $Y$ and $\epsilon$.
We make the same remark for the CV criterion in \eqref{eq:CV:theta:with:sum} and \eqref{eq:CVtheta}.
\end{rem}

\begin{rem} \label{rem:several:global:maximizers}
In this paper, we allow the criterion \eqref{eq:MLtheta} to have more than one global minimizer, in which case, the asymptotic results of Section \ref{section:asymptotic:optimality} hold for any sequence of random variables $\hat{\theta}_{ML}$ minimizing it. The same remark can be made for the CV criterion \eqref{eq:CV:theta:with:sum}. We refer to Remark 2.1 in \cite{bachoc14asymptotic} for the existence of measurable minimizers of the ML and CV criteria. 
\end{rem}

Under several increasing-domain asymptotics settings, ML is consistent and asymptotically normal, with mean vector $0$ and covariance matrix the inverse of the Fisher information matrix. This is shown in \cite{mardia84maximum}, assuming either some convergence conditions on the covariance matrices and their derivatives or gridded observation points. Similar results are provided for Restricted Maximum Likelihood in \cite{cressie93asymptotic,cressie96asymptotics}. In \cite{bachoc14asymptotic} asymptotic normality is also shown for Maximum Likelihood, using only simple conditions on the covariance model and for observation points that constitute a randomly perturbed regular grid.

The Cross Validation (CV) estimator, minimizing the Leave One Out (LOO) mean square error is defined by $\hat{\theta}_{CV} \in \argmin_{\theta} CV_{\theta}$, with
\begin{equation} \label{eq:CV:theta:with:sum}
CV_{\theta} := \frac{1}{n} \sum_{i=1}^n \left( y_i - \hat{y}_{i,\theta} \right)^2,
\end{equation}
where $\hat{y}_{i,\theta} := \EE_{\theta|X}( y_i | y_1,...,y_{i-1},y_{i+1},...,y_n )$ is the LOO prediction
of $y_i$ with parameter $\theta$.
The conditional mean value $\EE_{\theta|X}$ denotes the expectation with respect to the distribution of $Y$ and $\epsilon$ with the covariance function $K_{\theta}$ and the variance $\delta_{\theta}$, given $X$, so that
$\EE_{\theta|X}( y_i | y_1,...,y_{i-1},y_{i+1},...,y_n ) = \EE_{\theta}( y_i |X, y_1,...,y_{i-1},y_{i+1},...,y_n )$.

Let $r_{i,\theta} = \left(K_{\theta}(X_i,X_1),...,K_{\theta}(X_i,X_{i-1}),K_{\theta}(X_i,X_{i+1}),...,K_{\theta}(X_i,X_n) \right)^t$. Define $r_{i,0}$ similarly with $K_0$. Define the $(n-1) \times (n-1)$ covariance matrix $R_{i,\theta}$ as the matrix extracted from $R_{\theta}$ by deleting its line and column $i$. Define $R_{i,0}$ similarly with $R_0$. Then,
with $y_{-i} = (y_1,...,y_{i-1},y_{i+1},...,y_n)^t$,
we have $\hat{y}_{i,\theta} = r_{i,\theta}^t R_{i,\theta}^{-1}y_{-i}$.

Note that $\hat{y}_{i,\theta}$ is invariant if $K_{\theta}$ and $\delta_{\theta}$ are multiplied by a common positive constant. Thus, the CV criterion \eqref{eq:CV:theta:with:sum} is designed to select only a correlation function $K_{\theta}/K_{\theta}(0)$ and a corresponding relative noise variance $\delta_{\theta} / K_{\theta}(0)$. In particular, the CV criterion \eqref{eq:CV:theta:with:sum} does not assess the validity of quantities like $\var_{\theta|X}(Y(t)|y)$, where $\var_{\theta|X}$ denotes the variance under parameter $\theta$ given X. Hence, the Kriging predictive confidence intervals obtained by CV can be unreliable.

There exist extensions of the CV criterion \eqref{eq:CV:theta:with:sum} taking into account $\var_{\theta|X}(y_i|y_1,...,y_{i-1},y_{i+1},...,y_n)$, and aiming in particular at selecting values for $K_{\theta}(0)$ and $\delta_{\theta}$ \citep{bachoc13cross}. There also exist alternative CV criteria, different from \eqref{eq:CV:theta:with:sum}, like the log predictive probability (\cite{rasmussen06gaussian}, chapter 5, \cite{zhang10kriging}, \cite{sundararajan01predictive}). Nevertheless, these extensions and alternatives shall not be investigated in this paper. We focus on the criterion \eqref{eq:CV:theta:with:sum}, for which only the predictors $\EE_{\theta|X}(Y(t)|y)$ of the values of $Y$ at new points $t$ are relevant. These predictors alone provide the same applicability as many regression techniques like kernel regression or neural network methods and
can be used in a wide range of applications. 

The criterion \eqref{eq:CV:theta:with:sum} can be computed with a single matrix inversion, by means of
virtual LOO formulas (see e.g \cite{ripley81spatial,dubrule83cross}).
These virtual LOO formulas yield, when writing $\diag\left(A\right)$ for the matrix obtained by setting to $0$ all the off diagonal elements of a square matrix $A$,
\begin{equation} \label{eq:CVtheta}
CV_{\theta} := \frac{1}{n} y^t R_{\theta}^{-1} \left( \diag\left(R_{\theta}^{-1}\right) \right)^{-2} R_{\theta}^{-1} y,
\end{equation}
which is useful both in practice (to compute the CV criterion quickly) and in the proofs for CV.

Finally, in \cite{bachoc14asymptotic} it is shown that, in the well-specified case, the CV estimator is consistent and asymptotically normal for estimating correlation parameters, under increasing-domain asymptotics with a randomly perturbed grid of observation points.

\subsection{Random spatial sampling} 

We consider an increasing-domain asymptotic framework where the observation points are independent and uniformly distributed, which constitutes the archetype of an irregular spatial sampling.

\begin{cond} \label{cond:iid_obs_points}
For all $n \in \NN^*$, the observation points $X_1,...,X_n$ are random and follow independently the uniform distribution on $[0,n^{1/d}]^d$. The variables $Y$, $(X_1,...,X_n)$ and $\epsilon$ are mutually independent.
\end{cond}

Condition \ref{cond:iid_obs_points} constitutes an increasing-domain asymptotic framework in the sense that the volume of the observation domain is $n$ and the average density of observation points is constant. Some authors define increasing-domain asymptotics by the condition that the minimum distance between two different observation points is bounded away from zero (e.g. \cite{zhang05toward}), which is not the case here.
In \cite{lahiri03central} and \cite{lahiri04asymptotic}, the term increasing-domain is also used, when points are sampled randomly on a domain with volume proportional to $n$.

\section{Asymptotic optimality results} \label{section:asymptotic:optimality}

\subsection{Technical assumptions}

We shall assume the following condition for the covariance function $K_0$, which is satisfied in all the most classical cases, and especially for the Mat\'ern covariance function. Let $|t| = \max_{i=1,...,d}|t_i|$.

\begin{cond} \label{cond:Kzero:deltazero}
  The covariance function $K_0$ is stationary and continuous on $\RR^d$. There exists $C_{0} < + \infty$ so that for $t \in \RR^d$,
\begin{equation*} 
\left|  K_{0}\left(t\right) \right| \leq \frac{C_{0}}{ 1+|t|^{d+1} }.
\end{equation*}
\end{cond}

Next, the following condition for the parametric set of covariance functions and noise variances is slightly non-standard but not restrictive. We discuss it below. 

\begin{cond}  \label{cond:Ktheta:deltatheta}
For all $\theta \in \Theta $, the covariance function $K_{\theta}$ is stationary.
For all fixed $t \in \RR^d$, $K_{\theta}(t)$ is $p+1$ times continuously differentiable with respect to $\theta$.
For all $i_1,...,i_p \in \NN$ so that $i_1+...+i_p \leq p+1$, there exists $A_{i_1,...,i_p} < + \infty$ so that for all $t \in \RR^d$, $\theta \in \Theta$,
\begin{equation*} 
\left| \frac{\partial^{i_1}}{\partial \theta_1^{i_1}}...\frac{\partial^{i_p}}{\partial \theta_p^{i_p}}  K_{\theta}\left(t\right) \right| \leq \frac{A_{i_1,...,i_p}}{ 1+|t|^{d+1} }.
\end{equation*}
There exists a constant $C_{inf}>0$ so that, for any $\theta \in \Theta$, $\delta_{\theta} \geq C_{inf}$.
 Furthermore, $\delta_{\theta}$ is $p+1$ times continuously differentiable with respect to $\theta$.
For all $i_1,...,i_p \in \NN$ so that $i_1+...+i_p \leq p+1$, there exists $B_{i_1,...,i_p} < + \infty$ so that for all $\theta \in \Theta$,
\begin{equation*} 
\left| \frac{\partial^{i_1}}{\partial \theta_1^{i_1}}...\frac{\partial^{i_p}}{\partial \theta_p^{i_p}}  \delta_{\theta} \right| \leq B_{i_1,...,i_p}.
\end{equation*}

\end{cond}

In Condition \ref{cond:Ktheta:deltatheta}, we require a differentiability order of $p+1$ for $K_{\theta}$ and $\delta_{\theta}$ with respect to $\theta$. In the related context of \cite{bachoc14asymptotic}, where a well-specified covariance model is studied, consistency of ML and CV can be proved with a differentiability order of $1$ only. [One can check that the proofs of Propositions 3.1 and 3.4 in \cite{bachoc14asymptotic} require only the first order partial derivatives of the Likelihood function.] The reason for this difference is that, as discussed after Theorem \ref{theorem:oracle:CV}, an additional technical difficulty is present here, compared to \cite{bachoc14asymptotic}. The specific approach we use requires the condition of differentiability order of $p+1$ and we leave open the question of relaxing it.
Note, anyway, that many parametric covariance models are infinitely differentiable with respect to the covariance parameters, especially the Mat\'ern model.
In Condition \ref{cond:Ktheta:deltatheta}, assuming that the covariance function and its derivatives vanish with distance with a polynomial rate of order $d+1$ is not restrictive. Indeed, many covariance functions vanish at least exponentially fast with distance.

Finally, the condition that the noise variance $\delta_{\theta}$ is lower bounded uniformly in $\theta$ is crucial for our proof methods. Since we address here noisy observations of Gaussian processes, this condition is reasonable so that the results of Section \ref{section:MC} can cover a large variety of practical situations, some of which are listed in Section \ref{section:context}. Note that even when the Gaussian process under consideration is observed exactly, it can be desirable to incorporate an instrumental positive term $\delta_{\theta}$ in the parametric model, for numerical reasons or for not interpolating exactly the observed values \citep{andrianakis12effect}. Thus, Condition \ref{cond:Ktheta:deltatheta} could also be considered for Gaussian processes that are observed without noise.

\subsection{Maximum Likelihood}

In this paper, the analysis of the ML estimator in the misspecified case is based on the Kullback-Leibler divergence of the distribution of $y$ assumed under $(K_{\theta},\delta_{\theta})$, for $\theta \in \Theta$, from the true distribution of $y$. More precisely, conditionally to $X$, $y$ has a $\N(0,R_0)$ distribution and is assumed to have a $\N(0,R_{\theta})$ distribution. The conditional Kullback-Leibler divergence of the latter distribution from the former is, after multiplication by $2/n$,

\begin{equation} \label{eq:KL:zero:theta}
D_{n,\theta} := \frac{1}{n} \left\{  \log\left(\det\left( R_{\theta} R_{0}^{-1} \right)\right) + \Tr\left( R_{0} R_{\theta}^{-1} \right) \right\} - 1.
\end{equation}

The normalized Kullback-Leibler divergence in \eqref{eq:KL:zero:theta} is equal to $0$ if and only if $R_{\theta}= R_0$ and is strictly positive otherwise. It is interpreted as an error criterion for using $(K_\theta,\delta_{\theta})$ instead of $(K_0,\delta_{0})$, when making inference on the Gaussian process $Y$.  

Note that $D_{n,\theta}$ is here appropriately scaled so that, if for a fixed $\theta$ $(K_{\theta},\delta_{\theta}) \neq (K_0,\delta_0)$, $D_{n,\theta}$ should generally not vanish, nor diverge to infinity under increasing-domain asymptotics. This can be shown for instance in the framework of \cite{bachoc14asymptotic}, by using the methods employed there. It is also well-known that, in the case of a regular grid of observation points for $d=1$, $D_{n,\theta}$ converges to a finite limit as $n \to +\infty$ \citep{azencott86series}. This limit is
twice the asymptotic Kullback information in \cite{azencott86series} and is positive if $(K_{\theta}(t),\delta_{\theta})$ differs from $(K_{0}(t),\delta_{0})$ for at least one point $t$ in the regular grid of observation points.
Similarly, in the spatial sampling framework of Condition \ref{cond:iid_obs_points}, we observe in the Monte Carlo simulations of Section \ref{section:MC} that the order of magnitude of \eqref{eq:KL:zero:theta} does not change when $n$ increases, for $(K_{\theta},\delta_{\theta}) \neq (K_0,\delta_0)$.

The following theorem shows that the ML estimator asymptotically minimizes the normalized Kullback-Leibler divergence.

\begin{theo} \label{theorem:oracle:ML}

Under Conditions \ref{cond:iid_obs_points}, \ref{cond:Kzero:deltazero} and \ref{cond:Ktheta:deltatheta}, we have, as $n \to \infty$,
\[
D_{n,\hat{\theta}_{ML}} = \inf_{\theta \in \Theta} D_{n,\theta} + o_p(1),
\]
where the $o_p(1)$ in the above display is a function of $X$ and $y$ only that goes to $0$ in probability as $n \to \infty$.
\end{theo}

Theorem \ref{theorem:oracle:ML} is in line with the well-known fact that, in the i.i.d. setting,
ML asymptotically minimizes the Kullback-Leibler divergence (which does not depend on sample size) from the true distribution, within a misspecified parametric model \citep{white82maximum}.
Theorem \ref{theorem:oracle:ML} conveys a similar message, with the normalized Kullback-Leibler divergence that depends on the spatial sampling. As discussed above, the infimum in Theorem \ref{theorem:oracle:ML} is typically lower bounded as $n \to \infty$ in the misspecified case.

Note that Theorem \ref{theorem:oracle:ML} can be shown, in increasing-domain asymptotics, under other spatial samplings than that of Condition \ref{cond:iid_obs_points} (e.g. for the randomly perturbed regular grid of \cite{bachoc14asymptotic}). Nevertheless, to the best of our knowledge, in the context of Condition \ref{cond:iid_obs_points}, Theorem \ref{theorem:oracle:ML} is not a simple consequence of the existing literature, and an original proof is provided in Section \ref{subsection:proof:ML}.

The Kullback-Leibler divergence is of course a central quality criterion for covariance parameters. Nevertheless, in the misspecified-case, Theorem \ref{theorem:oracle:ML} does not imply that ML is optimal for other common quality criteria, such as the integrated square prediction error introduced below. In addition, note that the Kullback-Leibler divergence addresses the distribution of the Gaussian process only at the observation points, thus providing no information on the inference of the values of $Y$ at new points, obtained from $(K_{\theta},\delta_{\theta})$. 

\subsection{Cross Validation}

Let us recall the notation $\EE_{\theta|X}(Y(t)| y) = \EE_{\theta}(Y(t)| y,X)$ and let $\hat{y}_{\theta}(t) = \EE_{\theta|X}(Y(t)| y)$.
With the $n\times 1$ vector $r_{\theta}(t)$ so that
$\left(r_{\theta}(t)\right)_j = K_{\theta}( t - X_j )$,
we have $\hat{y}_{\theta}(t) = r_{\theta}^t(t) R_{\theta}^{-1}y$.
Then, define the family of random variables
\begin{equation} \label{eq:MISE}
E_{n,\theta} = \frac{1}{n} \int_{[0,n^{1/d}]^d}  \left( \hat{y}_{\theta}(t) - Y(t) \right)^2 dt,
\end{equation}
where the integral is defined in the $L^2$ sense since $K_0$ is continuous. We call the criterion \eqref{eq:MISE} the integrated square prediction error. This criterion (or evaluations of it) is very commonly used, in particular for Gaussian process surrogate models of computer experiments (see e.g. \cite{marrel08efficient,gramacy15local}). More generally, the square prediction error is largely considered to evaluate predictors, see \cite{gneiting11making}.

It is natural to consider that the first objective of the CV estimator $\hat{\theta}_{CV}$ is to yield a small $E_{n,\hat{\theta}_{CV}}$. If the observation points $X_1,...,X_n$ are regularly spaced, then this objective might however not be fulfilled. Indeed, the principle of CV does not really have grounds in this case, since the LOO prediction errors are not representative of actual prediction errors for new points. This fact is only natural and has been noted in e.g. \cite{iooss10numerical} and \cite{bachoc13cross}. If however the observation points $X_1,...,X_n$ are not regularly spaced, then it is shown numerically in \cite{bachoc13cross} that the CV estimator $\hat{\theta}_{CV}$ can yield a small $E_{n,\hat{\theta}_{CV}}$ and, especially, smaller than $E_{n,\hat{\theta}_{ML}}$. The following theorem, which is the main contribution of this paper, supports this conclusion under increasing-domain asymptotics.

\begin{theo} \label{theorem:oracle:CV}

Under Conditions \ref{cond:iid_obs_points}, \ref{cond:Kzero:deltazero} and \ref{cond:Ktheta:deltatheta}, we have, as $n \to \infty$,
\[
E_{n,\hat{\theta}_{CV}} = \inf_{\theta \in \Theta} E_{n,\theta} + o_p(1),
\]
where the $o_p(1)$ in the above display is a function of $X$, $y$ and $Y$ only that goes to $0$ in probability as $n \to \infty$.
\end{theo}

In \eqref{eq:MISE}, we stress that $E_{n,\theta}$ and the observation vector $y$ are defined with respect to the same Gaussian process $Y$. Thus, Theorem \ref{theorem:oracle:CV} gives a guarantee for the estimator $\hat{\theta}_{CV}$ relatively to the predictions it yields for the actual Gaussian process at hand. Theorem \ref{theorem:oracle:CV} not only confirm that CV will not provide asymptotically larger integrated square prediction errors than ML, with independent and uniformly distributed observation points, it also show that these integrated square prediction errors will be asymptotically minimal, over all possible estimators.

The setting of the proof of Theorem \ref{theorem:oracle:CV} combines independent and uniformly distributed observation points with the matrix-form estimation criteria \eqref{eq:MLtheta} and \eqref{eq:CVtheta}. These criteria and their derivatives involve imbrications of covariance matrix derivatives and inverse covariance matrices, which can generally not be put in explicit matrix-free forms. To the best of our knowledge, this specific combination has not been addressed in the previous literature. 

Indeed, on the one hand, when matrix-form criteria like \eqref{eq:MLtheta} and \eqref{eq:CVtheta} are treated, it is assumed, implicitly or explicitly that there exists a positive minimal distance between two different observation points. This is the case in \cite{bachoc14asymptotic}. Also, \cite{mardia84maximum} and \cite{cressie93asymptotic,cressie96asymptotics} work under non-trivial assumptions on the covariance matrices involved, and show that these assumptions are fulfilled for examples of spatial samplings for which the minimal distance between two different observation points is bounded away from zero. This minimal distance assumption does not hold with independent observation points. Instead clusters of closely spaced observation points may appear. As a consequence, the maximum eigenvalues of
the covariance matrices and their derivatives are not upper bounded, even in probability, which brings new obstacles for the analysis of criteria like \eqref{eq:MLtheta} and \eqref{eq:CVtheta}. In addition, considering random observation points with no underlying grid structure makes it more challenging to control the fluctuations of functions of (random) covariance matrices, compared to Proposition D.7 of \cite{bachoc14asymptotic} for instance.

On the other hand, when independent and uniformly distributed observation points are considered (see e.g. \cite{lahiri03central,lahiri04asymptotic,lahiri06resampling}), the quantities of interest do not involve derivatives and inverse of $n \times n$ covariance matrices.

As a consequence, the proof we propose for Theorem \ref{theorem:oracle:CV} is original and we do not address the asymptotic distribution of the ML and CV estimators. We leave this problem open to further research. Note nevertheless that, in the misspecified case addressed here, the fact that the ML and CV estimators minimize two different criteria and are thus typically asymptotically different is, in our opinion, at least as important as their asymptotic distributions.

\begin{rem}
An important element in the proof of Theorem \ref{theorem:oracle:CV} is that the variable $t$ in the expression of the integrated square prediction error $E_{n,\theta}$ in \eqref{eq:MISE} plays the same role as a new point $X_{n+1}$, uniformly distributed on $[0,n^{1/d}]^d$ and independent of $(X_1,...,X_n)$. Hence, using the symmetry of $X_1,...,X_{n+1}$, for fixed $\theta$, the mean value of $E_{n,\theta}$ is equal to the mean value of a modification of the CV criterion $CV_{\theta}$ in \eqref{eq:CV:theta:with:sum}, where there are $n+1$ observation points instead of $n$. Thus, one can indeed expect that the CV estimator minimizing $CV_{\theta}$ also asymptotically minimizes $E_{n,\theta}$. [The challenging part for proving Theorem \ref{theorem:oracle:CV} is to control the deviations of the criteria $E_{n,\theta}$ and $CV_{\theta}$ from their mean values, uniformly in $\theta$.] This discussion is exactly the paradigm of CV, that uses the LOO errors as empirical versions of the actual prediction errors. On the other hand, if the observation points constitute for instance a regular grid, then the variable $t$ in $E_{n,\theta}$ has close to nothing in common with them, so that Theorem \ref{theorem:oracle:CV} would generally not hold. This stresses that CV is generally not efficient for regular sampling of observation points, as discussed above.
\end{rem}

\section{Monte Carlo simulation} \label{section:MC}

We illustrate Theorems \ref{theorem:oracle:ML} and \ref{theorem:oracle:CV} in a Monte Carlo simulation. 
We consider the Mat\'ern covariance model in dimension $d=1$. A covariance function on $\RR$ is Mat\'ern ($\sigma^2,\ell,\nu$) when it is written

\begin{equation*} 
K_{\sigma^2,\ell,\nu}\left(t\right) = \frac{\sigma^2}{\Gamma\left(\nu\right)2^{\nu-1}} \left(  2 \sqrt{\nu} \frac{|t|}{\ell} \right)^{\nu} K_{\nu} \left( 2 \sqrt{\nu} \frac{|t|}{\ell} \right),
\end{equation*}
with $\Gamma$ the Gamma function and $K_{\nu}$ the modified Bessel function of second order. The parameters $\sigma^2$, $\ell$ and $\nu$ are respectively the variance, correlation length and smoothness parameters. 
We refer to e.g \cite{stein99interpolation} for a presentation of the Mat\'ern covariance function.

In the simulation, the true covariance function of $Y$ is Mat\'ern ($\sigma_0^2,\ell_0,\nu_0$) with $\sigma_0^2 =1$, $\ell_0 = 3$ and $\nu_0=10$. This choice of $\nu_0$ corresponds to a smooth Gaussian process and enables, as we see below, to illustrate Theorems \ref{theorem:oracle:ML} and \ref{theorem:oracle:CV} in a more striking manner. The true noise variance is $\delta_0 = 0.25^2$.

For the covariance and noise variance model, it is considered that the smoothness parameter $\nu_0$ is known, that the noise variance $\delta_1$ is fixed, and that the parameter $\theta = (\sigma^2,\ell)$ is estimated by ML or CV. For both ML and CV, the optimization is restricted to the domain $\Theta = [0.1^2,10^2] \times [0.2,10]$. [We experience that the conclusions of the Monte Carlo simulation are the same if a larger optimization domain is considered.]
The well-specified case corresponds to $\delta_1 = \delta_0$ and the misspecified case corresponds to $\delta_1 = 0.1^2 \neq \delta_0$. This covariance model is representative of practical applications. Indeed, first it is common practice to fix the value of the smoothness parameter in the Mat\'ern model, as is discussed in Section \ref{section:context}. Second, when using Gaussian process models on experimental or natural data, it can often occur that field experts provide an a priori value for the noise variance (see e.g. \cite{bachoc14calibration}). The misspecified case we address corresponds to an underestimation of the noise variance, possibly because some sources of measurement errors have been neglected.

The Monte Carlo simulation is carried out as follows. For $n=100$ and $N=1000$ or $n=500$ and $N=200$ we repeat $N$ data generations, estimations and quality criterion computations and average the results. More specifically, we simulate $N$ independent realizations of the observation points, of the observation vector and of the Gaussian process on $[0,n]$, under the true covariance function and noise variance. For each of these $N$ realizations, we compute the ML and CV estimates under the well-specified and misspecified models. For each of these estimates of the form $\hat{\theta} = (\hat{\sigma}^2, \hat{\ell})$, we compute the corresponding criteria $D_{n,\hat{\sigma}^2,\hat{\ell}}$ and $E_{n,\hat{\sigma}^2,\hat{\ell}}$. 

In Figure \ref{fig:delta:n100} we report, for $n=100$, for the well-specified and misspecified cases and for ML and CV, the histograms of the estimates $\hat{\ell}$, and of the values of the error criteria $D_{n,\hat{\sigma}^2,\hat{\ell}}$ and $E_{n,\hat{\sigma}^2,\hat{\ell}}$. In the well-specified case, the conclusions are in agreement with the main message of previous literature:
Both estimators estimate the true $\ell_0 =3$ with reasonable accuracy and have error criteria that are relatively small. By also considering Table \ref{table:MCresults}, where the averages and standard deviations corresponding to Figures \ref{fig:delta:n100} and \ref{fig:delta:n500} are reported, we observe that ML performs better than CV in all aspects. The estimation error for $\ell$ and the normalized Kullback-Leibler divergence are significantly smaller for ML, while the integrated square prediction error is similar under ML and CV estimation, but nonetheless smaller for ML. 

The conclusions are however radically different in the misspecified case, as is implied by Theorems \ref{theorem:oracle:ML} and \ref{theorem:oracle:CV}. First, the ML estimates of $\ell$ are significantly smaller than in the well-specified case, and can even be equal to the lower-bound $0.2$. The ML estimates of $\sigma^2$ are not reported in Figure \ref{fig:delta:n100} to save space and are close to $1$, so that, approximately, the variance of the observations, as estimated by ML, is close to the true variance of the observations. 
The reason for these small estimates of $\ell$ by ML is the underestimation of the noise variance $\delta_0$, coupled with the large smoothness parameter $\nu_0$. Indeed, there exist pairs of closely spaced observation points for which the corresponding differences of observed values are large compared to $\delta_1$, so that for values of $\ell$ that are larger than those computed by ML, the criterion \eqref{eq:MLtheta} blows up, for all values of $\sigma^2$. 
[Using a value of $\sigma^2$ smaller or approximately equal to $1$ does not counterbalance the damaging impact on \eqref{eq:MLtheta} of these pairs of closely spaced observation points with large observed value differences. Increasing $\sigma^2$ over $1$ is also not optimal for \eqref{eq:MLtheta}, since on a large scale, the observations do have variances close to $1$.] This phenomenon for ML is all the more important when the smoothness parameter $\nu_0$ is large, which is why we choose here the value $\nu_0 = 10$ to illustrate it.
To summarize, ML gives an important weight to pairs of closely spaced observation points with large observation differences and consequently estimates small correlation lengths to explain, so to speak, these observation differences.

On the contrary for CV, if we consider only the predictions $\hat{y}_{\sigma^2,\ell}(t)$ at new points $t$ and the LOO predictions $\hat{y}_{i,\sigma^2,\ell}$, with $(\ell,\sigma^2) \in \Theta$, then the situation is virtually the same as if the model was well-specified. Indeed, the covariance matrices and vectors obtained from $\sigma^2$, $\ell$ and $\delta_0$ are equal to $\delta_0 / \delta_1$ time those obtained from $\sigma^2 \delta_1/\delta_0$, $\ell$ and $\delta_1$, so that the corresponding predictions are identical. Hence, in Figure \ref{fig:delta:n100}, the empirical distribution of $\hat{\ell}_{CV}$ is approximately the same between the well-specified and misspecified cases. In the misspecified case, we find that the empirical distribution of $\hat{\sigma}^2_{CV}$ (not reported in Figure \ref{fig:delta:n100} to save space) is $\delta_1 / \delta_0$ time that of the well-specified case. Of course, although the CV predictions are not damaged by the misspecified $\delta_1$, the CV estimations of other characteristics of the conditional distribution of $Y$ given the observed data are damaged. 
[For example the confidence intervals for $Y(t)$ obtained from the CV estimates are significantly too small.]

The histograms of $E_{n,\hat{\sigma}^2,\hat{\ell}}$ and $D_{n,\hat{\sigma}^2,\hat{\ell}}$ for ML and CV in Figure \ref{fig:delta:n100} confirm the discussion on the estimated parameters. For ML which estimates small correlation lengths, the error criteria $E_{n,\hat{\sigma}^2_{ML},\hat{\ell}_{ML}}$ are significantly larger than in the well-specified case (by an approximate factor of $3$ on average as seen in Table \ref{table:MCresults}). The error criteria $D_{n,\hat{\sigma}^2_{ML},\hat{\ell}_{ML}}$ also increase and become larger than these of both ML and CV in the well-specified case. For CV, the error criteria $E_{n,\hat{\sigma}^2_{CV},\hat{\ell}_{CV}}$ are, as discussed, as small as in the well-specified case and approximately $3$ times smaller on average than for ML, illustrating Theorem \ref{theorem:oracle:CV}. However, the error criteria $D_{n,\hat{\sigma}^2_{CV},\hat{\ell}_{CV}}$ are $3$ times larger for CV than for ML, in the misspecified case, illustrating Theorem \ref{theorem:oracle:ML}.

Finally, in Figure \ref{fig:delta:n500}, the settings are the same as for Figure \ref{fig:delta:n100} but for $n=500$. The relative differences between ML and CV are the same as for $n=100$. The estimates of $\ell$ under ML and CV have less variance than for $n=100$, and their histograms are approximately unimodal and symmetric.
Finally, for ML and CV in the misspecified case, $E_{n,\hat{\sigma}^2,\hat{\ell}}$ and $D_{n,\hat{\sigma}^2,\hat{\ell}}$ keep the same averages between $n=100$ and $n=500$. In the well-specified case, $E_{n,\hat{\sigma}^2,\hat{\ell}}$ also keeps the same average, while $D_{n,\hat{\sigma}^2,\hat{\ell}}$ becomes very small.
This is because $D_{n,\sigma_0^2,\ell_0} =0$ in the well-specified case, while $E_{n,\sigma_0^2,\ell_0}$ is non-zero and should not vanish to $0$ as $n \to \infty$, since the density of observation points in the prediction domain is constant with $n$.

\begin{figure} 
\centering
\begin{tabular}{ccc}
\includegraphics[height=4cm,width=4.5cm]{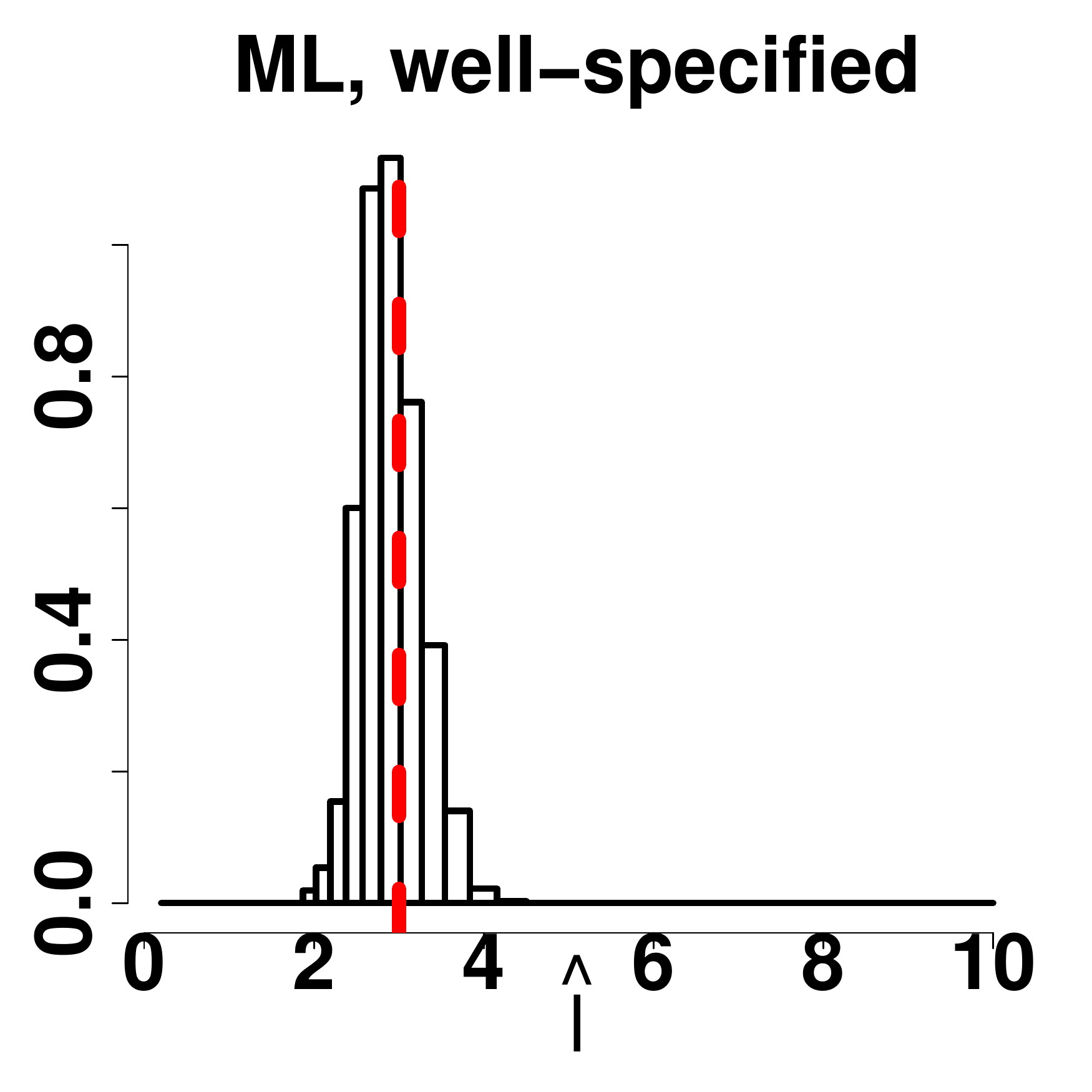}
&
\includegraphics[height=4cm,width=4.5cm]{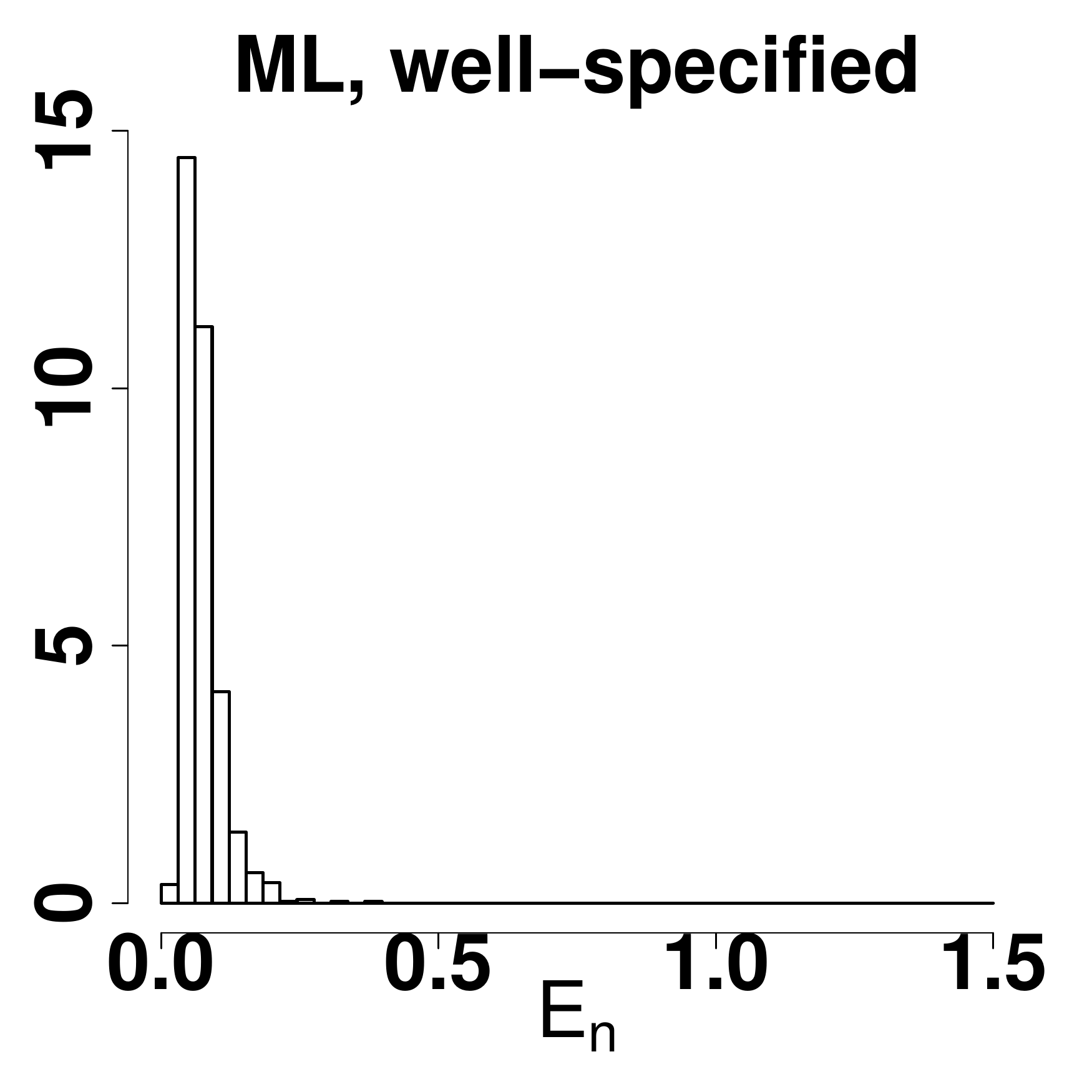}
&
\includegraphics[height=4cm,width=4.5cm]{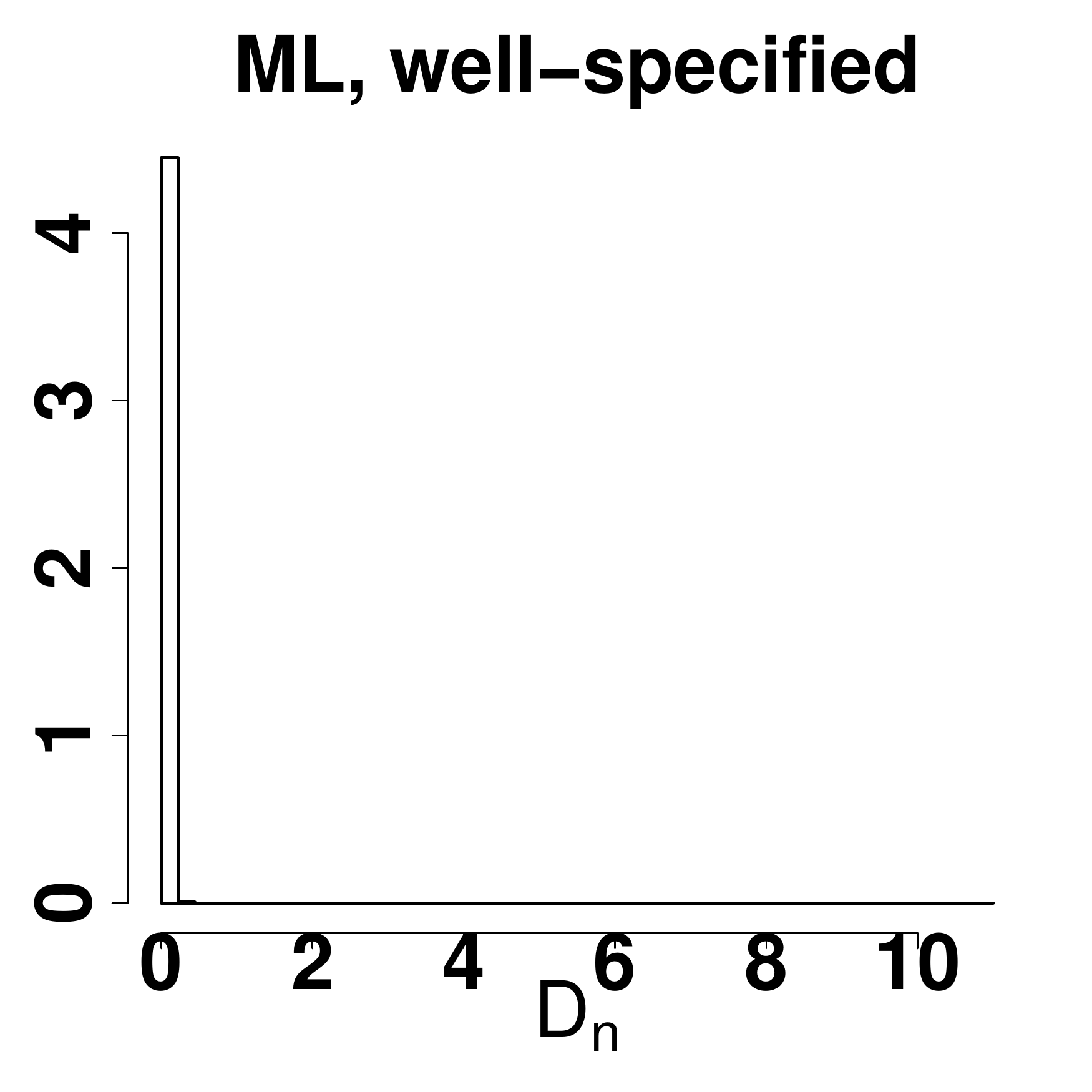}
\\
\includegraphics[height=4cm,width=4.5cm]{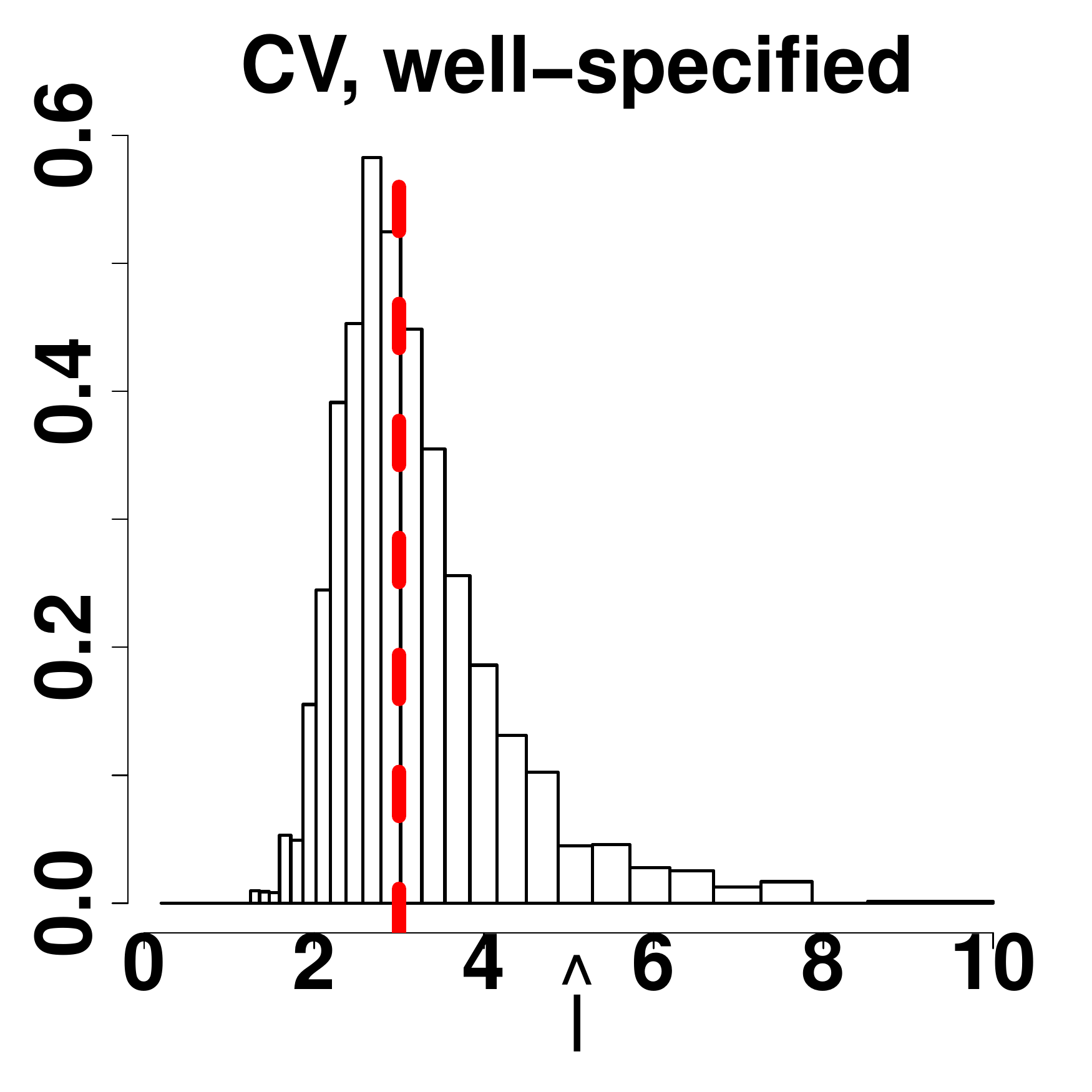}
&
\includegraphics[height=4cm,width=4.5cm]{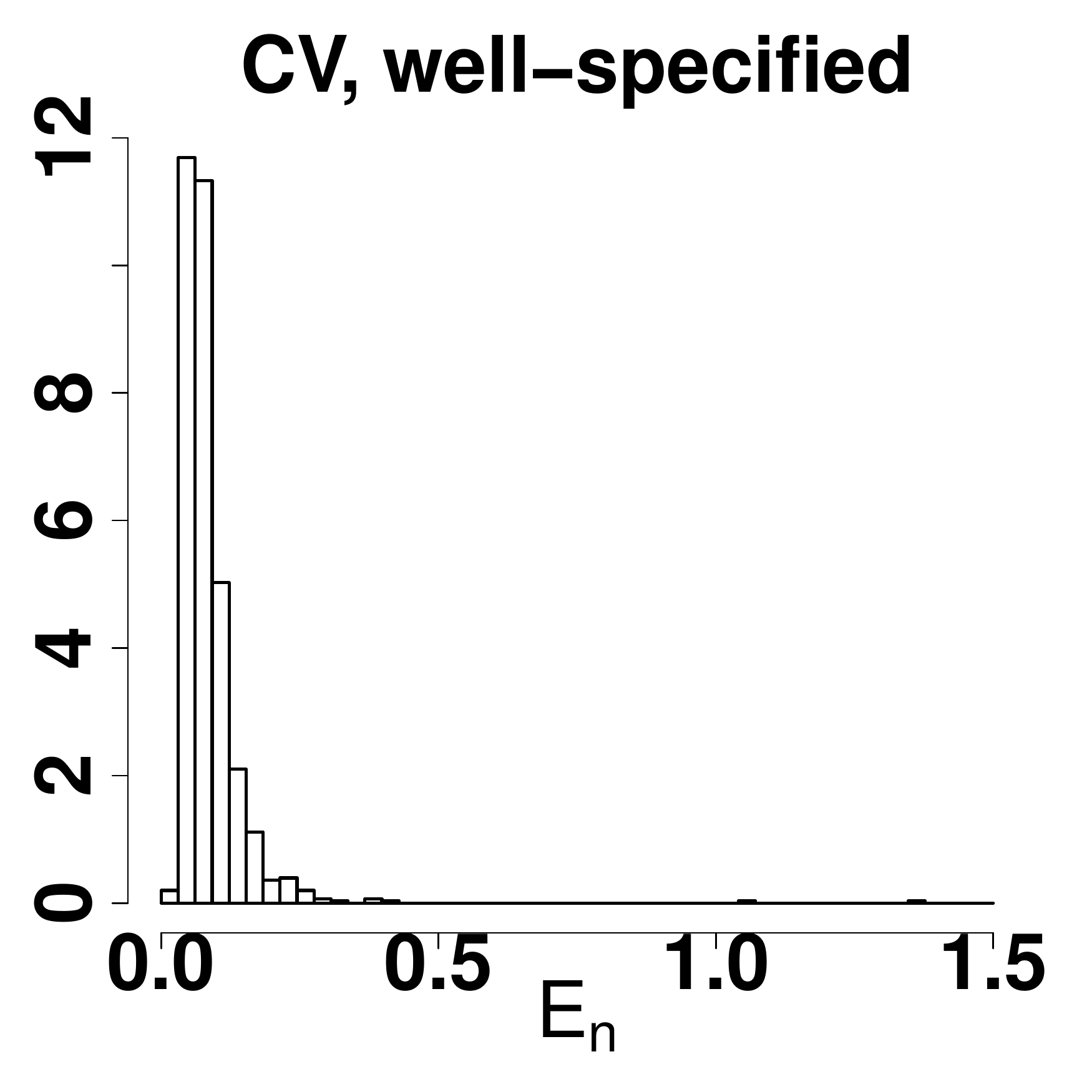}
&
\includegraphics[height=4cm,width=4.5cm]{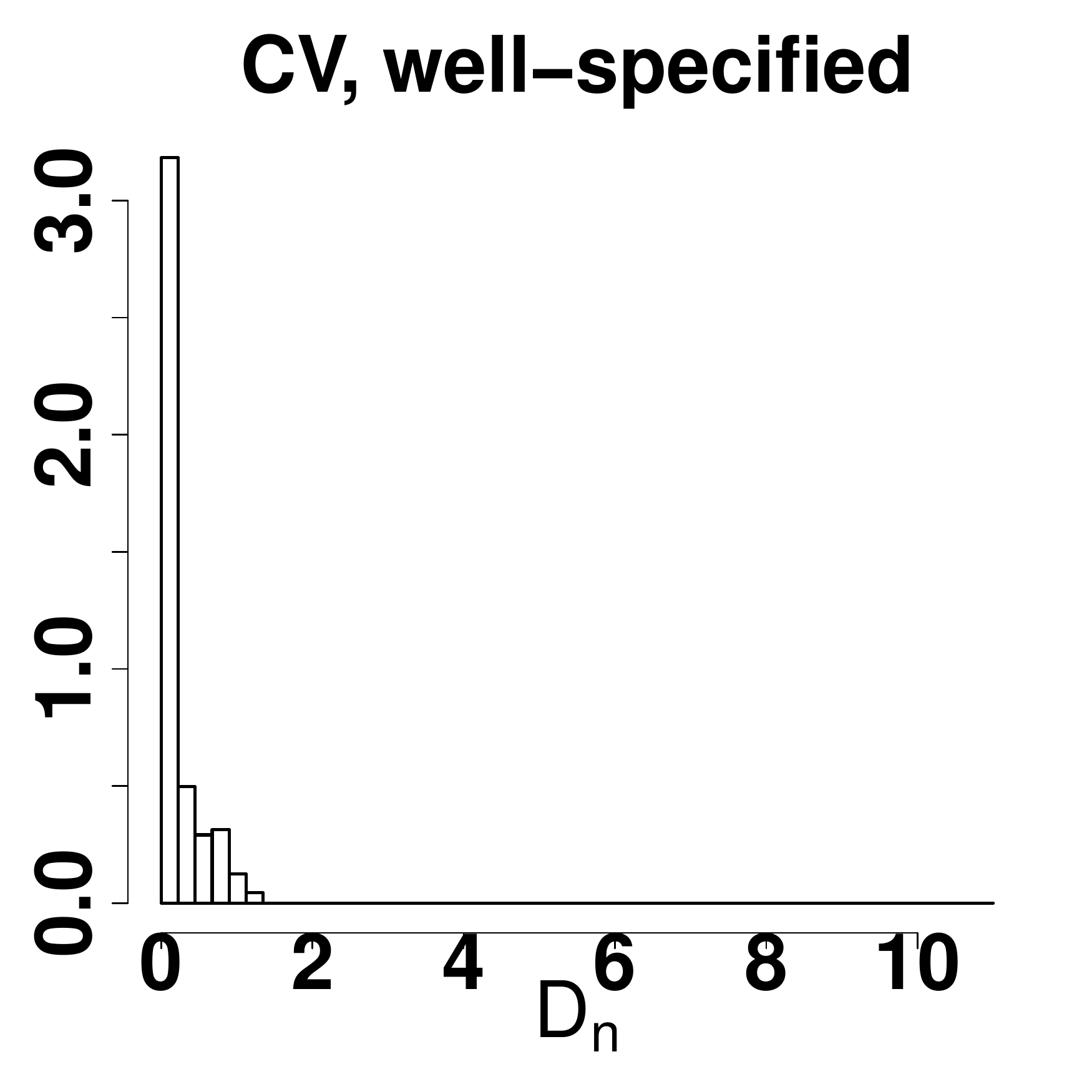}
\\
\includegraphics[height=4cm,width=4.5cm]{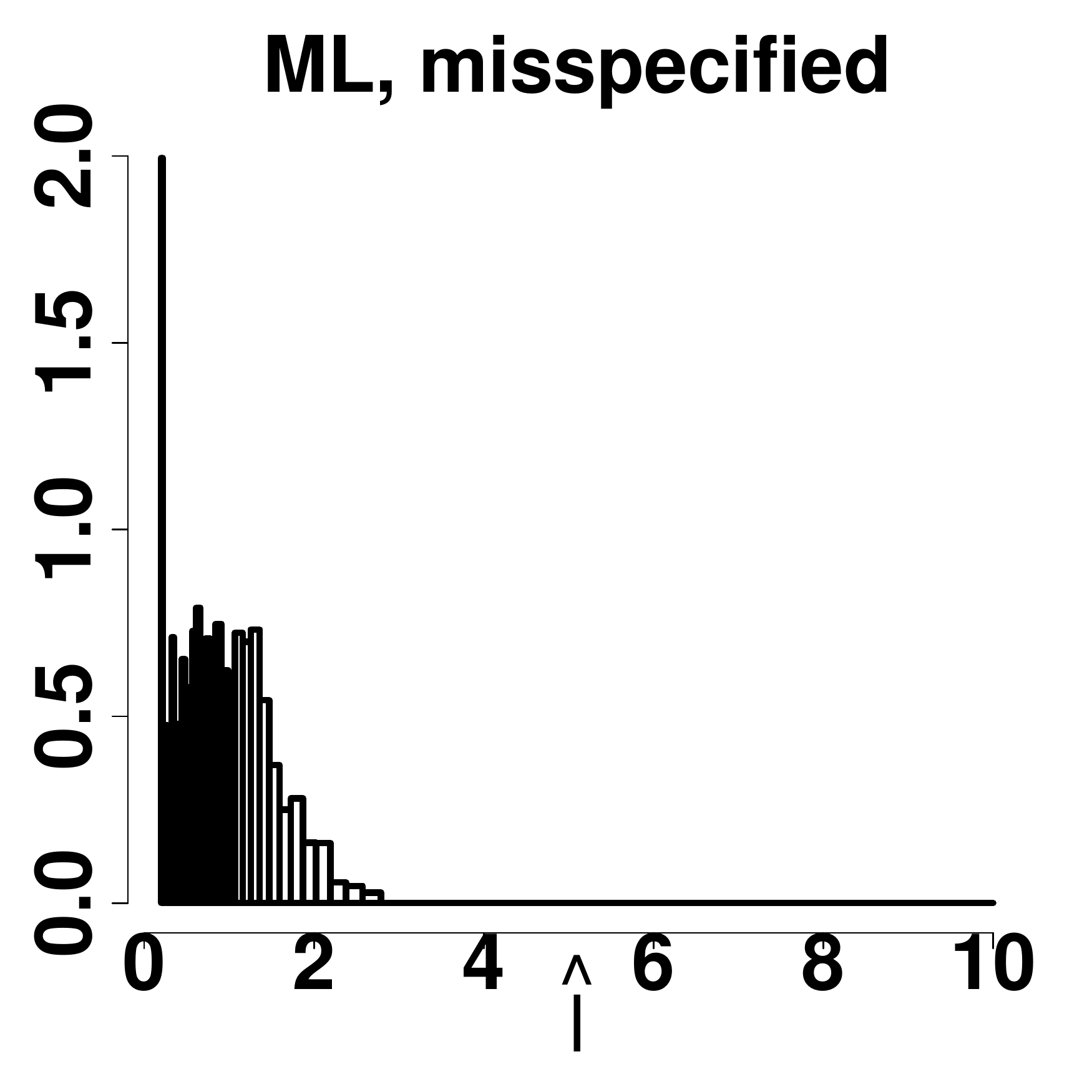}
&
\includegraphics[height=4cm,width=4.5cm]{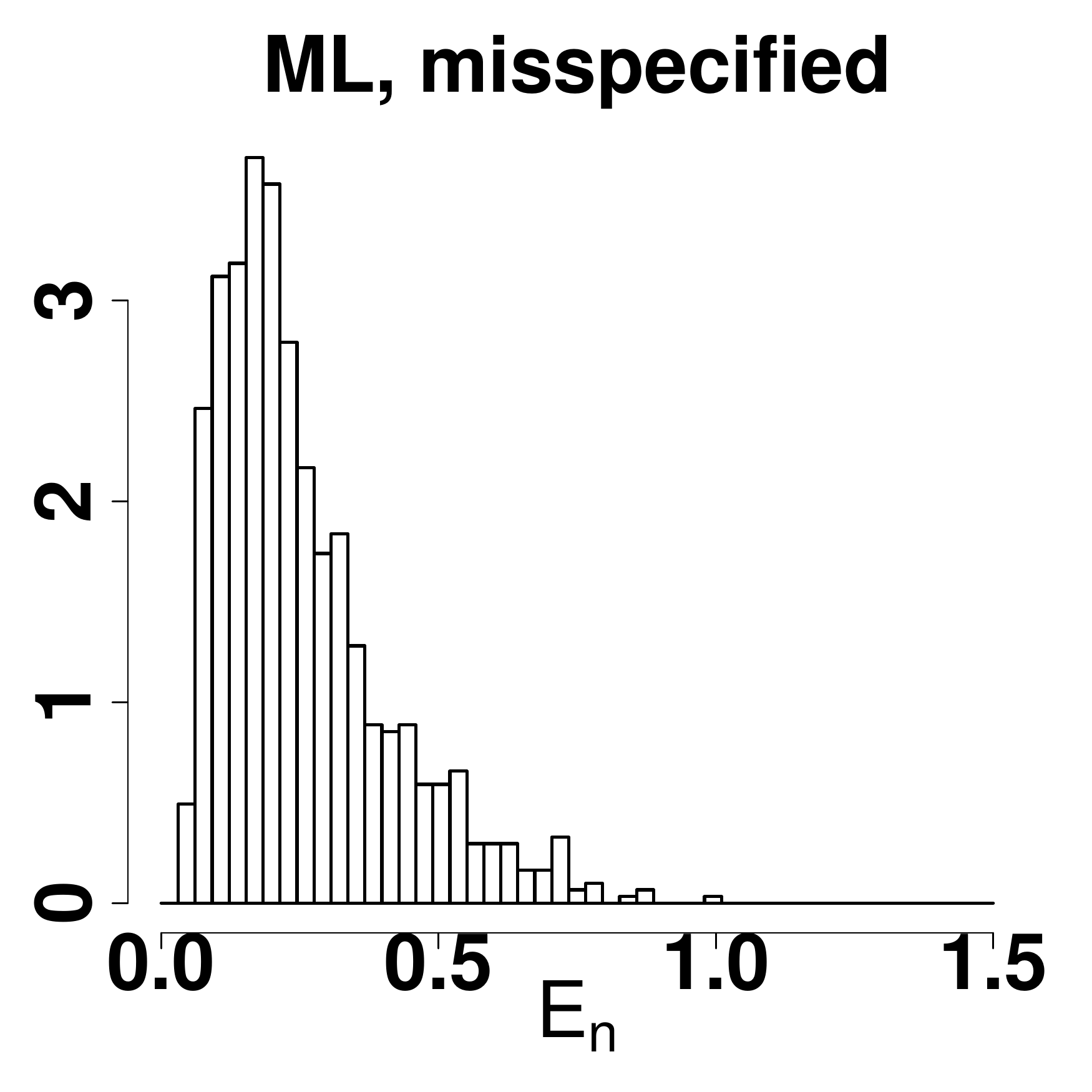}
&
\includegraphics[height=4cm,width=4.5cm]{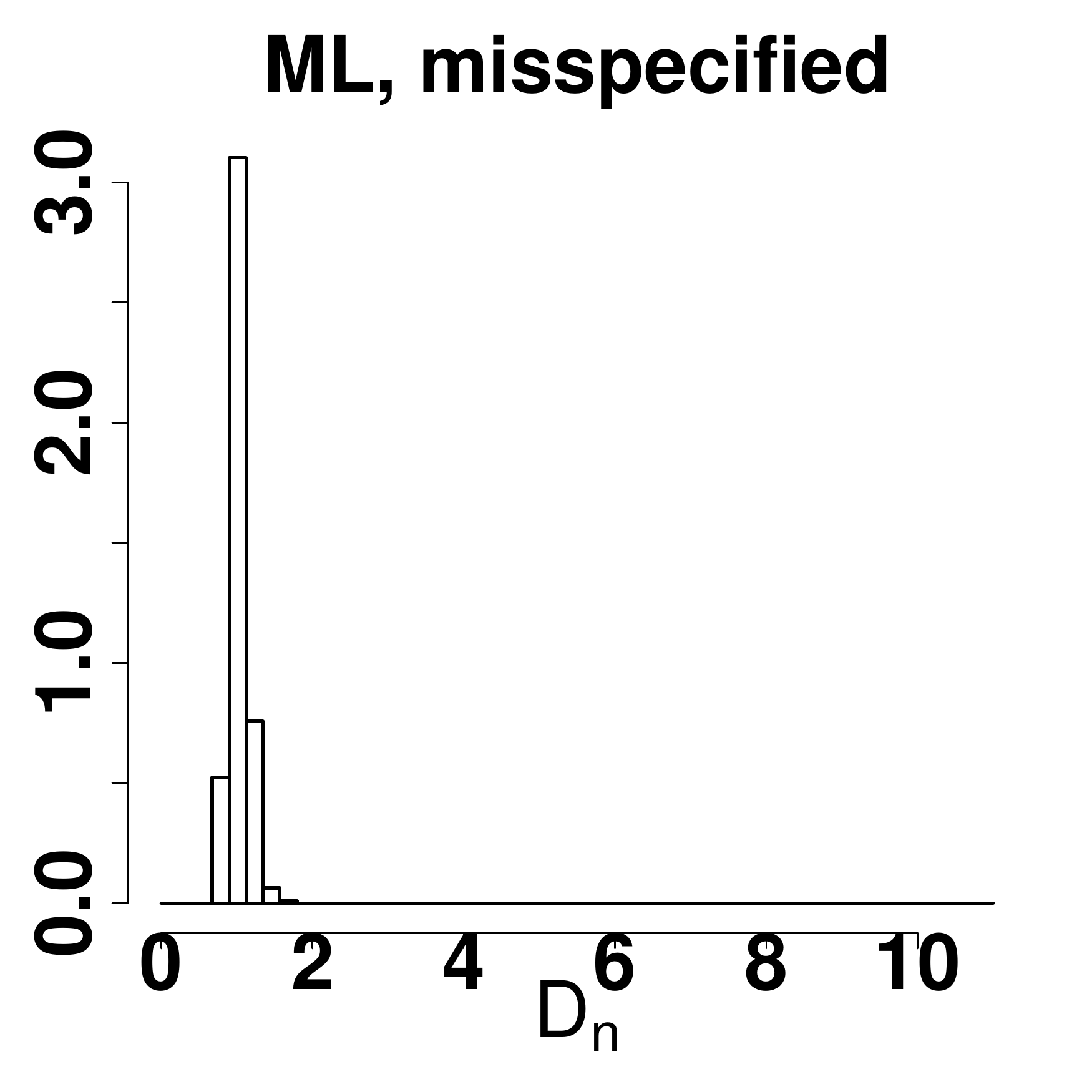}
\\
\includegraphics[height=4cm,width=4.5cm]{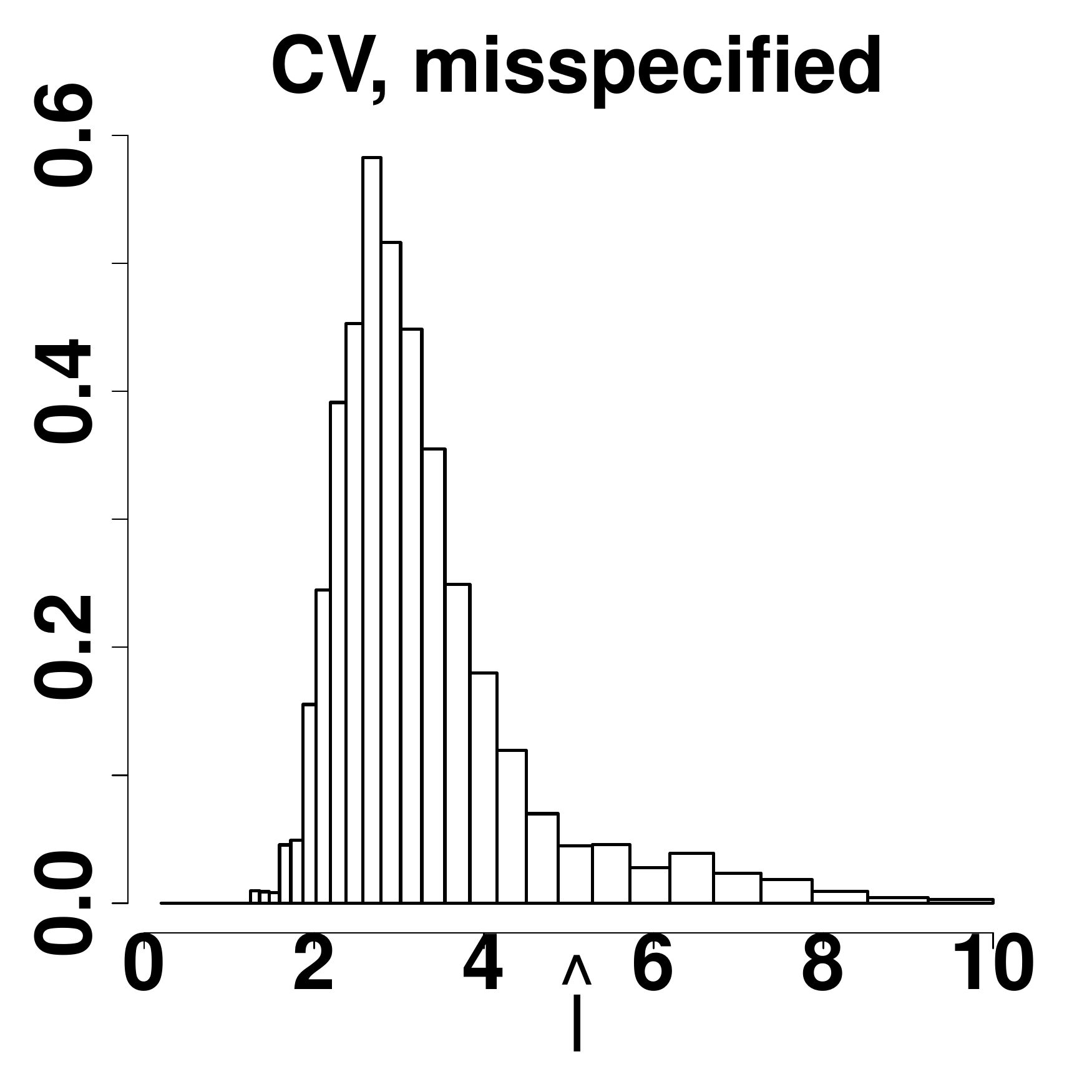}
&
\includegraphics[height=4cm,width=4.5cm]{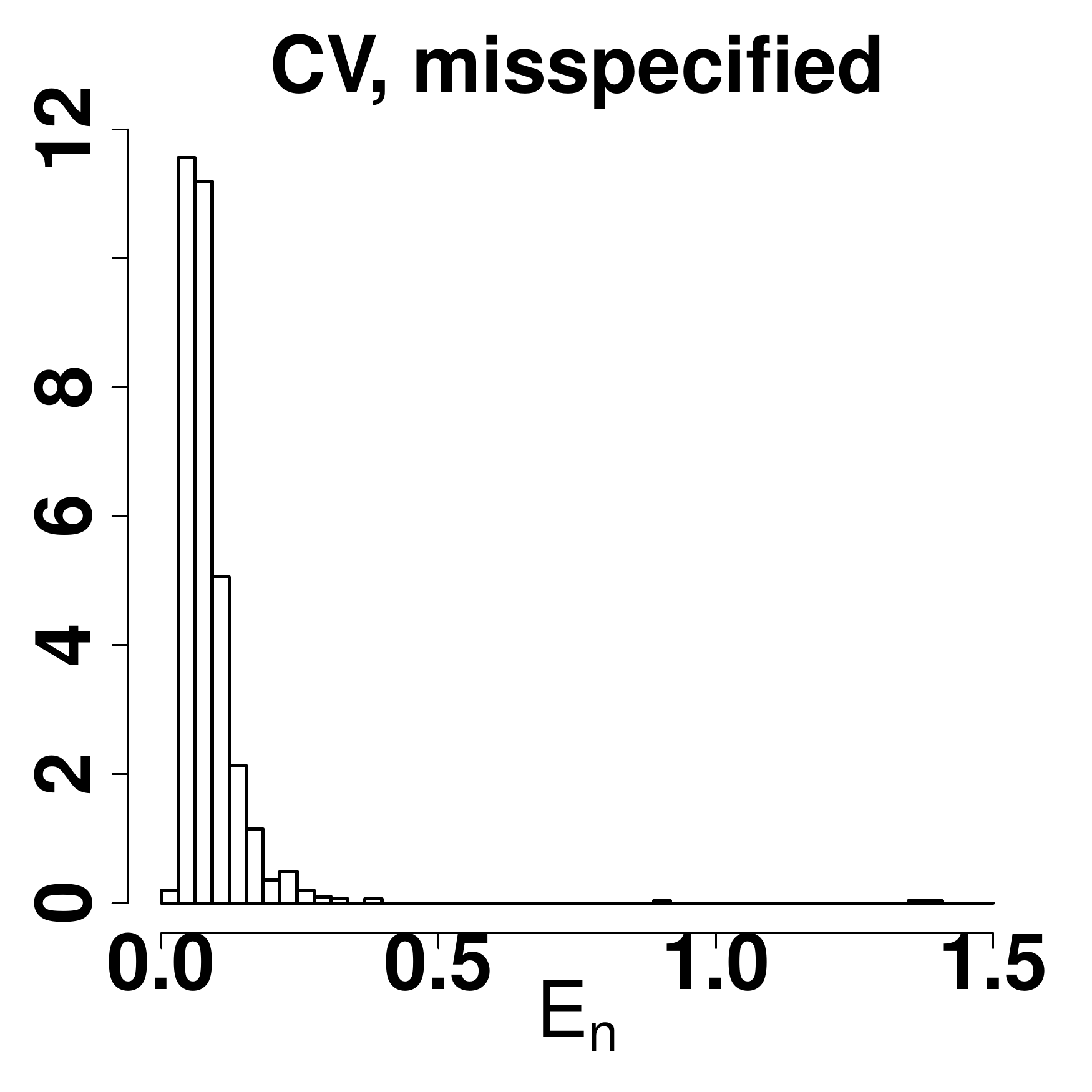}
&
\includegraphics[height=4cm,width=4.5cm]{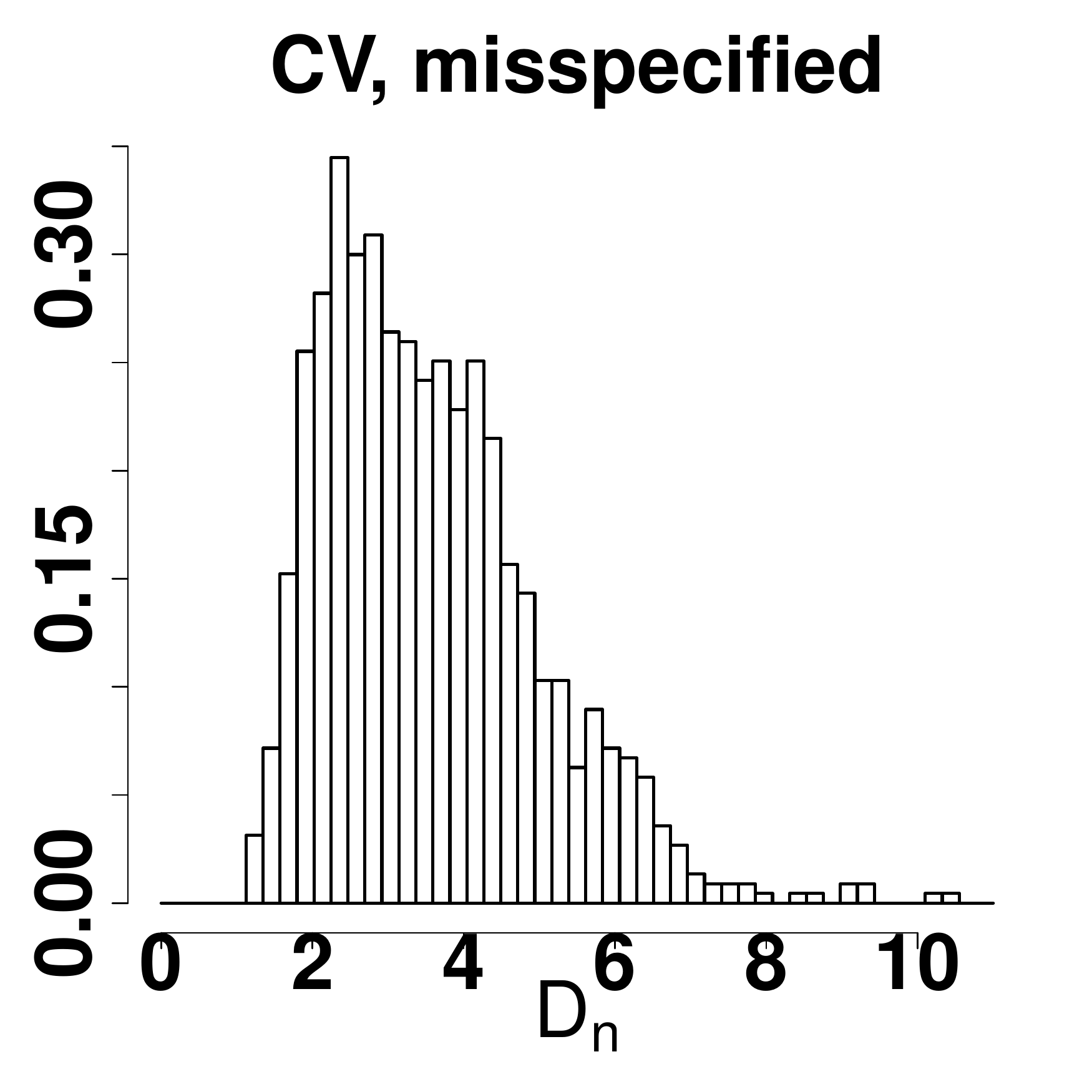}
\end{tabular}
\caption{Simulation of $N=1000$ independent realizations of $n = 100$ i.i.d. observation points with uniform distribution on $[0,n]$, of the Gaussian process $Y$ on $[0,n]$ with Mat\'ern ($\sigma_0^2=1,\ell_0 =3,\nu_0=10$) covariance function, and of the corresponding observation vector with noise variance $\delta_0 = 0.25^2$. For each simulation, $\nu_0$ is known, the noise variance is fixed to $\delta_1 = \delta_0$ (well-specified case) or $\delta_1 = 0.1^2 \neq \delta_0$ (misspecified case), $\sigma^2$ and $\ell$ are estimated by ML and CV and the corresponding error criteria $D_{n,\hat{\sigma}^2,\hat{\ell}}$ (normalized Kullback-Leibler divergence) and $E_{n,\hat{\sigma}^2,\hat{\ell}}$ (integrated square prediction error) are computed. The histograms of the $N$ estimates of $\ell$ and of the $N$ corresponding values of the error criteria are reported for ML and CV and in the well-specified and misspecified cases. In the well-specified case, the estimates are on average reasonably close to the true values, the error criteria are reasonably small and ML performs better than CV in all aspects. In the misspecified case, the ML and CV estimates of the correlation lengths are significantly different, ML performs better than CV for $D_{n,\hat{\sigma}^2,\hat{\ell}}$ and CV performs better than ML for $E_{n,\hat{\sigma}^2,\hat{\ell}}$.}
\label{fig:delta:n100}
\end{figure}

\begin{figure} 
\centering
\begin{tabular}{ccc}
\includegraphics[height=4cm,width=4.5cm]{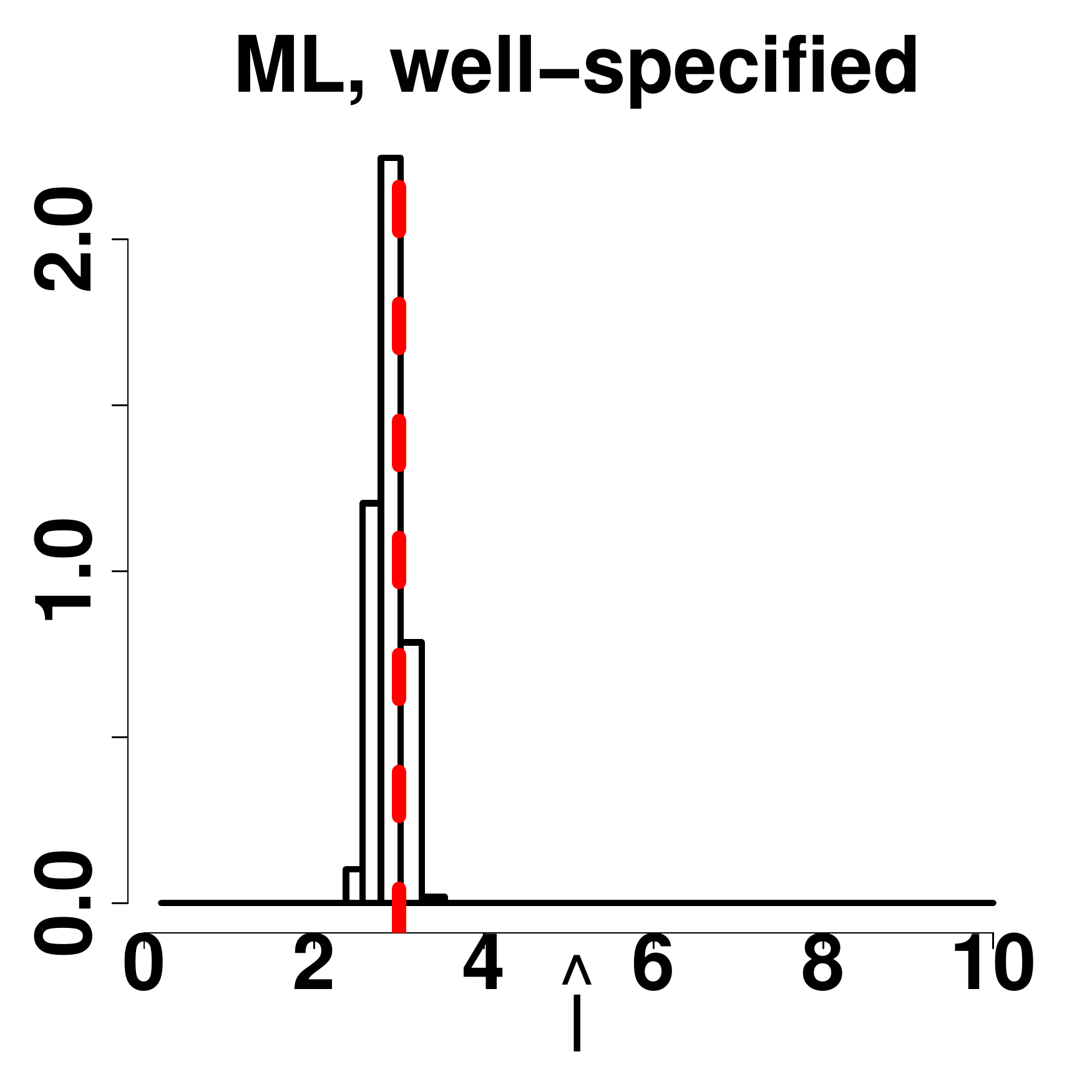}
&
\includegraphics[height=4cm,width=4.5cm]{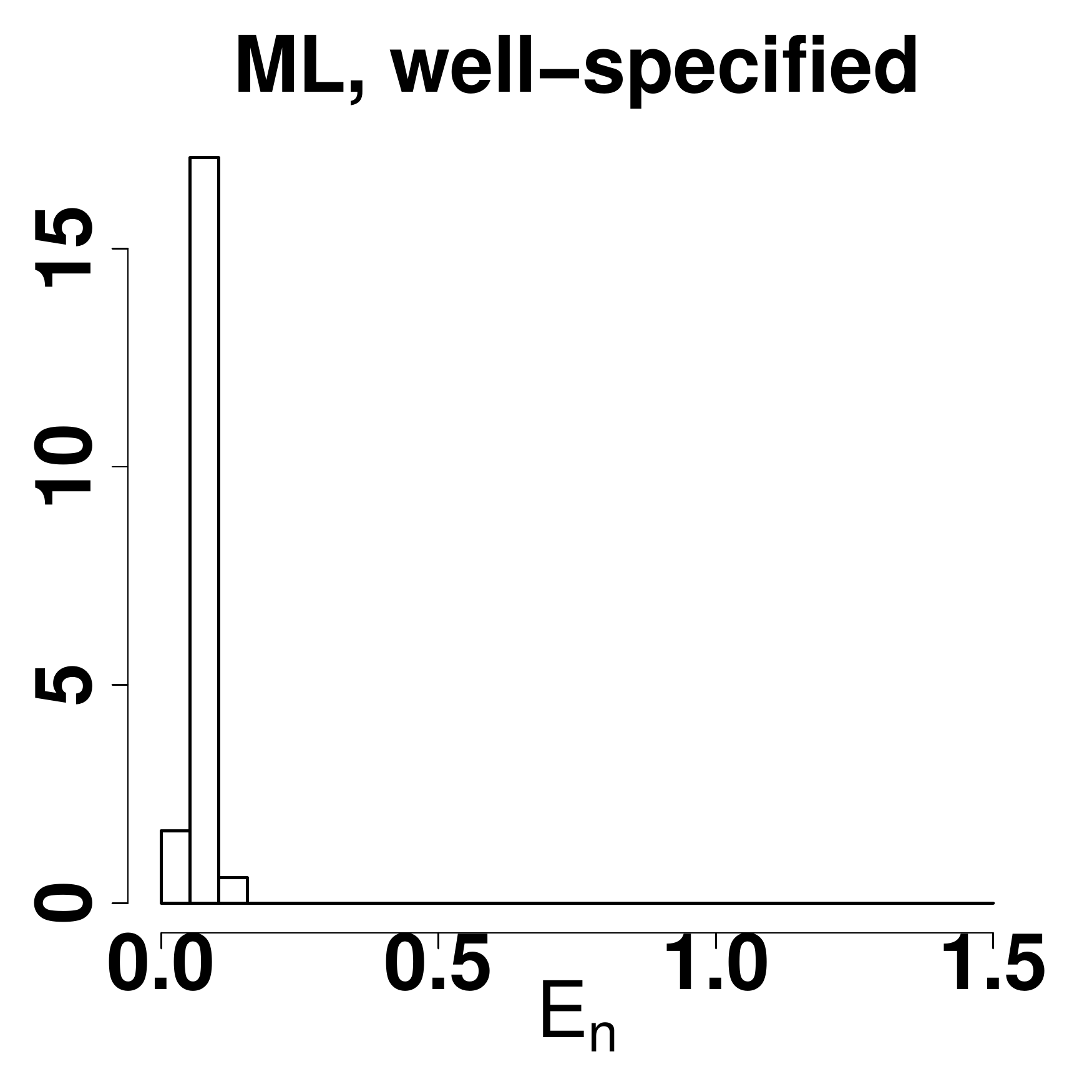}
&
\includegraphics[height=4cm,width=4.5cm]{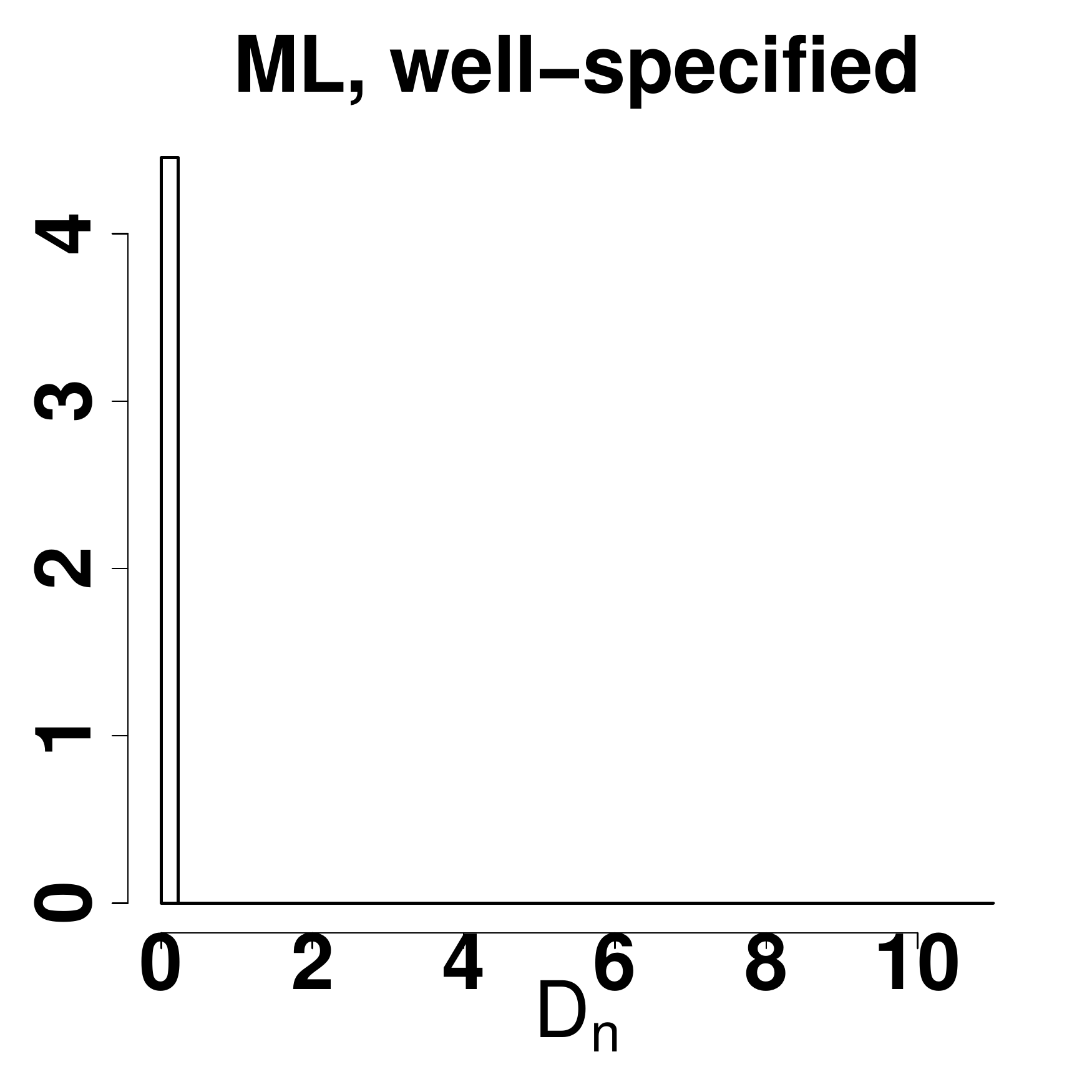}
\\
\includegraphics[height=4cm,width=4.5cm]{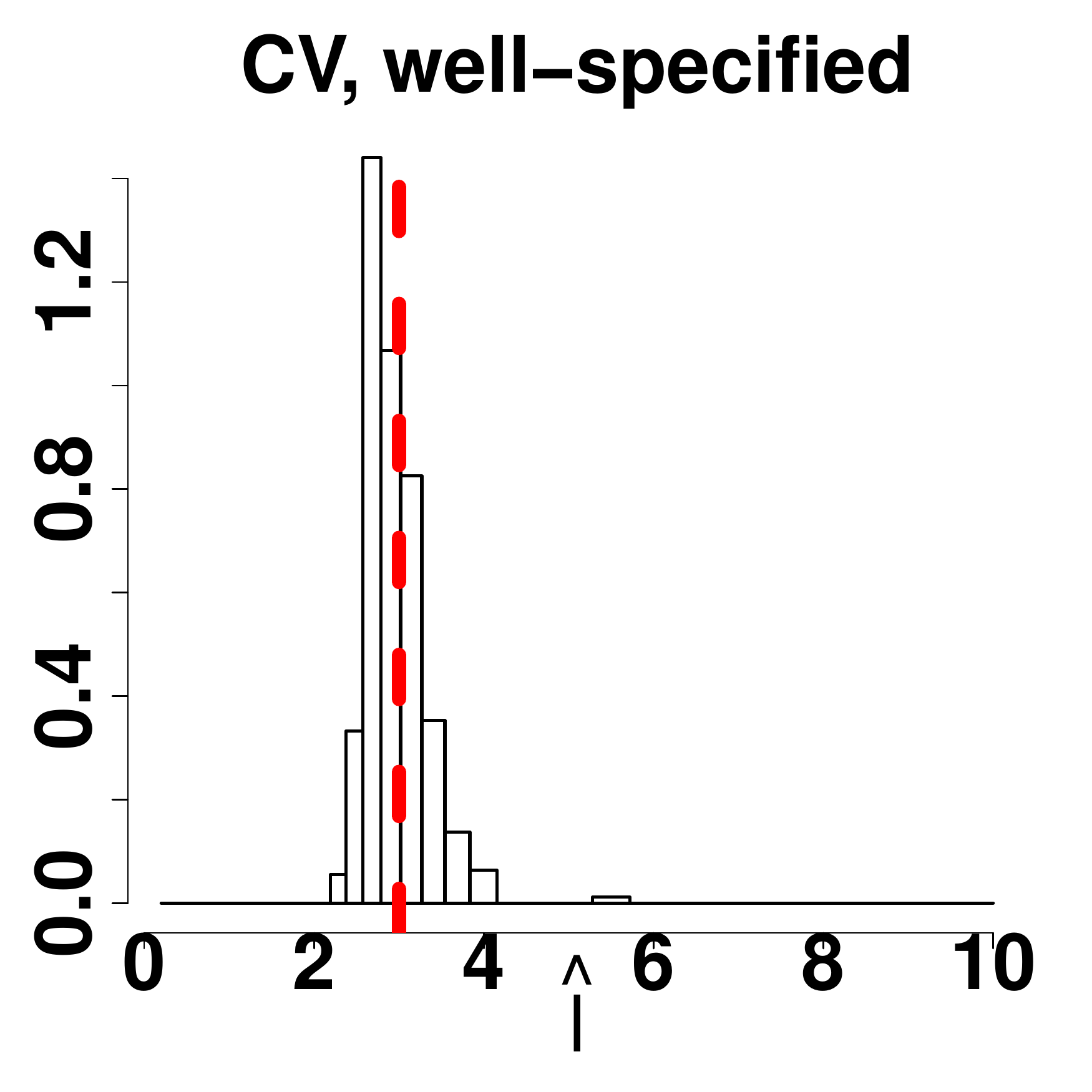}
&
\includegraphics[height=4cm,width=4.5cm]{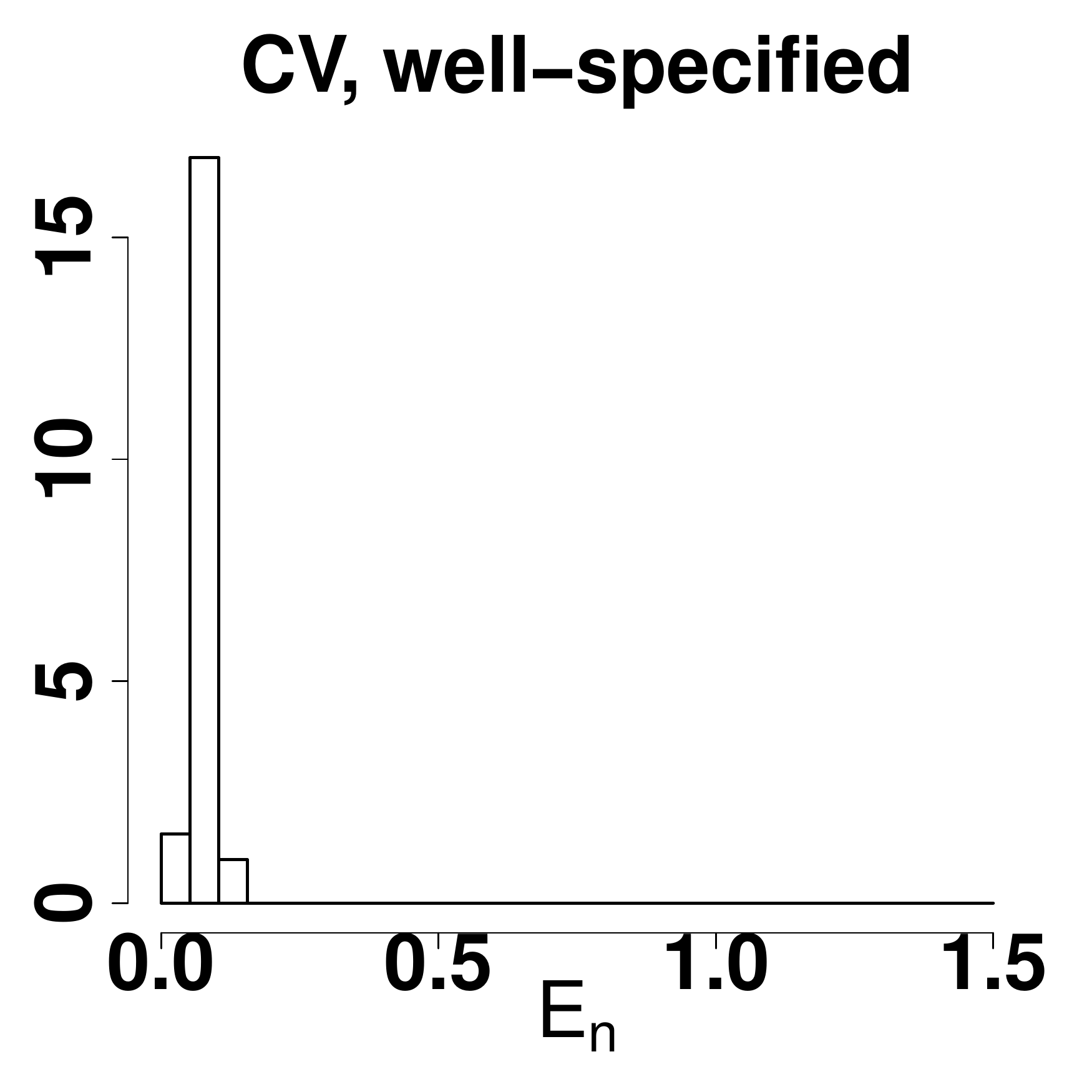}
&
\includegraphics[height=4cm,width=4.5cm]{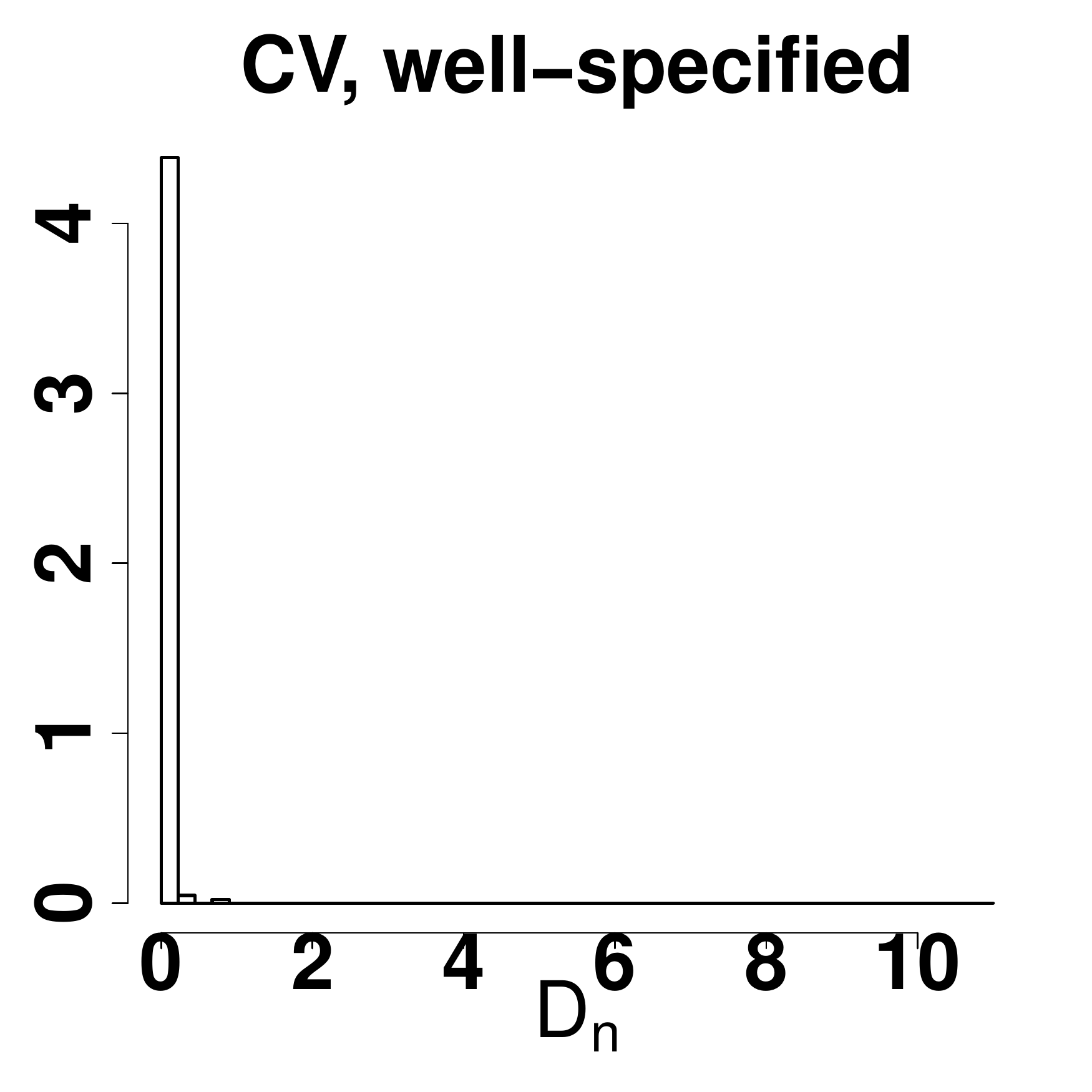}
\\
\includegraphics[height=4cm,width=4.5cm]{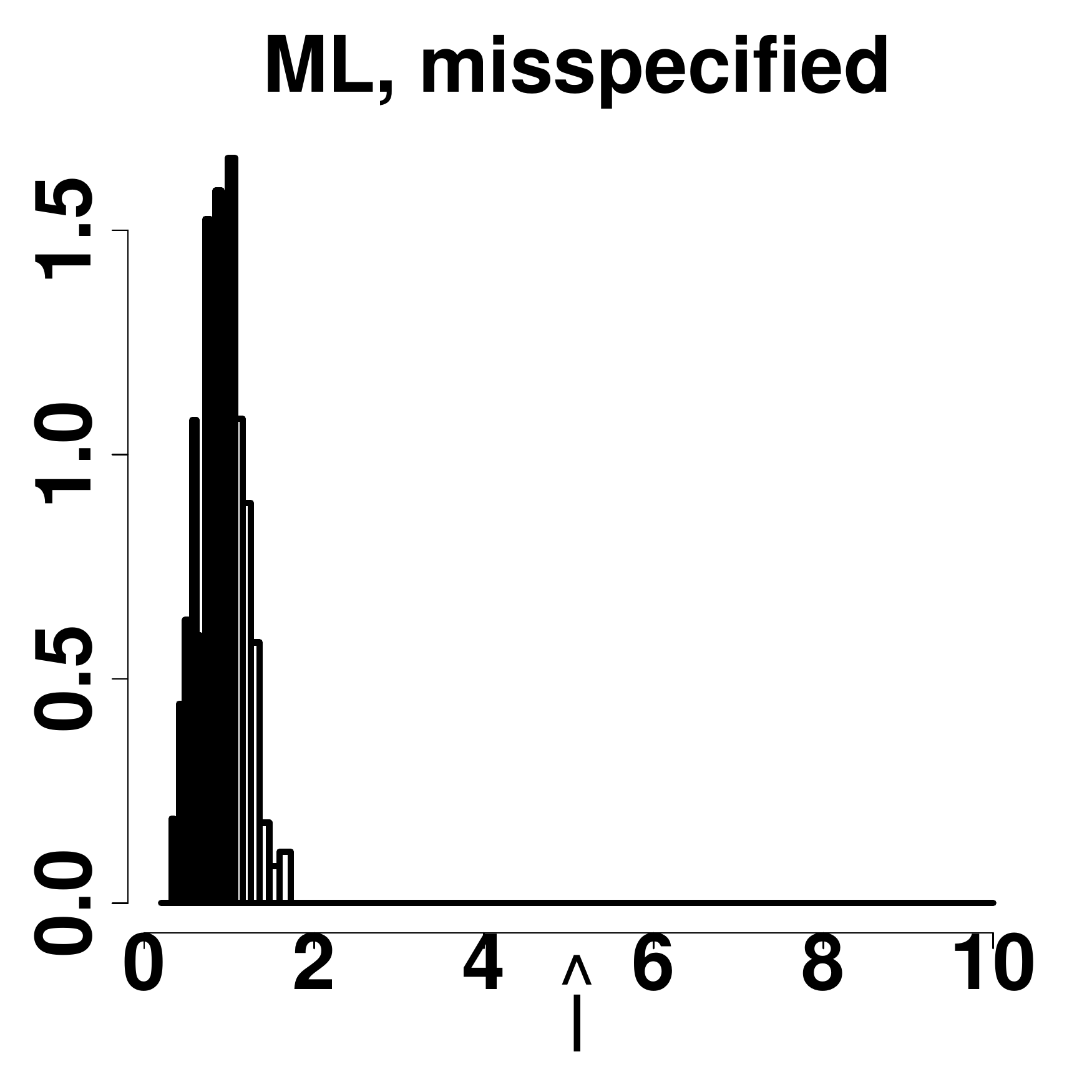}
&
\includegraphics[height=4cm,width=4.5cm]{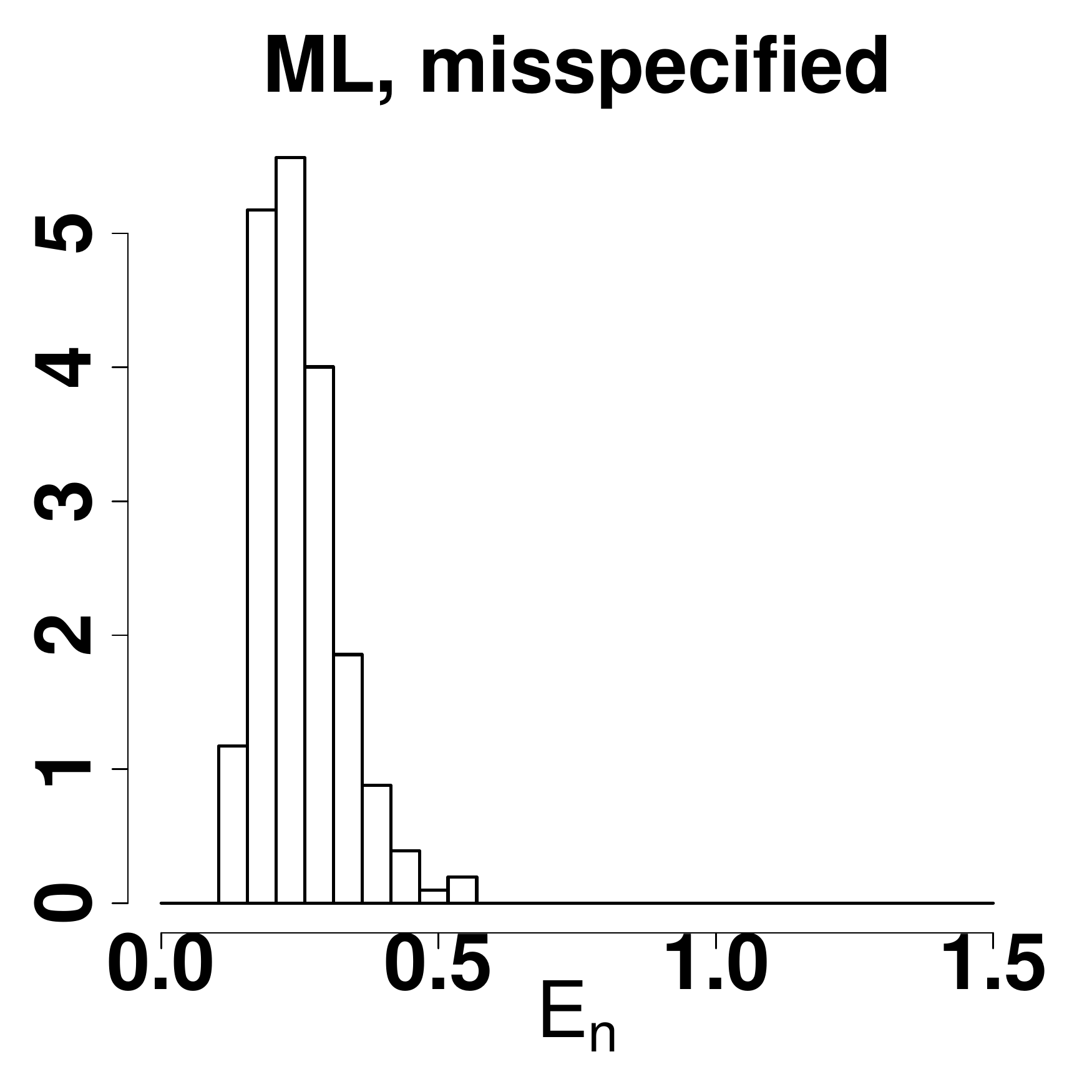}
&
\includegraphics[height=4cm,width=4.5cm]{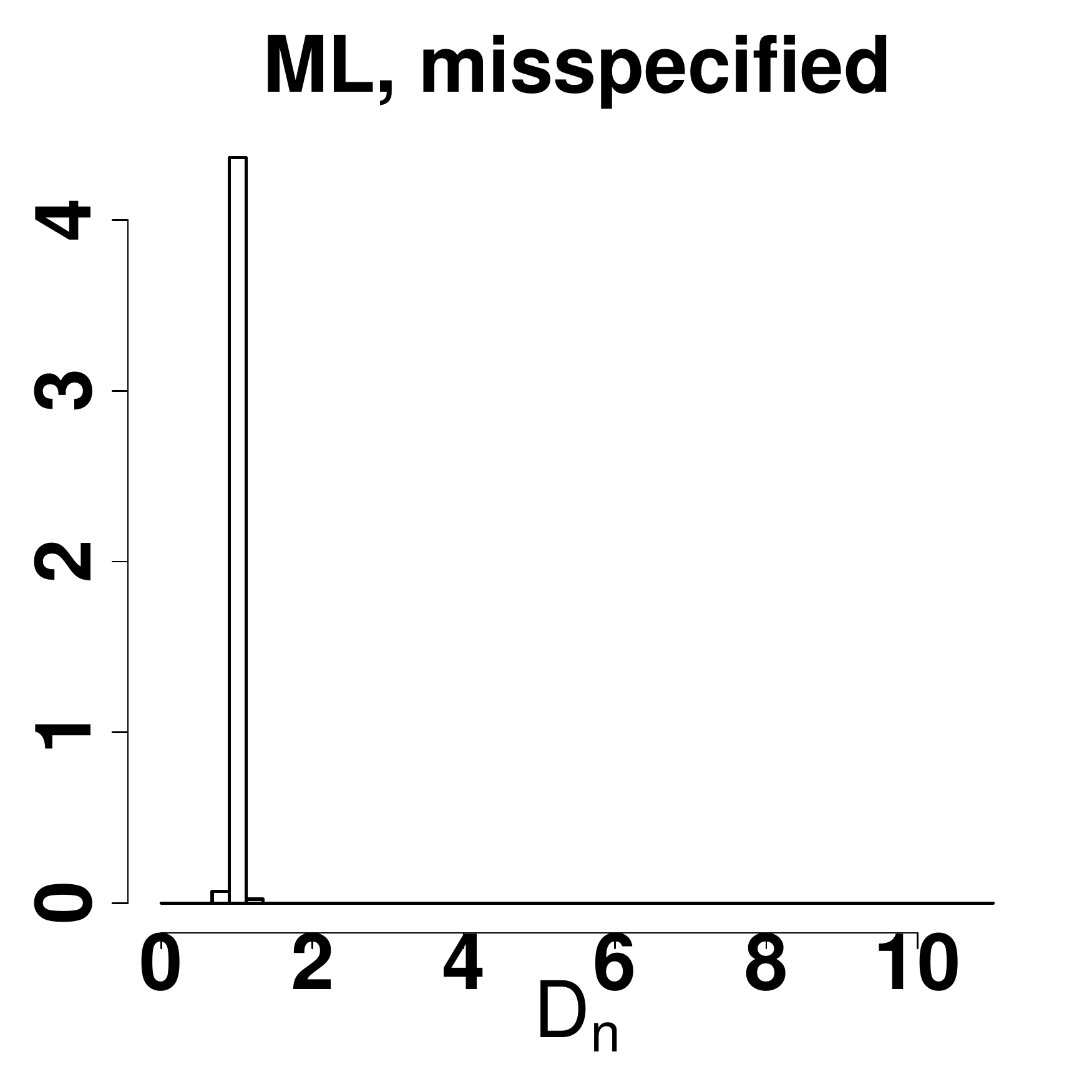}
\\
\includegraphics[height=4cm,width=4.5cm]{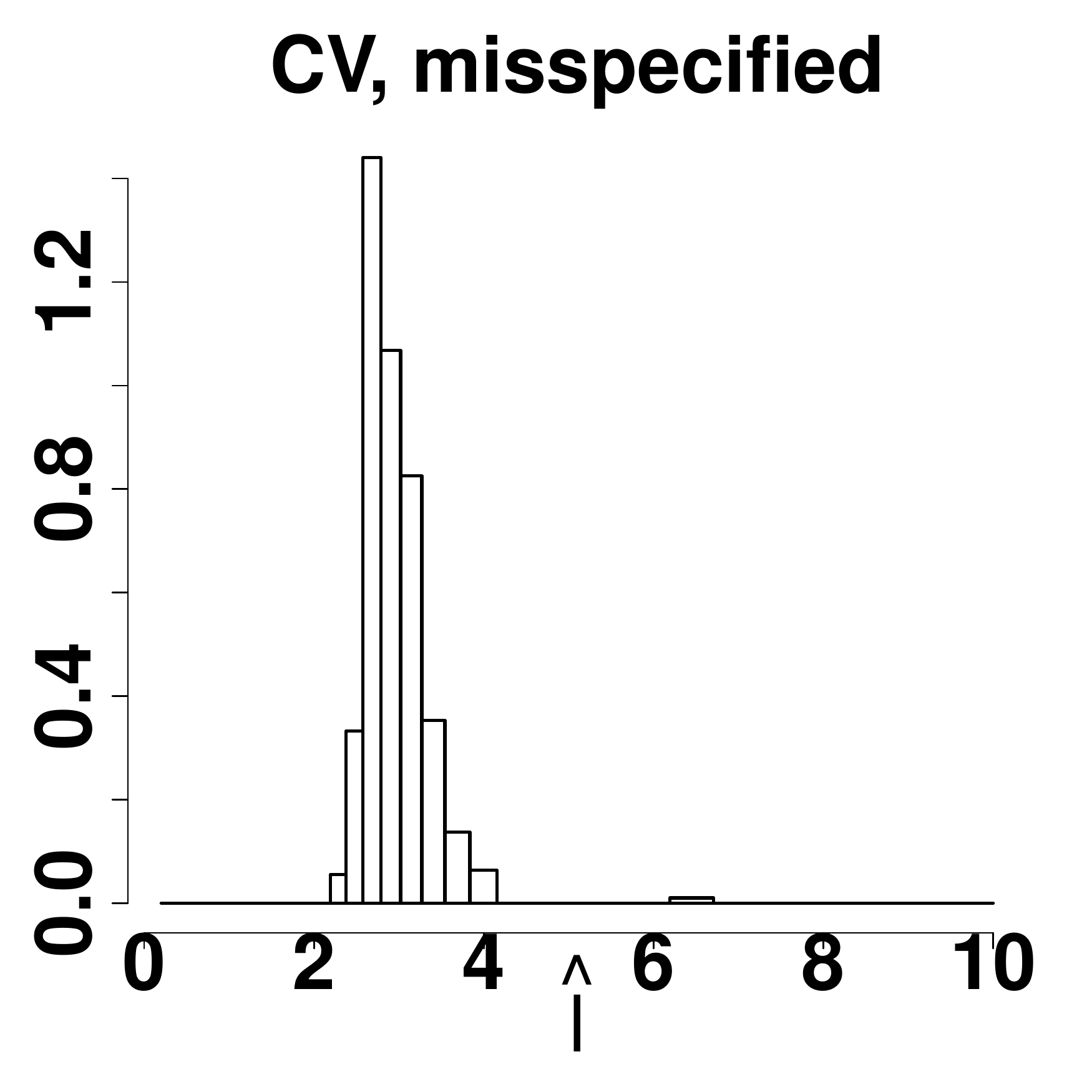}
&
\includegraphics[height=4cm,width=4.5cm]{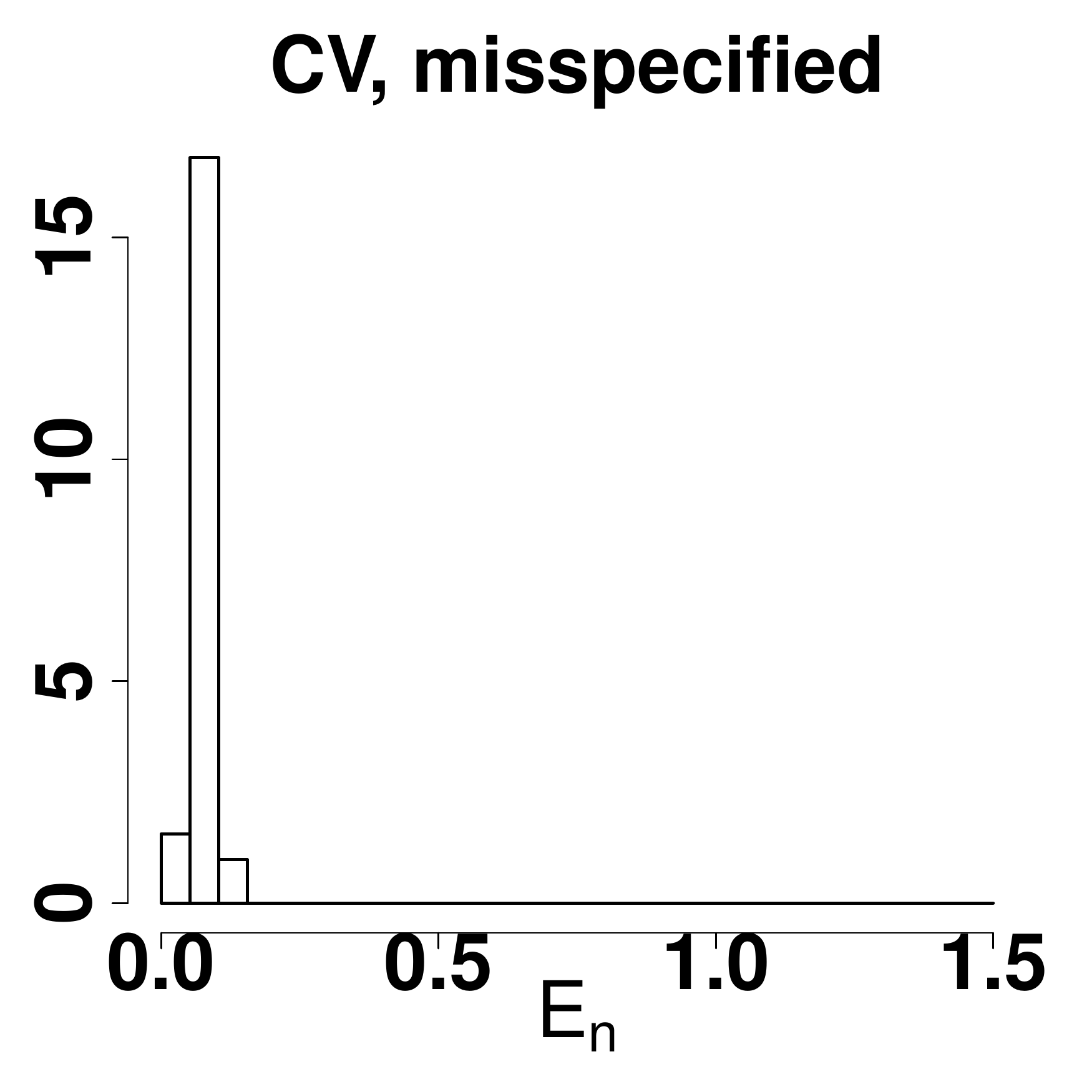}
&
\includegraphics[height=4cm,width=4.5cm]{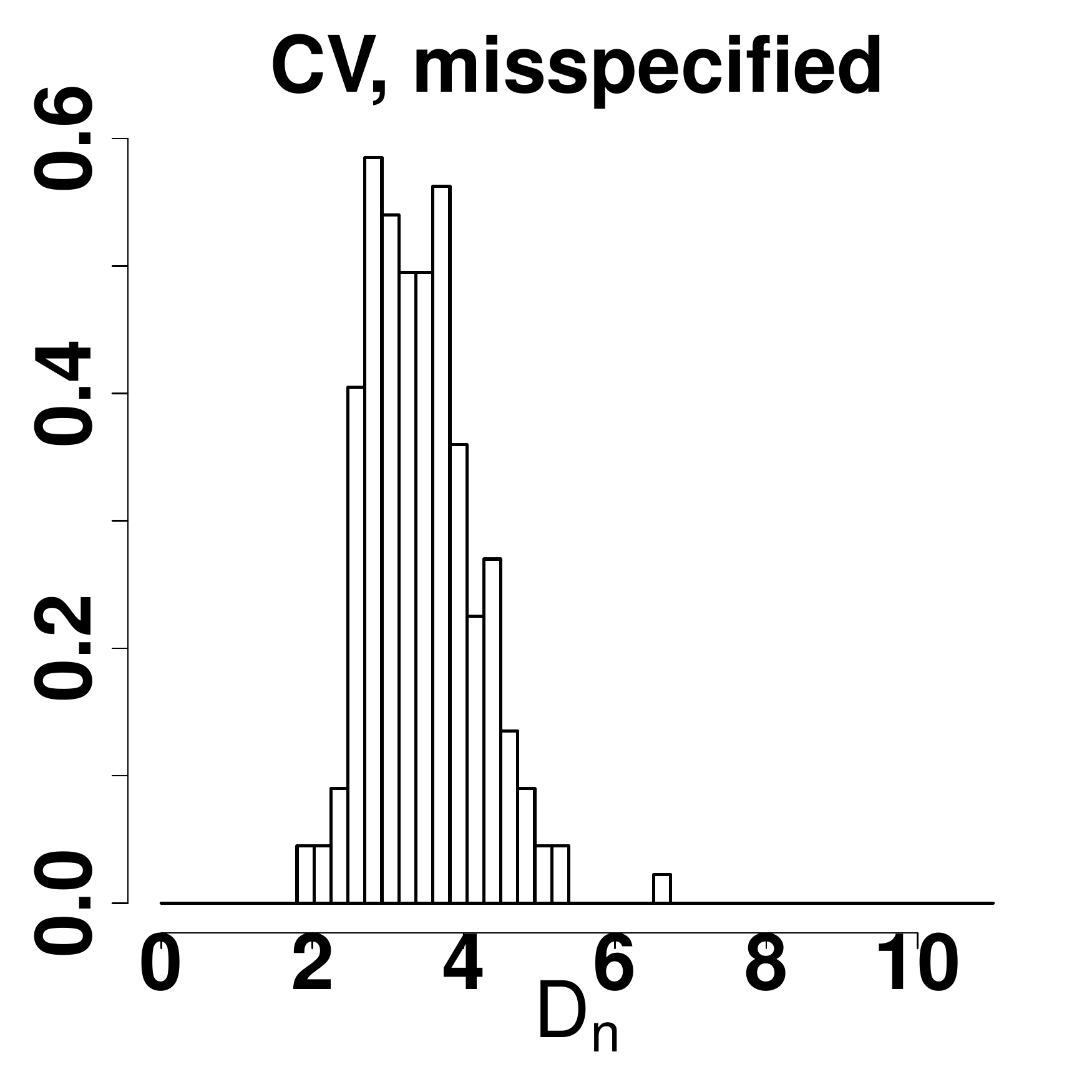}
\end{tabular}
\caption{Same settings as for Figure \ref{fig:delta:n100} but with $N=200$ and $n=500$. The relative differences between ML and CV are the same as for Figure \ref{fig:delta:n100}. The variances of the estimates $\hat{\ell}$ and of the error criteria are smaller than in Figure \ref{fig:delta:n100} and the histograms of the estimates are approximately symmetric and unimodal. Between Figure \ref{fig:delta:n100} and Figure \ref{fig:delta:n500}, all error criteria keep the same order of magnitude except $D_{n,\hat{\sigma}^2,\hat{\ell}}$ in the well-specified case which decreases and becomes very small for ML and CV (see also Table \ref{table:MCresults}). This is because in the well-specified case $D_{n,\sigma_0^2,\ell_0} = 0$.}
\label{fig:delta:n500}
\end{figure}

\begin{table}[] 
\begin{center} 
\begin{tabular}{  c | c | c | c | c | c | c}
$n$ & Specification & Estimation & Average & Standard deviation  & Average  & Average \\
&  & & of $\hat{\ell}$ & of $\hat{\ell}$ & of $E_{n,\hat{\sigma}^2,\hat{\ell}}$ & of $D_{n,\hat{\sigma}^2,\hat{\ell}}$ \\
\hline
\multirow{4}{*}{$100$} & Well-specified & ML & 3.031 & 0.370 & 0.073 & 0.023  \\
 & Well-specified & CV & 3.447 & 1.159 & 0.085 & 0.222  \\
 & Misspecified & ML & 1.090 & 0.553 & 0.255 & 1.025  \\
 & Misspecified & CV & 3.525 & 1.324 & 0.087 & 3.563  \\
\hline
\multirow{4}{*}{$500$} & Well-specified & ML & 3.003 & 0.177 & 0.070 & 0.005 \\
 & Well-specified & CV & 3.084 & 0.399 & 0.071 & 0.037  \\
 & Misspecified & ML &0.975 & 0.269 & 0.247 & 0.972  \\
 & Misspecified & CV & 3.089 & 0.437 & 0.071 & 3.442  \\
\end{tabular}
\end{center} 
\caption{For the settings of Figures \ref{fig:delta:n100} and \ref{fig:delta:n500}, the averages and standard-deviations of $\hat{\ell}$, $D_{n,\hat{\sigma}^2,\hat{\ell}}$
and $E_{n,\hat{\sigma}^2,\hat{\ell}}$ are reported.}
\label{table:MCresults}
\end{table}

\section{Discussion} \label{section:conclusion}

Theorems \ref{theorem:oracle:ML} and \ref{theorem:oracle:CV}, together with the results of the simulation study and the existing literature, draw the following conclusion. 

In the well-specified case, any covariance parameter estimator can be evaluated relatively to the estimation error criterion. The ML estimator is thus generally optimal in the well-specified case. In the simulation study, ML performs better than CV for the estimation error, the conditional Kullback-Leibler divergence and the integrated square prediction error.

On the other hand, in the misspecified case, there is not a unique quality criterion for covariance parameter estimation. Different criteria are optimized by different covariance parameters and estimators. We prove that ML asymptotically minimizes the conditional Kullback-Leibler divergence. In the case of independent and uniformly distributed observation points, we prove that CV asymptotically minimizes the integrated square prediction error.
Thus, CV is asymptotically optimal for the criterion it is designed for, provided the spatial sampling is in agreement with the CV principle.
As shown in the simulation study, the estimated covariance parameters and quality criterion values differ radically between ML and CV in the misspecified case.
In this regard, we point out that ML is not optimal relatively to all the common quality criteria, contrarily to the well-specified case. 
Note finally that our aim is not to provide a hierarchy between ML and CV in the misspecified case.

The fact that ML and CV typically optimize different criteria in the misspecified case can serve as a practical guideline. That is, one can compute the estimated covariance parameters with both methods and compare the two estimates and the corresponding log-likelihood and LOO mean square error values. If the differences between ML and CV are large, then it could be a warning that the covariance model at hand can be inappropriate.

We would like to mention two avenues for future research. First, the results of the Monte Carlo simulation make it conceivable that, for independent and uniform observation points, the ML and CV estimators converge to optimal parameters, for respectively the Kullback-Leibler divergence and the integrated square prediction error, and are asymptotically normal. [These optimal parameters would be equal to the true ones in the well-specified case.]
Considering asymptotic normality might require new techniques to account for independent and uniformly distributed observation points. 

Second, consider the alternative CV estimator, maximizing the log predictive probability criterion (\cite{rasmussen06gaussian}, chapter 5, \cite{zhang10kriging}, \cite{sundararajan01predictive}). It would be interesting to see whether this estimator can be shown to minimize with respect to $\theta$ the quality criterion $\int_{[0,n^{1/d}]^d} d_{\theta}(t)dt$, where $d_{\theta}(t)$ is the conditional Kullback-Leibler divergence of the conditional distribution of $Y(t)$, given $y$ and $X$, assumed under $(K_{\theta},\delta_{\theta})$, from the corresponding true conditional distribution obtained from $(K_{0},\delta_{0})$.

\appendix

\section{Proofs} \label{section:proofs}

\subsection{Notation}

In all the appendix, we consider that Conditions \ref{cond:iid_obs_points}, \ref{cond:Kzero:deltazero} and \ref{cond:Ktheta:deltatheta} hold.
For a column vector $v$ of size $m$, we let $||v||^2 = \sum_{i=1}^m v_i^2$ and $|v| = \max_{i=1,...,m} |v_i|$.
For a real $m \times m$ matrix $A$, we write as in \cite{gray06toeplitz}, $|A|^2 = \frac{1}{m} \sum_{i,j=1}^m A_{i,j}^2$ and $||A||$ for the largest singular value of $A$.
Both $|.|$ and $||.||$ are norms and $||.||$ is also a matrix norm.

For a sequence of real random variables $z_n$, we write $z_n \to_p 0$ and $z_n = o_p\left(1\right)$
when $z_n$ converges to zero in probability.
For a random variable $A$ and a deterministic function $f(A)$, we may write $\EE_A(f(A))$ for $\EE(f(A))$.
For two random variables $A$ and $B$ and a deterministic function $f(A,B)$ we may write $\EE_{A|B}(f(A,B))$ for $\EE( f(A,B) | B )$.

For a finite set $E$, we write $|E|$ for its cardinality. For a continuous set $E \subset \RR^d$, we write $|E|$ for its Lebesgue measure. For two sets $A$, $B$ in $\RR^d$, we write $d(A,B) = \inf_{a \in A,b\in B} |a-b|$.

We write $C_{sup}$ a generic non-negative finite constant (not depending on $n$, $X$, $Y$, $\epsilon$ and $\theta$). The actual value of $C_{sup}$ is of no interest and can change in the same sequence of equations. For instance, instead of writing, say, $a \leq 2b \leq 4c$, we shall write $a \leq C_{sup} b \leq C_{sup} c$.
Similarly, we write $C_{inf}$ a generic strictly positive constant (not depending on $n$, $X$, $Y$, $\epsilon$ and $\theta$). 

\subsection{Proofs for Maximum Likelihood} \label{subsection:proof:ML}

\begin{lem} \label{lem:control:derivative:ML}
For $i=1,...,p$,
\[
\EE \left( \sup_{\theta \in \Theta} \left| \frac{\partial}{\partial \theta_i} L_{\theta} \right| \right)
\]
is bounded w.r.t $n$.
\end{lem}

\begin{proof}[Proof of Lemma \ref{lem:control:derivative:ML}]

\begin{eqnarray} 
\EE \left( \sup_{\theta \in \Theta} \left| \frac{\partial}{\partial \theta_i} L_{\theta} \right| \right)
& = & \EE \left( \sup_{\theta \in \Theta} \left| \frac{1}{n} \Tr\left( R_{\theta}^{-1} \frac{\partial}{\partial \theta_i} R_{\theta} \right) - \frac{1}{n} y^t R_{\theta}^{-1} \left( \frac{\partial}{\partial \theta_i} R_{\theta} \right) R_{\theta}^{-1} y \right| \right) \nonumber \\
\mbox{(Cauchy-Schwarz:)} & \leq & \EE \left( \sup_{\theta \in \Theta}   \sqrt{ \left| R_{\theta}^{-1} \right|^2 } \sqrt{
\left| \frac{\partial}{\partial \theta_i} R_{\theta} \right|^2 }  \right) \label{eq:in:proof:controle:derivative:ML:1} \\ 
& & + 
\sqrt{ \frac{1}{n} \EE\left( \sup_{\theta \in \Theta} \left| \left| R_{\theta}^{-1} y \right| \right|^2 \right) }  
\sqrt{ \frac{1}{n} \EE\left( \sup_{\theta \in \Theta} \left| \left| \left( \frac{\partial}{\partial \theta_i} R_{\theta} \right) R_{\theta}^{-1} y \right| \right|^2 \right) }  \label{eq:in:proof:controle:derivative:ML:2}
\end{eqnarray}

Now, $\left| R_{\theta}^{-1} \right|^2 \leq \left| \left| R_{\theta}^{-1} \right| \right|^2$ because of (2.19) in \cite{gray06toeplitz} and $\left| \left| R_{\theta}^{-1} \right| \right|^2$ is bounded uniformly in $\theta$ because of Lemma \ref{lem:controle:Rmun:diagRmun}. Also, $\EE \left( \sup_{\theta \in \Theta}    \left| [\partial/\partial \theta_i] R_{\theta} \right|^2 \right)$ is bounded because of Condition \ref{cond:Ktheta:deltatheta} and of a simple case of Lemma \ref{lem:norm:Aun...Ak:carre}. So the right-hand side of \eqref{eq:in:proof:controle:derivative:ML:1} is bounded because of Jensen inequality. It remains to show that the term \eqref{eq:in:proof:controle:derivative:ML:2} is bounded. To show this, note first that
\begin{eqnarray*} 
\frac{1}{n} \EE\left( \sup_{\theta \in \Theta} \left| \left| R_{\theta}^{-1} y \right| \right|^2 \right)
& \leq & \frac{1}{n} \EE\left( \sup_{\theta \in \Theta} \left| \left| y \right| \right|^2 \left| \left| R_{\theta}^{-1} \right| \right|^2  \right) \\
\mbox{(Lemma \ref{lem:controle:Rmun:diagRmun}:)} & & \frac{C_{sup}}{n} \EE \left( || y ||^2 \right) \\
& = & C_{sup} (K_0(0) + \delta_0),
\end{eqnarray*}
is bounded. Thus, it remains to show that 
$\frac{1}{n} \EE\left( \sup_{\theta \in \Theta} \left| \left| \left( [\partial/\partial \theta_i] R_{\theta} \right) R_{\theta}^{-1} y \right| \right|^2 \right)$ is bounded. For this, we have
\begin{flalign*}
& \frac{1}{n} \EE\left( \sup_{\theta \in \Theta} \left| \left| \left( \frac{\partial}{\partial \theta_i} R_{\theta} \right) R_{\theta}^{-1} y \right| \right|^2 \right) & \\
& =  \frac{1}{n} \EE\left( \sup_{\theta \in \Theta}  y^t R_{\theta}^{-1} \left( \frac{\partial}{\partial \theta_i} R_{\theta} \right)^2 R_{\theta}^{-1} y  \right) & \\
 & \leq  C_{sup} \sum_{i_1+...+i_p \leq p} \int_{\Theta} \frac{1}{n} \EE\left( \left|  \frac{\partial^{i_1}}{\partial \theta_1^{i_1}}...\frac{\partial^{i_p}}{\partial \theta_p^{i_p}}  \left[ y^t R_{\theta}^{-1} \left(  \frac{\partial}{\partial \theta_i} R_{\theta} \right)^2 R_{\theta}^{-1} y  \right] \right| \right) d \theta ~ ~ \mbox{(Lemma \ref{lem:sobolev}).} & 
\end{flalign*}
Thus, it suffices to show that, for fixed $i_1,...,i_p \in \NN$ so that $i_1+...+i_p \leq p$,
\[
\frac{1}{n} \sup_{\theta \in \Theta} \EE\left( \left| \frac{\partial^{i_1}}{\partial \theta_1^{i_1}}...\frac{\partial^{i_p}}{\partial \theta_p^{i_p}}  \left[ y^t R_{\theta}^{-1} \left(  \frac{\partial}{\partial \theta_i} R_{\theta} \right)^2 R_{\theta}^{-1} y  \right] \right| \right)
\]
is bounded. The above display is smaller than a fixed sum of terms of the form $(1/n) \sup_{\theta \in \Theta} \EE \left( |y^t M_{\theta} y| \right)$, where the number of terms is independent of $n$ and $M_{\theta}$ is of the form $N_{1,\theta} M_{1,\theta} ... M_{k,\theta} N_{k+1,\theta}$ with $N_{i,\theta} = I_n$ or $N_{i,\theta} = R_{\theta}^{-1}$ and with $M_{i,\theta}$ of the form $[\partial^{c_1}/\partial \theta_1^{c_1}]...[\partial^{c_p}/\partial \theta_p^{c_p}] R_{\theta}$ with $c_1,...,c_p \in \NN$ and $c_1+...+c_p \leq p+1$.
Hence, it is enough to show that any term of the form $\sup_{\theta \in \Theta} (1/n) \EE \left( |y^t M_{\theta} y| \right)$ above is bounded. We have

\begin{flalign*}
& \sup_{\theta \in \Theta} \frac{1}{n} \EE \left( |y^t M_{\theta} y| \right) & \\
& \leq \sup_{\theta \in \Theta} \frac{1}{n} \EE \left( \left| y^t M_{\theta} y - \EE \left( \left. y^t M_{\theta} y \right| X \right) \right| \right) +\sup_{\theta \in \Theta} \frac{1}{n}  \EE \left( \left| \EE \left( \left. y^t M_{\theta} y \right| X \right) \right| \right) & \\
& \leq \sup_{\theta \in \Theta} \sqrt{ \EE \left( \var \left[  \left. \frac{1}{n} y^t M_{\theta} y \right| X  \right] \right) }+\sup_{\theta \in \Theta} \frac{1}{n}  \EE \left( \left| \EE \left( \left. y^t M_{\theta} y \right| X \right) \right| \right) ~ ~ \mbox{(Jensen inequality)} & \\
& = \sup_{\theta \in \Theta} \sqrt{ \EE \left(  \frac{1}{2n^2} \Tr \left[ R_0 \left\{ M_{\theta} + M_{\theta}^t \right\} R_0 \left\{ M_{\theta} + M_{\theta}^t \right\}  \right] \right) }+\sup_{\theta \in \Theta} \frac{1}{n}  \EE \left( \left| \Tr \left[ R_0 M_{\theta} \right] \right| \right) & \\
& \leq  \sup_{\theta \in \Theta} \sqrt{ \frac{1}{2n} \EE \left(   \left| R_0 \left\{ M_{\theta} + M_{\theta}^t \right\}   \right|^2 \right) }+\sup_{\theta \in \Theta} \sqrt{  \EE \left( |  R_0 |^2  \right)   \EE \left( |  M_{\theta} |^2  \right) }  ~ ~ \mbox{(Cauchy-Schwarz).} & \\
& \leq \sup_{\theta \in \Theta} \sqrt{ \frac{1}{n} \EE \left(   \left| R_0  M_{\theta} \right|^2 + \left| R_0  M_{\theta}^t \right|^2  \right) }+\sup_{\theta \in \Theta} \sqrt{  \EE \left( |  R_0 |^2  \right)   \EE \left( |  M_{\theta} |^2  \right) } &
\end{flalign*}
In the display above, the first term goes to $0$ because of Conditions \ref{cond:Kzero:deltazero} and \ref{cond:Ktheta:deltatheta} and Lemmas \ref{lem:norm:Aun...Ak:carre} and \ref{lem:controle:Rmun:diagRmun}.
The second term is bounded because of Lemmas \ref{lem:norm:Aun...Ak:carre}, \ref{lem:controle:Rmun:diagRmun} and \ref{lem:alternate:Rmun}. This completes the proof.

\end{proof}

\begin{cor} \label{cor:controle:derivative:Ltheta}
For any $i=1,...,p$,
\[
\EE \left( \sup_{\theta \in \Theta} \left| \frac{\partial}{\partial \theta_i} \EE \left( \left.  L_{\theta} \right| X \right) \right| \right)
~ ~ ~
\mbox{and}
~ ~ ~
 \sup_{\theta \in \Theta} \left| \frac{\partial}{\partial \theta_i}  D_{n,\theta} \right|
\]
are bounded w.r.t $n$.
\end{cor}

\begin{proof}[Proof of corollary \ref{cor:controle:derivative:Ltheta}]

Note first that $[\partial/\partial \theta_i] \EE \left( \left.  L_{\theta} \right| X \right) = [\partial/\partial \theta_i]  D_{n,\theta}$.
The corollary is then a consequence of the fact that, for fixed $n$, we have $(\partial/\partial \theta_i) \EE \left( \left.  L_{\theta} \right| X \right) =  \EE \left( \left. (\partial/\partial \theta_i) L_{\theta} \right| X \right)$ and of $\sup_{\theta} |\EE(.)| \leq \EE( \sup_{\theta} |.| )$.
\end{proof}

\begin{lem} \label{lem:var:given:X:Ltheta}

Consider a fixed $\theta \in \Theta$. Then
\[
\EE \left(  \left|  L_{\theta} - \EE(L_{\theta} | X) \right|    \right) \to_{n \to \infty} 0.
\]
\end{lem}

\begin{proof}[Proof of Lemma \ref{lem:var:given:X:Ltheta}]

We have, applying Jensen inequality twice
\[
 \EE \left( \left | L_{\theta} - \EE(L_{\theta}|X) \right| \right) 
 \leq  \EE \left(  \sqrt{ var( L_{\theta} | X ) } \right)  
 \leq  \sqrt{ \EE \left(   var( L_{\theta} | X )  \right) }  
 =  \sqrt{ \EE \left(   \frac{2}{n^2} \Tr\left[ R_0 R_{\theta}^{-1} R_0 R_{\theta}^{-1}  \right]  \right) }.  
\]
The eigenvalues of $R_{\theta}^{-1}$ are smaller than a finite constant $C_{sup}$ for any $n,X,\theta$ from Lemma \ref{lem:controle:Rmun:diagRmun}. Thus, by applying Cauchy schwarz inequality and Lemmas \ref{lem:norm:Aun...Ak:carre} and \ref{lem:alternate:Rmun},
\[
\EE \left( \left | L_{\theta} - \EE(L_{\theta}|X) \right| \right)  \leq 
\sqrt{\frac{2}{n}} \sqrt{ \EE \left( |R_0 R_{\theta}^{-1} |^2 \right)} 
 \leq  \frac{C_{sup}}{\sqrt{n}}.
\]

\end{proof}

\begin{proof}[Proof of Theorem \ref{theorem:oracle:ML}]

We have
\begin{flalign*}
& \sup_{\theta \in \Theta} \left| L_{\theta} - \log(\det(R_0)) -1 - D_{n,\theta} \right| & \\
& \leq \sup_{\theta \in \Theta} \left| L_{\theta} - \EE(L_{\theta}|X) \right| + \sup_{\theta \in \Theta} \left| \EE(L_{\theta}|X) - \log(\det(R_0)) -1 - D_{n,\theta} \right|  & \\
& = \sup_{\theta \in \Theta} \left| L_{\theta} - \EE(L_{\theta}|X) \right| .
\end{flalign*}
The term in the above display goes to $0$ in probability. Indeed, for fixed $\theta$, the function of $\theta$ goes to $0$ in probability because of Lemma \ref{lem:var:given:X:Ltheta}. The convergence of the supremum over $\theta$ to $0$ is then a consequence of the fact that $\Theta$ is compact and of Lemma \ref{lem:control:derivative:ML} and corollary \ref{cor:controle:derivative:Ltheta}.

Finally, since $\hat{\theta}_{ML}$ minimizes $L_{\theta}$ and so also  $L_{\theta} - \log(\det(R_0)) -1$, we conclude with, for any $\theta \in \Theta$,
\begin{eqnarray*}
D_{n,\hat{\theta}_{ML}} - D_{n,\theta} & \leq & L_{\hat{\theta}_{ML}} - L_{\theta} + 2\sup_{\theta \in \Theta} \left| L_{\theta} - \log(\det(R_0)) -1 - D_{n,\theta} \right| \\ 
& \leq & 2\sup_{\theta \in \Theta} \left| L_{\theta} - \log(\det(R_0)) -1 - D_{n,\theta} \right| . \\
\end{eqnarray*}

Hence $\sup_{\theta \in \Theta} \left( D_{n,\hat{\theta}_{ML}} - D_{n,\theta} \right) = o_p(1)$.

\end{proof}

\subsection{Proofs for Cross Validation} \label{subsection:proof:CV}

\begin{lem} \label{lem:control:derivative:CV}
For $i=1,...,p$,
\[
\EE \left( \sup_{\theta \in \Theta} \left| \frac{\partial}{\partial \theta_i} CV_{\theta} \right| \right)
\]
is bounded w.r.t $n$.
\end{lem}
\begin{proof}[Proof of Lemma \ref{lem:control:derivative:CV}]
We have
\begin{eqnarray*}
\EE \left( \sup_{\theta \in \Theta} \left| \frac{\partial}{\partial \theta_i} CV_{\theta} \right| \right)
& \leq & \EE \left( \frac{1}{n} \sum_{k=1}^n  \sup_{\theta \in \Theta} \left| \frac{\partial}{\partial \theta_i} ( y_k - \hat{y}_{k,\theta} )^2 \right| \right) \\
\mbox{(symmetry of $X_1,...,X_n$:)} & = & \EE \left(   \sup_{\theta \in \Theta} \left| \frac{\partial}{\partial \theta_i} ( y_1 - \hat{y}_{1,\theta} )^2 \right| \right) \\
\mbox{(Lemma \ref{lem:sobolev}:)} & \leq & C_{sup} \sum_{i_1+...+i_p \leq p} \int_{\Theta} \EE \left(  \left| \frac{\partial^{i_1}}{\partial \theta_1^{i_1}}...\frac{\partial^{i_p}}{\partial \theta_p^{i_p}} \frac{\partial}{\partial \theta_i} ( y_1 - \hat{y}_{1,\theta} )^2 \right| \right) d \theta.
\end{eqnarray*}
Let us consider a specific $i_1,...,i_p$. Then $ [\partial^{i_1}/\partial \theta_1^{i_1}]...[\partial^{i_p}/\partial \theta_p^{i_p}] [\partial/\partial \theta_i] ( y_1 - \hat{y}_{1,\theta} )^2$ is a weighted sum (weights and number of terms depending only on $i_1,...,i_p$), so that the terms are of the two following forms: 
\[ 
( y_1 - \hat{y}_{1,\theta} ) \left( \frac{\partial^{k_1}}{\partial \theta_1^{k_1}}...\frac{\partial^{k_p}}{\partial \theta_p^{k_p}} \frac{\partial}{\partial \theta_i}  \hat{y}_{1,\theta} \right)
~ ~ ~
\mbox{or}
~ ~ ~
\left( \frac{\partial^{k_1}}{\partial \theta_1^{k_1}}...\frac{\partial^{k_p}}{\partial \theta_p^{k_p}} \frac{\partial}{\partial \theta_i}  \hat{y}_{1,\theta} \right)\left( \frac{\partial^{l_1}}{\partial \theta_1^{l_1}}...\frac{\partial^{l_p}}{\partial \theta_p^{l_p}} \frac{\partial}{\partial \theta_i}  \hat{y}_{1,\theta} \right).
\]
Thus, we just have to show that the mean values of the absolute values of the terms of the form above (for $k_1+...+k_p \leq p$ and $l_1+...+l_p \leq p$) are bounded uniformly in $\theta \in \Theta$. 
By using Cauchy-Schwarz inequality, these means of absolute values are smaller than either 
\[
\sqrt{ \EE\left( ( y_1 - \hat{y}_{1,\theta} )^2 \right)  } \sqrt{ \EE\left( \left( \frac{\partial^{k_1}}{\partial \theta_1^{k_1}}...\frac{\partial^{k_p}}{\partial \theta_p^{k_p}} \frac{\partial}{\partial \theta_i}  \hat{y}_{1,\theta} \right)^2 \right)  }
\]
or 
\[
 \sqrt{ \EE\left( \left( \frac{\partial^{k_1}}{\partial \theta_1^{k_1}}...\frac{\partial^{k_p}}{\partial \theta_p^{k_p}} \frac{\partial}{\partial \theta_i}  \hat{y}_{1,\theta} \right)^2 \right)  }
\sqrt{ \EE\left( \left( \frac{\partial^{l_1}}{\partial \theta_1^{l_1}}...\frac{\partial^{l_p}}{\partial \theta_p^{l_p}} \frac{\partial}{\partial \theta_i}  \hat{y}_{1,\theta} \right)^2 \right)  }.
\]
Now, $\EE\left( ( y_1 - \hat{y}_{1,\theta} )^2 \right) \leq 2 \EE( y_1^2 ) + 2 \EE( \hat{y}_{1,\theta}^2 )$. The term $\EE( y_1^2 )$ is bounded uniformly in $\theta$. Thus, finally, it remains to show that for any $a_1+...+a_p \leq p+1$, $\sup_{\theta \in \Theta} \EE\left( \left( [\partial^{a_1}/\partial \theta_1^{a_1}]...[\partial^{a_p}/\partial \theta_p^{a_p}]   \hat{y}_{1,\theta} \right)^2 \right) $ is bounded.
For that, we have $\hat{y}_{1,\theta} = r_{1,\theta}^t R_{1,\theta}^{-1} y_{-1}$. Thus, $[\partial^{a_1}/\partial \theta_1^{a_1}]...[\partial^{a_p}/\partial \theta_p^{a_p}]   \hat{y}_{1,\theta}$ is a fixed sum of weighted terms of the form $w_{\theta}^t M_{\theta} y_{-1}$, where $w_{\theta}$ is of the form $[\partial^{b_1}/\partial \theta_1^{b_1}]...[\partial^{b_p}/\partial \theta_p^{b_p}] r_{1,\theta}$ ($b_1+...+b_p \leq p+1$) and $M_{\theta}$ is of the form $R_{1,\theta}^{-1} M_{1,\theta}...R_{1,\theta}^{-1} M_{k,\theta} R_{1,\theta}^{-1}$. Finally, $k$ is smaller than a finite constant $C_{sup}$ (function of $p$) and $M_{i,\theta}$ is of the form $[\partial^{c_1}/\partial \theta_1^{c_1}]...[\partial^{c_p}/\partial \theta_p^{c_p}] R_{1,\theta}$, with $c_1+...+c_p \leq p+1$.
Thus, it is sufficient to show that a generic $\sup_{\theta \in \Theta} \EE \left( \left( w_{\theta}^ tM_{\theta}y_{-1} \right)^2 \right)$, as previously defined, is bounded.

Then,
\begin{eqnarray} \label{eq:cond:on:2n:trick}
 \sup_{\theta \in \Theta} \EE \left( \left( w_{\theta}^ tM_{\theta}y_{-1} \right)^2 \right)
 & = &  \sup_{\theta \in \Theta} \EE_X \EE_{y|X} \left(  y_{-1}^t M_{\theta}^t w_{\theta} w_{\theta}^t M_{\theta} y_{-1}  \right) \nonumber \\
  & = &  \sup_{\theta \in \Theta} \EE_X \Tr \left(  R_{1,0} M_{\theta}^t w_{\theta} w_{\theta}^t M_{\theta}   \right) \nonumber \\
   & \leq &  \sup_{\theta \in \Theta} \EE_X \left[ \sum_{i,j=2}^n \left| \left( M_{\theta}  R_{1,0} M_{\theta}^t \right)_{i,j} \right| \left| \left(w_{\theta} w_{\theta}^t \right)_{i,j} \right|  \right] 
      \\
        & = &  \sup_{\theta \in \Theta} \left[  \sum_{i,j=2}^n \EE_{X_2,...,X_n} \left( \left| \left( M_{\theta}  R_{1,0} M_{\theta}^t \right)_{i,j} \right| \EE_{X_1|X_2,...,X_n} \left| \left(w_{\theta} w_{\theta}^t \right)_{i,j} \right| \right) \right]. \nonumber
\end{eqnarray}
Now, because of Conditions \ref{cond:iid_obs_points} and \ref{cond:Ktheta:deltatheta},
\begin{eqnarray*}
\EE_{X_1|X_2,...,X_n} \left| \left(w_{\theta} w_{\theta}^t \right)_{i,j} \right| & \leq &  \frac{C_{sup}}{n} \int_{[0,n^{1/d}]^d} \frac{1}{1+|X_i-x_1|^{d+1}} \frac{1}{1+|X_j-x_1|^{d+1}} dx_1 \\
\mbox{ (Lemma \ref{lem:integrale:abc}:)} & \leq & \frac{1}{n} \frac{C_{sup}}{1+|X_i-X_j|^{d+1}}.
\end{eqnarray*}

So,
\begin{eqnarray*}
 \sup_{\theta \in \Theta}  \EE \left( \left( w_{\theta}^ tM_{\theta}y_{-1} \right)^2 \right)
 & \leq & C_{sup}  \frac{1}{n} \sup_{\theta \in \Theta} \left[ \sum_{i,j=2}^n \EE_{X_2,...,X_n} \left( \left| \left( M_{\theta}  R_{1,0} M_{\theta}^t \right)_{i,j} \right| \frac{1}{1+|X_i-X_j|^{d+1}} \right) \right] \\
 \mbox{(Cauchy-Schwarz:)} & \leq &  \sup_{\theta \in \Theta} \left[ \sqrt{ \EE \left\{ \left| M_{\theta} R_{1,0} M_{\theta}^t \right|^2 \right\}  } \sqrt{ \EE \left\{ \frac{1}{n} \sum_{i,j=2}^n \left( \frac{1}{1+|X_i-X_j|^{d+1}} \right)^2 \right\} } \right].
 \end{eqnarray*}
The supremum over $\theta$ of the second term above is bounded because of Lemma \ref{lem:norm:Aun...Ak:carre}. The supremum over $\theta$ of the first term above is bounded because of Lemmas \ref{lem:norm:Aun...Ak:carre}, \ref{lem:controle:Rmun:diagRmun} and \ref{lem:alternate:Rmun}.

\end{proof}

\begin{cor} \label{cor:controle:derivative}
For any $i=1,...,p$,
\[
\EE \left( \sup_{\theta \in \Theta} \left| \frac{\partial}{\partial \theta_i} \EE \left( \left.  CV_{\theta} \right| X \right) \right| \right)
~ ~ ~
\mbox{and}
~ ~ ~
 \sup_{\theta \in \Theta} \left| \frac{\partial}{\partial \theta_i}  \EE \left(  CV_{\theta}  \right) \right|
\]
are bounded w.r.t $n$.
\end{cor}

\begin{proof}[Proof of corollary \ref{cor:controle:derivative}]
The corollary is a consequence of Lemma \ref{lem:control:derivative:CV}, $\sup_{\theta} |\EE(.)| \leq \EE( \sup_{\theta} |.| )$ and of the fact that, for fixed $n$, we have $(\partial/\partial \theta_i) \EE \left( \left.  CV_{\theta} \right| X \right) =  \EE \left( \left. (\partial/\partial \theta_i) CV_{\theta} \right| X \right)$
and $(\partial/\partial \theta_i) \EE( CV_{\theta} ) =  \EE( (\partial/\partial \theta_i) CV_{\theta} )$.
\end{proof}

\begin{lem} \label{lem:var:given:X:CVtheta}

For any fixed $\theta \in \Theta$ we have
\[
\EE \left(  \left| CV_{\theta} - \EE( CV_{\theta} | X ) \right|\right) \to_{n \to \infty} 0.
\]

\end{lem}

\begin{proof}[Proof of Lemma \ref{lem:var:given:X:CVtheta}]

From \eqref{eq:CVtheta}, we have $CV_{\theta} = y^t M_{\theta} y$, with $M_{\theta} = (1/n) R_{\theta}^{-1} diag(R_{\theta}^{-1})^{-2} R_{\theta}^{-1}$. Because of Lemma \ref{lem:controle:Rmun:diagRmun}, the eigenvalues of $M_{\theta}$ are bounded uniformly in $n,X,\theta$ by a finite constant $C_{sup}$. Thus, the proof of the lemma is exactly the same as that of Lemma \ref{lem:var:given:X:Ltheta}, with $R_{\theta}^{-1}$ replaced by $M_{\theta}$.

\end{proof}

\begin{defi} \label{def:blockwise}

Consider a fixed $\theta \in \Theta$.
Consider two functions of $n$: $n_2(n) \in \NN^*$ and $\Delta(n) \geq 0$, that we write $n_2$ and $\Delta$ for simplicity, so that, for any $n \in \NN^*$, $n_2$ can be written $n_2 = N_2^d$, with $N_2 \in \NN^*$, and so that $n = n_2 \Delta$. 
Let, for $i=1,...,N_2-1$, $c_i = [((i-1)/N_2)n^{1/d},(i/N_2)n^{1/d})$. Let $c_{N_2} = [((N_2-1)/N_2)n^{1/d},n^{1/d}]$. Let, for $x \in [0,n^{1/d}]$, $i(x)$ be the unique $i \in \{1,...,N_2\}$ so that $x \in c_i$. Let, for $t=(t_1,...,t_d)^t \in [0,n^{1/d}]^d$, $C(t) = \prod_{j=1}^d c_{i(t_j)}$.
Define the non-stationary covariance function $\tilde{K}_{\theta}(t_1,t_2) = K_{\theta}(t_1,t_2) \indun_{C(t_1) = C(t_2)}$.  Define $\Rt_{\theta}$, $\Rt_{i,\theta}$, $\rt_{i,\theta}$, $\hatty_{i,\theta}$, $\tilde{CV}_{\theta}$ similarly to $R_{\theta}$, $R_{i,\theta}$, $r_{i,\theta}$, $\hat{y}_{i,\theta}$, $CV_{\theta}$ but with $K_{\theta}$  replaced by $\Kt_{\theta}$. 
Furthermore, let us write the $n_2$ aforementioned sets of the form $\prod_{j=1}^d c_{i_j}$, for $i_1,...,i_d \in \{1,...,N_2\}$, as the sets $C_1,...,C_{n_2}$. [The specific one-to-one correspondence we use between $\{1,...,N_2\}^d$ and $\{1,...,n_2\}$ is of no interest. Note that this one-to-one correspondence depends on $n$. The sets $C_1,...,C_{n_2}$ also depend of $n$, but we drop this dependence in the notation for simplicity.]

Let $N_i$ be the random number of observation points in $C_i$ and let $X^i$ be the random $N_i$-tuple obtained from $X$ by keeping only the observation points that are in $C_i$ and by preserving the order of the indices in $X$.
Let $y^i$ be the column vector of size $N_i$, composed by the components $y_j$ of $y$ for which $X_j$ is in $C_i$ (preserving the order of indexes).
Let $\Rb_{i,\theta}$ and $\Rb_{i,0}$ be the covariance matrices, under $(K_{\theta},\delta_{\theta})$ and $(K_0,\delta_0)$, of $y^i$, given $X$.

Finally, for $1 \leq i,j \leq n_2$, let $v_i$ and $w_j$ be two $N_i \times 1$ and $N_j \times 1$ vectors and $M^{ij}$ be a $N_i \times N_j$ matrix.
Then we use the convention that, when $N_i=0$, $|M^{ij}| = ||M^{ij}|| = 0$, $||v_i||=|v_i|=0$ and $v_i^t M^{ij} w_j=0$.
Furthermore, if $i=j$ and $M^{ii}$ is invertible when $N_i  \geq 1$, we use the convention that $v_i^t (M^{ii})^{-1} w_i = 0$ when $N_i=0$.
[These conventions enable to write equalities or inequalities involving matrices and vectors of size $N_i$, $N_j$ or $N_i \times N_j$, that hold regardless of whether $N_i$ or $N_j$ are zero or not. As can be checked along the proofs involving Definition \ref{def:blockwise}, these relations boil down to trivial relations (e.g. $0=0$) when $N_i=0$ or $N_j=0$. This way of proceeding considerably simplifies the exposition in these proofs.]

\end{defi}

\begin{lem} \label{lem:blockwise}

Consider a fixed $\theta \in \Theta$.
In the context of Definition \ref{def:blockwise}, if $n_2 = o(n)$,
\[
\EE \left( \left| CV_{\theta} - \tilde{CV}_{\theta} \right| \right) \to_{n \to \infty} 0.
\]

\end{lem}

\begin{proof}[Proof of Lemma \ref{lem:blockwise}]

Assume that $n_2 = o(n)$, or equivalently that $\Delta \to_{n \to \infty} \infty$. We have
\begin{eqnarray*}
\EE \left( \left| CV_{\theta} - \tilde{CV}_{\theta} \right| \right) & \leq &
\frac{1}{n} \sum_{i=1}^n \EE \left( \left| (y_i - \hat{y}_{i,\theta})^2 - (y_i - \hatty_{i,\theta})^2 \right| \right) \\
 \mbox{(symmetry:)} & = & \EE \left( \left| (y_1 - \hat{y}_{1,\theta})^2 - (y_1 - \hatty_{1,\theta})^2 \right| \right) \\
 & = & \EE \left( \left| (y_1 - r_{1,\theta}^t R_{1,\theta}^{-1} y_{-1})^2 - (y_1 - \rt_{1,\theta}^t \Rt_{1,\theta}^{-1} y_{-1})^2 \right| \right) \\
  & = & \EE \left( \left| \rt_{1,\theta}^t \Rt_{1,\theta}^{-1} y_{-1}  - r_{1,\theta}^t R_{1,\theta}^{-1} y_{-1} \right| \left| 2 y_1 - r_{1,\theta}^t R_{1,\theta}^{-1} y_{-1} - \rt_{1,\theta}^t \Rt_{1,\theta}^{-1} y_{-1} \right| \right) \\
\mbox{ (Cauchy-Schwarz:)}   & \leq & \sqrt{\EE \left( \left( \rt_{1,\theta}^t \Rt_{1,\theta}^{-1} y_{-1}  - r_{1,\theta}^t R_{1,\theta}^{-1} y_{-1} \right)^2 \right)} \\
& &  \sqrt{ \EE \left( \left( 2 y_1 - r_{1,\theta}^t R_{1,\theta}^{-1} y_{-1} - \rt_{1,\theta}^t \Rt_{1,\theta}^{-1} y_{-1} \right)^2 \right) }
\end{eqnarray*}

Now, the second square root in the above display is bounded, because of $(a+b+c)^2 \leq 3 \left( a^2 + b^2 + c^2 \right)$ and of arguments similar to but simpler than those given in the proof of Lemma \ref{lem:control:derivative:CV}.
Thus it only remains to show that
\[
\EE \left( \left( \rt_{1,\theta}^t \Rt_{1,\theta}^{-1} y_{-1}  - r_{1,\theta}^t R_{1,\theta}^{-1} y_{-1} \right)^2 \right) \to_{n \to \infty} 0.
\]
For this,
\begin{eqnarray} \label{eq:in:proof:lem:blockwise:two:terms}
\EE \left( \left( \rt_{1,\theta}^t \Rt_{1,\theta}^{-1} y_{-1}  - r_{1,\theta}^t R_{1,\theta}^{-1} y_{-1} \right)^2 \right) & \leq & 2 \EE \left( \left( \rt_{1,\theta}^t ( \Rt_{1,\theta}^{-1}  - R_{1,\theta}^{-1} ) y_{-1}  \right)^2 \right)  \\
& & + 2  \EE \left( \left( (\rt_{1,\theta} - r_{1,\theta})^t R_{1,\theta}^{-1} y_{-1}  \right)^2 \right). \nonumber
\end{eqnarray}

We show separately that both terms in the right-hand side of \eqref{eq:in:proof:lem:blockwise:two:terms} converge to $0$.
For the first term,
\begin{eqnarray*}
\EE \left( \left( \rt_{1,\theta}^t ( \Rt_{1,\theta}^{-1}  - R_{1,\theta}^{-1} ) y_{-1}  \right)^2 \right)
& = & \EE \left( \Tr \left[ R_{1,0} ( \Rt_{1,\theta}^{-1}  - R_{1,\theta}^{-1} ) \rt_{1,\theta} \rt_{1,\theta}^t
( \Rt_{1,\theta}^{-1}  - R_{1,\theta}^{-1} ) \right] \right) \\
& \leq & \sum_{i,j=2}^n \EE \left( | ( \Rt_{1,\theta}^{-1}  - R_{1,\theta}^{-1} ) R_{1,0} ( \Rt_{1,\theta}^{-1}  - R_{1,\theta}^{-1} )  |_{i,j} |\rt_{1,\theta} \rt_{1,\theta}^t|_{i,j} \right). 
\end{eqnarray*}
Hence, by the same arguments as after \eqref{eq:cond:on:2n:trick} in the proof of Lemma \ref{lem:control:derivative:CV}, we obtain
\begin{eqnarray*}
\left[ \EE \left( \left( \rt_{1,\theta}^t ( \Rt_{1,\theta}^{-1}  - R_{1,\theta}^{-1} ) y_{-1}  \right)^2 \right) \right]^2 & \leq & C_{sup} 
\EE \left( | ( \Rt_{1,\theta}^{-1}  - R_{1,\theta}^{-1} ) R_{1,0} ( \Rt_{1,\theta}^{-1}  - R_{1,\theta}^{-1} )  |^2 \right)  \\
& \leq & C_{sup} \EE \left( \left\{ ||\Rt_{1,\theta}^{-1} || +|| R_{1,\theta}^{-1}|| \right\} |  R_{1,0} ( \Rt_{1,\theta}^{-1}  - R_{1,\theta}^{-1} )  |^2  \right) \\
\mbox{(Lemmas \ref{lem:controle:Rmun:diagRmun} and \ref{lem:controle:Rmun:diagRmun:tilde}:)} & \leq &
\frac{C_{sup}}{n} \EE \left(  \Tr\left[ ( \Rt_{1,\theta}^{-1}  - R_{1,\theta}^{-1} ) R_{1,0}^2 ( \Rt_{1,\theta}^{-1}  - R_{1,\theta}^{-1} ) \right]  \right) \\
\mbox{(Cauchy-Schwarz:)} & \leq & C_{sup} \sqrt{\EE \left( | ( \Rt_{1,\theta}^{-1}  - R_{1,\theta}^{-1} )^2 |^2 \right)} \sqrt{\EE \left( | R_{1,0}^2 |^2 \right)}. 
\end{eqnarray*}

From Lemma \ref{lem:norm:Aun...Ak:carre}, $\EE \left( | R_{1,0}^2 |^2 \right)$ is bounded, so it remains to show that $\EE \left( | ( \Rt_{1,\theta}^{-1}  - R_{1,\theta}^{-1} )^2 |^2 \right)$ converges to $0$.
For this,
\begin{eqnarray*}
\EE \left( | ( \Rt_{1,\theta}^{-1}  - R_{1,\theta}^{-1} )^2 |^2 \right) & = & 
\EE \left( | ( \Rt_{1,\theta}^{-1} ( R_{1,\theta} - \Rt_{1,\theta} ) R_{1,\theta}^{-1} )^2 |^2 \right) \\
\mbox{(Lemma \ref{lem:controle:Rmun:diagRmun}:)}  & \leq & C_{sup} \EE \left( |  ( R_{1,\theta} - \Rt_{1,\theta} )  R_{1,\theta}^{-1} \Rt_{1,\theta}^{-1} ( R_{1,\theta} - \Rt_{1,\theta} )  |^2 \right) \\
& = & C_{sup} \frac{1}{n} \EE\left(  \Tr \left[  ( R_{1,\theta} - \Rt_{1,\theta} )^2  R_{1,\theta}^{-1} \Rt_{1,\theta}^{-1} ( R_{1,\theta} - \Rt_{1,\theta} )^2  \Rt_{1,\theta}^{-1} R_{1,\theta}^{-1} \right] \right) \\
\mbox{(Cauchy-Schwarz:)}  & \leq & C_{sup} \sqrt{ \EE \left( |  ( R_{1,\theta} - \Rt_{1,\theta} )^2  R_{1,\theta}^{-1} \Rt_{1,\theta}^{-1}  |^2 \right) }  
\sqrt{ \EE \left( |  ( R_{1,\theta} - \Rt_{1,\theta} )^2  \Rt_{1,\theta}^{-1} R_{1,\theta}^{-1}  |^2 \right) }.   \\
\end{eqnarray*}
Hence, with Lemmas \ref{lem:controle:Rmun:diagRmun}, \ref{lem:controle:Rmun:diagRmun:tilde} and \ref{lem:norm:Run:moins:Rt1:carre:carre}, we conclude that the first term of the right hand side of \eqref{eq:in:proof:lem:blockwise:two:terms} goes to $0$.
Let us now show that the second term of the right hand side of \eqref{eq:in:proof:lem:blockwise:two:terms} goes to $0$. 
We have,
\begin{flalign*}
& \EE \left( \left( (\rt_{1,\theta} - r_{1,\theta})^t R_{1,\theta}^{-1} y_{-1}  \right)^2 \right)  & \\
& = 
\EE \left( \Tr \left( R_{1,\theta}^{-1} R_{1,0} R_{1,\theta}^{-1} (\rt_{1,\theta} - r_{1,\theta})(\rt_{1,\theta} - r_{1,\theta})^t   \right) \right) & \\
 & \leq   \sum_{i,j=2}^n \EE \left( \left| [R_{1,\theta}^{-1} R_{1,0} R_{1,\theta}^{-1}]_{i,j} \right|  \frac{1}{n} \int_{[0,n^{1/d}]^d} \frac{1}{1+|X_i - x_1|^{d+1}} \frac{1}{1+|X_j - x_1|^{d+1}} \indun_{C(X_i) \neq C(x_1)} \indun_{C(X_j) \neq C(x_1)} dx_1 \right), & 
\end{flalign*}
where the last line is obtained similarly to after \eqref{eq:cond:on:2n:trick} in the proof of Lemma \ref{lem:control:derivative:CV}.
Thus we have, with the notation and result of Lemma \ref{lem:integrale:abc:T},
\begin{eqnarray*}
\EE \left( \left( (\rt_{1,\theta} - r_{1,\theta})^t R_{1,\theta}^{-1} y_{-1}  \right)^2 \right)
& \leq & C_{sup} \frac{1}{n} \sum_{i,j=2}^n \EE \left( \left| [R_{1,\theta}^{-1} R_{1,0} R_{1,\theta}^{-1}]_{i,j} \right|  \frac{1}{1+|X_i-X_j|^{d+1}} f(D_{\Delta}(X_i,X_j)) \right) \\
\mbox{(Cauchy-Schwarz:)} & \leq  & \sqrt{ \EE \left( | R_{1,\theta}^{-1} R_{1,0} R_{1,\theta}^{-1} |^2 \right) } 
\\
& & \sqrt{ \frac{1}{n} \sum_{i,j=2}^n  \EE  \left[  \left( \frac{1}{1+|X_i-X_j|^{d+1}} \right)^2 f^2(D_{\Delta}(X_i,X_j)) \right] }.
\end{eqnarray*}
From Lemmas \ref{lem:norm:Aun...Ak:carre}, \ref{lem:controle:Rmun:diagRmun} and \ref{lem:alternate:Rmun}, the first $\sqrt{.}$ in the above display is bounded. Thus it remains to show that the second $\sqrt{.}$ goes to $0$. For this, noting that $f^2(t) \leq C_{sup} f(t)$ and distinguishing the case $i=j$ from the case $i \neq j$,
\begin{flalign} \label{eq:in:proof:lem:blockwise:int:f:Delta}
& \frac{1}{n} \sum_{i,j=2}^n \EE \left[ \left( \frac{1}{1+|X_i-X_j|^{d+1}} \right)^2 f(D_{\Delta}(X_1,X_j)) \right] & \nonumber  \\
& \leq  \frac{C_{sup}}{n} \int_{[0,n^{1/d}]^d} f (D_{\Delta}(x) ) dx 
 + \frac{C_{sup}}{n} \int_{[0,n^{1/d}]^d} dx_1 \int_{[0,n^{1/d}]^d} dx_2  \frac{1}{1+|x_1-x_2|^{d+1}} f(D_{\Delta}(x_1,x_2) ) \nonumber & \\
 & =  \frac{C_{sup}}{n} \int_{[0,n^{1/d}]^d} f (D_{\Delta}(x) ) dx + o(1) ~ ~ ~ \mbox{(Lemma \ref{lem:integrale:abf}).} &
\end{flalign}
Now, for any $\epsilon>0$, there is a finite $T$ so that $f(T) \leq \epsilon$, and by defining $E_n = \{x \in [0,n^{1/d}]^d; D_{\Delta}(x) \leq T \}$, we have $|E_n| = o(n)$, as can be seen easily, and
\[
\frac{1}{n} \int_{[0,n^{1/d}]^d} f (D_{\Delta}(x) ) dx \leq f(0) \frac{|E_n|}{n} + \epsilon.
\]
This finally shows that the second term of the right hand side of \eqref{eq:in:proof:lem:blockwise:two:terms} goes to $0$ which finishes the proof.

\end{proof}

\begin{lem} \label{lem:var:ECV:sachant:X}

For any fixed $\theta \in \Theta$, 
\[
\EE \left( \left| \EE \left(CV_{\theta}\right) - \EE \left[ CV_{\theta} | X \right] \right| \right)
\]
goes to $0$ as $n \to \infty$.

\end{lem}

\begin{proof}[Proof of Lemma \ref{lem:var:ECV:sachant:X}]

Fix $\theta \in \Theta$. Because of Lemma \ref{lem:blockwise} and of $|\EE(.)| \leq \EE(|.|)$, it is sufficient to show that there exists a sequence $\Delta \to + \infty$ so that the lemma holds with $CV_{\theta}$ replaced by $\tilde{CV}_{\theta}$. Then, because of $\left( \EE(.) \right)^2 \leq \EE \left( (.)^2 \right)$, it is sufficient to show $var \left( \EE \left[ \left. \tilde{CV}_{\theta} \right| X \right] \right) \to_{n \to \infty} 0$.

Let $C_1,...,C_{n_2}$ be as in Definition \ref{def:blockwise}. Define, for $k=1,...,n_2$,
\[
f_k(X) = \frac{1}{\Delta} \sum_{X_i \in C_k} \EE \left( \left. \left[ y_i - \tilde{\hat{y}}_{i,\theta} \right]^2 \right| X  \right). 
\]
[Note that, following the discussion in Definition \ref{def:blockwise}, we have $f_k(X) =0$ if $N_k=0$ and $f_k(X) = K_0(0) + \delta_0$ if $N_k=1$.]
Then $ \EE \left( \tilde{CV}_{\theta} | X \right) = (1/n_2) \sum_{k=1}^{n_2} f_k(X)$. Let $\Rb_{k,\theta}$ and $\Rb_{k,0}$ be as in Definition \ref{def:blockwise}. Because of the definition of $\tilde{K}$ and by \eqref{eq:CVtheta}, we have 
\[
f_k(X) = \frac{1}{\Delta} \Tr \left( \Rb_{k,0} \Rb_{k,\theta}^{-1} diag( \Rb_{k,\theta}^{-1} )^{-2} \Rb_{k,\theta}^{-1} \right).
\]

The functions $f_k(X)$ satisfy the conditions of Lemma \ref{lem:dist:fk}. Furthermore, by using the notation $N_k$ of Lemma \ref{lem:dist:fk}, we have
\begin{eqnarray*}
\EE \left( \left. f^2_{k}(X) \right| N_k = N \right) & = & \frac{1}{\Delta^2} \EE \left( \left. \left[ \Tr \left( \Rb_{k,0} \Rb_{k,\theta}^{-1} diag( \Rb_{k,\theta}^{-1} )^{-2} \Rb_{k,\theta}^{-1} \right) \right]^2 \right| N_k = N \right) \\
\mbox{(Cauchy-Schwarz:)} & \leq & \frac{1}{\Delta^2} \EE \left( \left. N^2 |\Rb_{k,0}|^2 
   | \Rb_{k,\theta}^{-1} diag( \Rb_{k,\theta}^{-1} )^{-2} \Rb_{k,\theta}^{-1}  |^2 \right| N_k=N \right) \\
\mbox{(Lemma \ref{lem:controle:Rmun:diagRmun:bar}:)} & \leq & C_{sup} \frac{N^2}{\Delta^2} \EE \left( \left. |\Rb_{k,0}|^2 \right| N_k=N \right) \\
\mbox{(Condition \ref{cond:Kzero:deltazero} and Lemma \ref{lem:dist:fk}:)} & \leq & C_{sup} \frac{N^2}{\Delta^2}  \left( 1 +  \frac{N}{\Delta^2} \int_{[0,\Delta^{1/d}]^d} \int_{[0,\Delta^{1/d}]^d} \frac{1}{1+|x_1-x_2|^{d+1}} dx_1 dx_2 \right) \\
& \leq & C_{sup} \left( \frac{N^2}{\Delta^2} + \frac{N^3}{\Delta^3} \right) \\
& \leq & C_{sup} \left( 1 + \frac{N^4}{\Delta^4} \right). \\
\end{eqnarray*} 
Thus, because of Lemma \ref{lem:cvg:mean:fk}, there exists a sequence $\Delta \to_{n \to \infty} \infty$ so that $var \left( \EE \left[ \left. \tilde{CV}_{\theta} \right| X \right] \right) \to_{n \to \infty} 0$, which completes the proof.
\end{proof}

\begin{lem} \label{lem:blockwise:Entheta}
Let, with the notation of Definition \ref{def:blockwise}, $\tilde{E}_{n,\theta}$ be defined as $E_{n,\theta}$, with $K_{\theta}$ replaced by $\Kt_{\theta}$. Fix $\theta \in \Theta$. Then, if $n_2 = o(n)$, 
\[
\EE \left( \left| E_{n,\theta} - \tilde{E}_{n,\theta} \right| \right) \to_{n \to \infty} 0.
\]
\end{lem}

\begin{proof}[Proof of Lemma \ref{lem:blockwise:Entheta}]

We have, by letting $\tilde{\hat{y}}_{\theta}(t)$ be as $\hat{y}_{\theta}(t)$, with $K_{\theta}$ replaced by $\tilde{K}_{\theta}$.
\begin{eqnarray*}
\EE \left( \left| E_{n,\theta} - \tilde{E}_{n,\theta} \right| \right) & = & 
\EE \left( \left| \frac{1}{n} \int_{[0,n^{1/d}]^d} \left[ Y(t) - \hat{y}_{\theta}(t) \right]^2 dt - \frac{1}{n} \int_{[0,n^{1/d}]^d} \left[ Y(t) - \hatty_{\theta}(t) \right]^2 dt \right| \right) \\
& \leq & \EE \left(  \frac{1}{n} \int_{[0,n^{1/d}]^d} \left| \left[ Y(t) - \hat{y}_{\theta}(t) \right]^2 - \left[ Y(t) - \hatty_{\theta}(t) \right]^2  \right| dt \right) \\
\end{eqnarray*}

The variable $t$ in the integral above is formally equivalent to a new observation point $X_{n+1}$, so that $X_1,...,X_{n+1}$ are independent and uniformly distributed on $[0,n^{1/d}]^d$. Thus,

\begin{eqnarray*}
\EE \left( \left| E_{n,\theta} - \tilde{E}_{n,\theta} \right| \right)
& \leq & \EE \left( \left| \left( Y(X_{n+1}) -\hat{y}_{\theta}(X_{n+1}) \right)^2 - \left( Y(X_{n+1}) -\tilde{\hat{y}}_{\theta}(X_{n+1}) \right)^2  \right| \right).
\end{eqnarray*}
The rest of the proof is carried out as in Lemma \ref{lem:blockwise}, the only difference being that there are $n+1$ observation points instead of $n$.

\end{proof}

\begin{lem} \label{lem:var:given:X:Entheta}

For any fixed $\theta \in \Theta$ we have
\[
\EE \left(  \left| E_{n,\theta} - \EE( E_{n,\theta} | X ) \right| \right) \to_{n \to \infty} 0.
\]

\end{lem}

\begin{proof}[Proof of Lemma \ref{lem:var:given:X:Entheta}]

Because of Lemma \ref{lem:blockwise:Entheta} and using $|\EE(.)| \leq \EE(|.|)$ and $\EE^2(.) \leq \EE( (.)^2 )$, it is sufficient to show that there exists a sequence $\Delta \to_{n \to \infty}$ so that

\[
\EE \left( var \left( \left. \tilde{E}_{n,\theta} \right| X \right) \right)
\]
goes to $0$ as $n \to \infty$.

Let us use the notation $C_1,...,C_{n_2}$ of Definition \ref{def:blockwise}. Let, for $t \in \RR^d$ and $v = (v_1,...,v_m) \in (\RR^d)^m$, $r_{\theta} (t,v) = (K_{\theta}(t,v_1),...,K_{\theta}(t,v_m))^t$. We define $r_0(t,v)$ similarly. Let $y^i$, $\bar{R}_{i,\theta}$ and $\bar{R}_{i,0}$ be as in Definition \ref{def:blockwise}. Let for $i \neq j$, $R_{0} (X^i,X^j) = \left[K_{0}((X^i)_k,(X^j)_l)\right]_{k=1,...,N_i;l=1,...,N_j}$. Let $R_{0} (X^i,X^i) = \left[K_{0}((X^i)_k,(X^i)_l)\right]_{k,l=1,...,N_i} + \delta_0 I_{N_i}$.
Then,
\[
\tilde{E}_{n,\theta} = \frac{1}{n_2} \sum_{i=1}^{n_2} \frac{1}{\Delta} \int_{C_i} dt_i \left[ Y(t_i) - r_{\theta}^t(t_i,X^i) \bar{R}_{i,\theta}^{-1} y^i \right]^2
\]
Hence, using the relation $cov \left( A^2 , B^2 \right) = 2 \left( cov(A,B) \right)^2$, for two centered Gaussian variables $A$ and $B$, we obtain
\begin{flalign*}
& var \left( \left. \tilde{E}_{n,\theta} \right| X \right) & \\
& =  \frac{2}{n_2} \sum_{i=1}^{n_2} \frac{1}{n_2} \sum_{j=1}^{n_2} \frac{1}{\Delta^2} \int_{C_i} dt_i \int_{C_j} dt_j cov^2 \left( \left. \left[ Y(t_i) - r_{\theta}^t(t_i,X^i) \bar{R}_{i,\theta}^{-1} y^i \right] , \left[ Y(t_j) - r_{\theta}^t(t_j,X^j) \bar{R}_{j,\theta}^{-1} y^j \right] \right| X \right) & \\
& =  \frac{1}{n_2} \sum_{i=1}^{n_2} \frac{1}{n_2} \sum_{j=1}^{n_2} \frac{1}{\Delta^2} \int_{C_i} dt_i \int_{C_j} dt_j \left\{  K_0(t_i,t_j) - r_{\theta}^t(t_i,X^i) \bar{R}_{i,\theta}^{-1} r_{0}(t_j,X^i)
- r_{\theta}^t(t_j,X^j) \bar{R}_{j,\theta}^{-1} r_{0}(t_i,X^j) \right. & \\
& ~~~   \left. + r_{\theta}^t(t_i,X^i) \bar{R}_{i,\theta}^{-1} R_0(X^i,X^j) \bar{R}_{j,\theta}^{-1} r_{\theta}(t_j,X^j) \right\}^2. &
\end{flalign*}

Now, we use $(a_1+a_2+a_3+a_4)^2 \leq 4 (a_1^2+a_2^2+a_3^2+a_4^2)$. Hence we obtain
\begin{equation} \label{eq:in:proof:var:given:X:Entheta:def:Tun:quatre}
\EE \left( var \left( \left. \tilde{E}_{n,\theta} \right| X \right) \right) \leq C_{sup} \left( T_1+T_2+T_3+T_4 \right),
\end{equation}
where $T_1,T_2,T_3,T_4$ are defined and treated below, and with $T_2 = T_3$ by symmetry.

For $T_1$,
\begin{eqnarray} \label{eq:in:proof:var:given:X:Entheta:Tun}
T_1 & = & \frac{1}{n_2} \sum_{i=1}^{n_2} \frac{1}{n_2} \sum_{j=1}^{n_2} \frac{1}{\Delta^2} \int_{C_i} dt_i \int_{C_j} dt_j K_0^2(t_i,t_j) \nonumber \\
\mbox{(Condition \ref{cond:Kzero:deltazero}:)} & \leq &  \frac{C_{sup}}{n_2} \sum_{i=1}^{n_2} \frac{1}{n_2} \frac{1}{\Delta^2} \int_{C_i} dt_i \int_{\RR^d} dt \left( \frac{1}{1+|t_i-t|^{d+1}} \right)^2 \nonumber \\
 & \leq & C_{sup} \frac{1}{n_2 \Delta}.
\end{eqnarray}

For $T_2$, using Cauchy-Schwarz and Lemma \ref{lem:controle:Rmun:diagRmun:bar},
\begin{flalign*}
& T_2 & \\
& = \frac{1}{n_2} \sum_{i=1}^{n_2} \frac{1}{n_2} \sum_{j=1}^{n_2} \frac{1}{\Delta^2} \int_{C_i} dt_i \int_{C_j} dt_j 
\EE \left[ \left( r_{\theta}^t(t_i,X^i) \bar{R}_{i,\theta}^{-1} r_{0}(t_j,X^i) \right)^2 \right] & \\
& \leq C_{sup} \frac{1}{n_2} \sum_{i=1}^{n_2} \frac{1}{n_2} \sum_{j=1}^{n_2} \frac{1}{\Delta^2} \int_{C_i} dt_i \int_{C_j} dt_j 
\EE \left[ || r_{\theta}(t_i,X^i)||^2 ||r_{0}(t_j,X^i)||^2  \right]. &
\end{flalign*}

Now, using the notation $N_i$ of Lemma \ref{lem:dist:fk} and Conditions \ref{cond:Kzero:deltazero} and \ref{cond:Ktheta:deltatheta},
\begin{flalign} \label{eq:in:proof:var:given:X:Entheta:Tdeux}
& T_2 & \nonumber \\
& \leq C_{sup} \frac{1}{n_2} \sum_{i=1}^{n_2} \frac{1}{n_2} \sum_{j=1}^{n_2} \frac{1}{\Delta^2} \int_{C_i} dt_i \int_{C_j} dt_j 
\EE \left[ N_i^2  \left\{  \frac{1}{1+d(C_i,C_j)^{d+1}} \right\}^4  \right] & \nonumber \\
& \leq C_{sup} \frac{1}{n_2} \sum_{i=1}^{n_2} \frac{1}{n_2} \sum_{j=1}^{n_2} \Delta^2
 \left\{ \frac{1}{1+d(C_i,C_j)^{d+1}} \right\}  ~ ~ \mbox{(Lemma \ref{lem:moment:binomial})} & \nonumber \\
& \leq C_{sup} \frac{\Delta^2}{n_2} \max_{i=1,...,n_2} \sum_{j=1}^{n_2}   \left\{ \frac{1}{1+d(C_i,C_j)^{d+1}} \right\} & \nonumber \\
& \leq C_{sup} \frac{\Delta^2}{n_2} ~ ~ \mbox{(Lemma \ref{lem:sum:d:Ci:Cj}, and because we will set $\Delta \to_{n \to \infty} \infty$)}.   & 
\end{flalign}

For $T_4$ in \eqref{eq:in:proof:var:given:X:Entheta:def:Tun:quatre}, using Cauchy-Schwarz and Lemma \ref{lem:controle:Rmun:diagRmun:bar},
\begin{flalign} \label{eq:in:proof:var:given:X:Entheta:Tquatre:a}
& T_4 & \nonumber \\
& = \frac{1}{n_2} \sum_{i=1}^{n_2} \frac{1}{n_2} \sum_{j=1}^{n_2} \frac{1}{\Delta^2} \int_{C_i} dt_i \int_{C_j} dt_j 
\EE \left[ \left( r_{\theta}^t(t_i,X^i) \bar{R}_{i,\theta}^{-1} R_0(X^i,X^j) \bar{R}_{j,\theta}^{-1} r_{\theta}(t_j,X^j) \right)^2 \right] & \nonumber \\
& \leq C_{sup} \frac{1}{n_2} \sum_{i=1}^{n_2} \frac{1}{n_2} \sum_{j=1}^{n_2} \frac{1}{\Delta^2} \int_{C_i} dt_i \int_{C_j} dt_j 
\sqrt{ \EE \left[ || r_{\theta}^t(t_i,X^i) ||^4 \right] }
\sqrt{ \EE \left[ ||  R_0(X^i,X^j) \bar{R}_{j,\theta}^{-1} r_{\theta}(t_j,X^j) ||^4 \right] }. &
\end{flalign}

Using Condition \ref{cond:Kzero:deltazero}, Lemma \ref{lem:controle:Rmun:diagRmun:bar} and Lemma \ref{lem:norm:A:b}, we obtain
 
\begin{eqnarray*}
||  R_0(X^i,X^j) \bar{R}_{j,\theta}^{-1} r_{\theta}(t_j,X^j) ||^2 & \leq &
C_{sup} N_i N_j \left\{ \frac{1}{1+d(C_i,C_j)^{d+1}} \right\}^2  ||\bar{R}_{j,\theta}^{-1} r_{\theta}(t_j,X^j)||^2 \\
& \leq & C_{sup} N_i N_j^2 \left\{ \frac{1}{1+d(C_i,C_j)^{d+1}} \right\}^2  .  \\
\end{eqnarray*}

Hence, going back to \eqref{eq:in:proof:var:given:X:Entheta:Tquatre:a},

\begin{flalign} \label{eq:in:proof:var:given:X:Entheta:Tquatre:b}
& T_4 & \nonumber \\
& \leq C_{sup} \frac{1}{n_2} \sum_{i=1}^{n_2} \frac{1}{n_2} \sum_{j=1}^{n_2} \frac{1}{\Delta^2} \int_{C_i} dt_i 
\int_{C_j} dt_j 
\sqrt{ \EE \left[ N_i^2 \right] }
 \left\{ \frac{1}{1+d(C_i,C_j)^{d+1}} \right\}^2
\sqrt{ \EE \left[ N_i^2 N_j^4 \right] } & \nonumber \\
& \leq C_{sup} \frac{1}{n_2} \sum_{i=1}^{n_2} \frac{1}{n_2} \sum_{j=1}^{n_2} \frac{1}{\Delta^2} \int_{C_i} dt_i \int_{C_j} dt_j 
 \left\{ \frac{1}{1+d(C_i,C_j)^{d+1}} \right\}^2
\sqrt{ \EE \left[ N_i^2 \right] }
\sqrt{ \sqrt{\EE \left[ N_i^4 \right]} \sqrt{\EE \left[ N_j^8 \right]} } & \nonumber \\
 & \leq C_{sup} \frac{\Delta^4}{n_2} ~ ~ \mbox{(Lemmas \ref{lem:moment:binomial} and \ref{lem:sum:d:Ci:Cj}).} & 
\end{flalign}

Hence, from \eqref{eq:in:proof:var:given:X:Entheta:Tun}, \eqref{eq:in:proof:var:given:X:Entheta:Tdeux} and \eqref{eq:in:proof:var:given:X:Entheta:Tquatre:b}, we can set $\Delta = n^{1/6}$ to complete the proof.

\end{proof}

\begin{lem} \label{lem:var:E:Entheta:sachant:X}

For any fixed $\theta \in \Theta$, 
\[
\EE \left( \left| \EE \left(E_{n,\theta}\right) - \EE \left[ E_{n,\theta} | X \right] \right| \right)
\]
goes to $0$ as $n \to \infty$.

\end{lem}

\begin{proof}[Proof of Lemma \ref{lem:var:E:Entheta:sachant:X}]

Fix $\theta \in \Theta$. Because of Lemma \ref{lem:blockwise:Entheta} and of $|\EE(.)| \leq \EE(|.|)$, it is sufficient to show that there exists a sequence $\Delta \to + \infty$ so that the lemma holds with $E_{n,\theta}$ replaced by $\tilde{E}_{n,\theta}$. Then, because of $\left( \EE(.) \right)^2 \leq \EE \left( (.)^2 \right)$, it is sufficient to show $var \left( \EE \left[ \left. \tilde{E}_{n,\theta} \right| X \right] \right) \to_{n \to \infty} 0$.

Let $C_1,...,C_{n_2}$ be as in Definition \ref{def:blockwise} and let $\tilde{\hat{y}}_{\theta}(t)$ be as in the proof of Lemma \ref{lem:blockwise:Entheta}. Define, for $k=1,...,n_2$,
\[
g_k(X) = \frac{1}{\Delta} \int_{C_k} dt_k \EE \left( \left. \left[ Y(t_k) - \tilde{\hat{y}}_{\theta}(t_k)  \right]^2 \right| X \right). 
\]
[Note that, following the discussion in Definition \ref{def:blockwise}, we have $g_k(x) = K_0(0)$ if $N_k = 0$.]
Then $ \EE \left( \tilde{E}_{n,\theta} | X \right) = (1/n_2) \sum_{k=1}^{n_2} g_k(X)$. Following the notation of Lemma \ref{lem:var:given:X:Entheta} we have,
\begin{flalign*}
& g_k(X)  = &  \\
& \frac{1}{\Delta} \int_{C_k} dt_k \EE \left( \left. \left[ Y(t_k) - r_{\theta}^t(t_k,X^k) \bar{R}_{k,\theta}^{-1}y^k  \right]^2 \right| X \right) & \\ 
& \leq  2 \frac{1}{\Delta} \int_{C_k} dt_k \left( K_0(0) +  r_{\theta}^t(t_k,X^k) \bar{R}_{k,\theta}^{-1} \bar{R}_{k,0} \bar{R}_{k,\theta}^{-1} r_{\theta}(t_k,X^k) \right) & \\ 
 & \leq  C_{sup} + 2 \frac{1}{\Delta} \int_{C_k} dt_k   ||r_{\theta}^t(t_k,X^k)|| || \bar{R}_{k,0} \bar{R}_{k,\theta}^{-1} r_{\theta}(t_k,X^k) || ~ ~ ~ \mbox{(Lemma \ref{lem:controle:Rmun:diagRmun:bar})} & \\ 
 & \leq  C_{sup} + C_{sup} \frac{1}{\Delta} \int_{C_k} dt_k   \sqrt{N_k} N_k || \bar{R}_{k,\theta}^{-1} r_{\theta}(t_k,X^k) ||  ~ ~ ~ \mbox{(Conditions \ref{cond:Kzero:deltazero} and \ref{cond:Ktheta:deltatheta} and Lemma \ref{lem:norm:A:b})} & \\
   & \leq  C_{sup} \left( 1 + N_k^2 \right)
~ ~ ~   \mbox{(Lemma \ref{lem:controle:Rmun:diagRmun:bar})}. & 
\end{flalign*}

Hence $\EE \left( \left. g^2_{k}(X) \right| N_k = N \right) \leq C_{sup} (1 + N^4)$, so that we can complete the proof with Lemma \ref{lem:cvg:mean:fk}.

\end{proof}

\begin{lem} \label{lem:diff:mean:CV:En}
Consider a fixed $\theta \in \Theta$. Then
\[
\EE(CV_{\theta})  - \EE( E_{n,\theta} ) - \delta_0
\]
goes to $0$ as $n \to \infty$.
\end{lem}

\begin{proof}[Proof of Lemma \ref{lem:diff:mean:CV:En}]

Let us consider a random observation point $X_{n+1}$ with uniform distribution on $[0,n^{1/d}]^d$. Let us also consider a Gaussian variable $\epsilon_{n+1}$ with mean $0$ and variance $\delta_0$. Consider that $X_{n+1}$ and $\epsilon_{n+1}$ are independent and independent of $X$, $Y$ and $\epsilon$. With the same argument as in the proof of Lemma \ref{lem:blockwise:Entheta}, we have
\begin{eqnarray*}
\EE \left( E_{n,\theta} \right) & = & \EE \left( \left[ Y(X_{n+1}) - \hat{y}_{\theta}(X_{n+1}) \right]^2 \right), 
\end{eqnarray*}
where we remind that $\hat{y}_{\theta}(X_{n+1}) = r_{\theta}^t(X_{n+1}) R_{\theta}^{-1} y$, with $r_{\theta}(X_{n+1}) = (K(X_1,X_{n+1}),...,K(X_n,X_{n+1}))^t$. Now, by symmetry of the roles of $X_1,...,X_{n+1}$ and $\epsilon_1,...\epsilon_{n+1}$, we have
\begin{eqnarray*}
\EE \left( CV_{\theta} \right) & = & \EE \left( \left[ Y(X_{n+1}) - \hat{y}_{n-1,\theta}(X_{n+1}) \right]^2 \right) + \delta_0, 
\end{eqnarray*}
with $\hat{y}_{n-1,\theta}(X_{n+1}) = r_{n-1,\theta}^t \tilde{R}_{n-1,\theta}^{-1} y$, with $r_{n-1,\theta} = (K(X_1,X_{n+1}),...,K(X_{n-1},X_{n+1}),0)^t$ and 
\[
\tilde{R}_{n-1,\theta} =
\begin{pmatrix}
(K_{\theta}(X_i,X_j))_{i,j=1,...,(n-1)} + \delta_{\theta} I_{n-1} & 0 \\
0 & 1
\end{pmatrix}.
\] 

Hence, using Cauchy-Schwarz
\begin{flalign} \label{eq:in:proof:diff:mean:CV:En:a}
& \left| \EE(CV_{\theta})  - \EE( E_{n,\theta} ) - \delta_0 \right| & \nonumber \\
& \leq 
\sqrt{ \EE \left( \left[ r_{\theta}^t(X_{n+1}) R_{\theta}^{-1} y -  r_{n-1,\theta}^t \tilde{R}_{n-1,\theta}^{-1} y \right]^2 \right) }
\sqrt{ \EE \left( \left[ r_{\theta}^t(X_{n+1}) R_{\theta}^{-1} y +  r_{n-1,\theta}^t \tilde{R}_{n-1,\theta}^{-1} y - 2 Y(X_{n+1})\right]^2 \right) }. &
\end{flalign}
The second term in \eqref{eq:in:proof:diff:mean:CV:En:a} is shown to be bounded with techniques similar to but simpler than in the proof of Lemma \ref{lem:control:derivative:CV}. The first term in \eqref{eq:in:proof:diff:mean:CV:En:a} is shown to go to zero with techniques similar to but simpler than in the proof of Lemma \ref{lem:blockwise}.

\end{proof}

\begin{cor} \label{cor:controle:derivative:Entheta}
For any $i=1,...,p$,
\[
\EE \left( \sup_{\theta \in \Theta} \left| \frac{\partial}{\partial \theta_i}   E_{n,\theta} \right| \right),
~ ~ ~
\EE \left( \sup_{\theta \in \Theta} \left| \frac{\partial}{\partial \theta_i} \EE \left( \left.  E_{n,\theta} \right| X \right) \right| \right)
~ ~ ~
\mbox{and}
~ ~ ~
 \sup_{\theta \in \Theta} \left| \frac{\partial}{\partial \theta_i}  \EE \left(  E_{n,\theta}  \right) \right|
\]
are bounded w.r.t $n$.
\end{cor}

\begin{proof}[Proof of corollary \ref{cor:controle:derivative:Entheta}]
We have
\begin{eqnarray*}
\EE \left( \sup_{\theta \in \Theta} \left| \frac{\partial}{\partial \theta_i}   E_{n,\theta} \right| \right) 
& = & \EE \left( \sup_{\theta \in \Theta} \left| \frac{\partial}{\partial \theta_i}   \frac{1}{n} \int_{[0,n^{1/d}]^d}  \left[ Y(t) - \hat{y}_{\theta}(t) \right]^2 dt \right| \right) .
\end{eqnarray*}
For fixed $n$ we can exchange derivative and integration, so we obtain
\begin{eqnarray*}
\EE \left( \sup_{\theta \in \Theta} \left| \frac{\partial}{\partial \theta_i}   E_{n,\theta} \right| \right) 
& = & \EE \left( \sup_{\theta \in \Theta} \left|    \frac{1}{n} \int_{[0,n^{1/d}]^d}  \frac{\partial}{\partial \theta_i} \left[ Y(t) - \hat{y}_{\theta}(t) \right]^2 dt \right| \right)  \\
& \leq & \EE \left(     \frac{1}{n} \int_{[0,n^{1/d}]^d}  \sup_{\theta \in \Theta} \left|  \frac{\partial}{\partial \theta_i} \left[ Y(t) - \hat{y}_{\theta}(t) \right]^2 \right| dt  \right). \\
\end{eqnarray*}
Hence, by considering $t$ as a new random observation point $X_{n+1}$ as in the proof of Lemma \ref{lem:blockwise:Entheta}, we show the first bound of the lemma as in the proof of Lemma \ref{lem:control:derivative:CV}, the only difference being that there are $n+1$ observation points instead of $n$.
The second and third bounds are proved as in the proof of corollary \ref{cor:controle:derivative}.

\end{proof}

\begin{proof}[Proof of Theorem \ref{theorem:oracle:CV}]

We have
\begin{flalign*}
& \sup_{\theta \in \Theta} \left| CV_{\theta} - \delta_0 - E_{n,\theta} \right| & \\
& \leq \sup_{\theta \in \Theta} \left| CV_{\theta} - \EE(CV_{\theta}|X) \right| + \sup_{\theta \in \Theta} \left| E(CV_{\theta} |X) - \EE(CV_{\theta}) \right| + \sup_{\theta \in \Theta} \left| \EE(CV_{\theta}) - \delta_0 - \EE(E_{n,\theta}) \right| & \\
& ~ ~ + \sup_{\theta \in \Theta} \left| \EE(E_{n,\theta}) -\EE(E_{n,\theta}|X) \right| + \sup_{\theta \in \Theta} \left| \EE(E_{n,\theta}|X) - E_{n,\theta} \right|.  &
\end{flalign*}
The five terms in the right-hand size of the above equation go to $0$ in probability. Indeed, for fixed $\theta$, the functions of $\theta$ go to $0$ in probability because of Lemmas \ref{lem:var:given:X:CVtheta}, \ref{lem:var:ECV:sachant:X}, \ref{lem:var:given:X:Entheta}, \ref{lem:var:E:Entheta:sachant:X} and \ref{lem:diff:mean:CV:En}. The convergence of the supremums over $\theta$ to $0$ is then a consequence of the fact that $\Theta$ is compact and of Lemma \ref{lem:control:derivative:CV} and corollaries \ref{cor:controle:derivative} and \ref{cor:controle:derivative:Entheta}.
Finally, since $\hat{\theta}_{CV}$ minimizes $CV_{\theta} + \delta_0$, we conclude with, for any $\theta \in \Theta$
\begin{eqnarray*}
E_{n,\hat{\theta}_{CV}} - E_{n,\theta} & \leq & CV_{\hat{\theta}_{CV}} - CV_{\theta} + 2 \sup_{\theta \in \Theta} \left| CV_{\theta} - \delta_0 - E_{n,\theta} \right| \\ 
& \leq & 2 \sup_{\theta \in \Theta} \left| CV_{\theta} - \delta_0 - E_{n,\theta} \right|. \\
\end{eqnarray*}

Hence
\[
\sup_{\theta \in \Theta} \left( E_{n,\hat{\theta}_{CV}} - E_{n,\theta} \right) = o_p(1).
\]

\end{proof}

\subsection{Technical results} \label{subsection:technical:results}

The following technical results are proved in the supplementary material.

\begin{lem} \label{lem:sobolev}
Consider a fixed number $n$ of observation points.
Consider a function $f_{\theta}(X,y)$ that is $p$ times continuously differentiable w.r.t $\theta$ for any $X,y$ and so that, for $i_1+...+i_p \leq p$, 
\[\sup_{\theta} \left| (\partial^{i_1}/\partial \theta_1^{i_1}) ...(\partial^{i_p}/\partial \theta_p^{i_p}) f_{\theta}(X,y) \right|
\]
has finite mean value w.r.t $X$ and $y$. Then, there exists a constant $C_{sup}$ (depending only of $\Theta$) so that
\[
\EE \left( \sup_{\theta \in \Theta} \left| f_{\theta}(X,y) \right|  \right) \leq C_{sup} \sum_{i_1+...+i_p \leq p}
\int_{\Theta} \EE \left( \left| \frac{\partial^{i_1}}{\partial \theta_1^{i_1}} ...\frac{\partial^{i_p}}{\partial \theta_p^{i_p}} f_{\theta}(X,y) \right| \right) d \theta.
\]
\end{lem}

\begin{lem} \label{lem:integrale:abc}
There exists a finite constant $C_{sup}$ so that, for any $a,b \in \RR^d$,
\[
\int_{\RR^d} \frac{1}{1+ |a-c|^{d+1}} \frac{1}{1+|b-c|^{d+1}} dc \leq C_{sup} \frac{1}{1+|a-b|^{d+1}}.
\]
\end{lem}

\begin{lem}  \label{lem:norm:Aun...Ak:carre}
Let $0 < C_{inf} \leq C_{sup} < \infty$ be fixed independently of $n$. Let $s_n$ be a function of $n$ so that $s_n \in \NN^*$ and $ C_{inf} n \leq s_n \leq C_{sup} n$. Consider $s_n$ observation points $\bar{X}_1,...,\bar{X}_{s_n}$, independent and uniformly distributed on $[0,n^{1/d}]^d$.
Let $A_1,...,A_k$ be $k$ sequences of $s_n \times s_n$ random matrices so that, for $l=1,...,k$, $(A_l)_{i,j}$ depends only on $\bar{X}_i$ and $\bar{X}_j$ and satisfies $|(A_l)_{i,j} | \leq 1/(1+|\bar{X}_i-\bar{X}_j|^{d+1})$. Then $\EE_X \left( |A_1...A_k|^2 \right)$ is bounded w.r.t. $n$.
\end{lem}

\begin{lem}  \label{lem:controle:Rmun:diagRmun}
The supremum over $n$, $\theta$ and $X$ of the eigenvalues of $R^{-1}_{\theta}$, $R_{1,\theta}^{-1}$, $diag(R^{-1}_{\theta})$, $diag(R_{1,\theta}^{-1})$, $diag(R^{-1}_{\theta})^{-1}$ and $diag(R_{1,\theta}^{-1})^{-1}$ is smaller than a constant $C_{sup}<+ \infty$.
\end{lem}

\begin{lem} \label{lem:controle:Rmun:diagRmun:tilde}
Lemma \ref{lem:controle:Rmun:diagRmun} also holds when $K_{\theta}$ is replaced by $\Kt_{\theta}$ of Definition \ref{def:blockwise}.
\end{lem}

\begin{lem} \label{lem:controle:Rmun:diagRmun:bar}
Lemma \ref{lem:controle:Rmun:diagRmun} also holds when $R_{\theta}$ is replaced by $\Rb_{k,\theta}$ of Definition \ref{def:blockwise}.
\end{lem}

\begin{lem}  \label{lem:alternate:Rmun}
Let $k \in \NN$. Let $A_{1,\theta},...,A_{k,\theta}$ be $k$ sequences of symmetric random matrices (functions of $X$ and $\theta$) so that, for any $m \in \NN$, $a_1,...,a_m \in \{1,...,k\}$, $\sup_{\theta \in \Theta} \EE_X \left| A_{a_1,\theta}...A_{a_m,\theta} \right|^2$ is bounded (w.r.t $n$). Let $B_{1,\theta},...,B_{k+1,\theta}$ be $k+1$ sequences of random symmetric non-negative matrices (functions of $X$ and $\theta$) so that $\sup_{\theta}||B_{1,\theta}||,...,\sup_{\theta}||B_{k+1,\theta}||$ are bounded (w.r.t $n$ and $X$). Then 
\[
\sup_{\theta \in \Theta} \EE_{X} \left| B_{1,\theta} A_{1,\theta} B_{2,\theta}...B_{k,\theta}A_{k,\theta}B_{k+1,\theta} \right|^2
\]
is bounded w.r.t $n$.
\end{lem}

\begin{lem} \label{lem:norm:Run:moins:Rt1:carre:carre}
Consider a fixed $\theta \in \Theta$. With the notation of Definition \ref{def:blockwise}, we have, when $n_2 = o(n)$,
\[
\EE \left( \left| ( R_{1,\theta} - \Rt_{1,\theta} )^2 \right|^2 \right) \to_{n \to \infty} 0.
\]
\end{lem}

\begin{lem} \label{lem:integrale:abc:T}
Let $C(t)$ be as in Definition \ref{def:blockwise}. Define, for $T \geq 0$, $f(T) = \int_{\RR^d \backslash [-T,T]^d}1/(1+|t|^{d+1}) dt$. Define, for $x \in [0,n^{1/d}]^d$, $D_{\Delta}(x) = \inf_{t \in \RR^d \backslash C(x)} |x-t|$. Define $D_{\Delta}(x_1,...,x_m) = \min_{i=1,...,m}D_{\Delta}(x_i)$.
Then, there exists a finite constant $C_{sup}$ so that, for any $n$, for any $x_1,x_2 \in [0,n^{1/d}]^d$,
\[
\int_{\RR^d} \frac{1}{1+|x_1 - x|^{d+1}} \frac{1}{1+|x_2- x|^{d+1}} \indun_{C(x) \neq C(x_1)} \indun_{C(x) \neq C(x_2)} dx \leq C_{sup} f(D_{\Delta}(x_1,x_2)) \frac{1}{1+|x_1-x_2|^{d+1}}.
\]
\end{lem}

\begin{lem} \label{lem:integrale:abf}

Use the notation $n_2,\Delta$, $C(t)$, $f(T)$ and $D_{\Delta}(x_1,x_2)$ of Definition \ref{def:blockwise} and Lemma \ref{lem:integrale:abc:T}. Then, when $n_2 = o(n)$,
\[
\frac{1}{n} \int_{[0,n^{1/d}]^d} dx_1 \int_{[0,n^{1/d}]^d} dx_2 \frac{1}{1+|x_1-x_2|^{d+1}} f(D_{\Delta}(x_1,x_2)) \to_{n \to +\infty} 0.
\]

\end{lem}

\begin{lem} \label{lem:dist:fk}
Use the notation $n_2$, $\Delta$ and $C_1,...,C_{n_2}$ of Definition \ref{def:blockwise}. Let, for $i=1,...,n_2$, $X_1^{i},...,X_{N_i}^{i}$ be the $N_i$ components of $X$ that are in $C_i$ (so that the order of their indexes in $X$ is preserved). Then

\begin{itemize}
\item[i)] For $i=1,...,n_2$, $N_i$ follows a binomial $B(n,1/n_2)$ distribution. For any $i,j =1,...,n_2; i \neq j$, conditionally to $N_i = k_i$, $N_j$ follows a binomial $B(n-k_i,1/(n_2-1))$ distribution.
\item[ii)] Conditionally to $N_i = k_i$, $X_1^{i},...,X_{k_i}^{i}$ are independent and uniformly distributed on $C_i$.
\item[iii)] For $1 \leq i \neq j \leq n_2$, conditionally to $N_i = k_i,N_j=k_j$, the sets of random variables $(X_1^{i},...,X_{k_i}^{i})$ and  $(X_1^{j},...,X_{k_j}^{j})$ are independent, and their components are independent and uniformly distributed on $C_i$ and $C_j$ respectively.
\end{itemize}

Consider $n_2$ real-valued functions $f_1,...,f_{n_2}$ of $X$ that can be written $f_i(X) = \barf(N_i,X_1^{i},...,X_{N_i}^{i}) $, and so that, for any $t \in \RR^d$, $x_1,...,x_{N} \in \RR^d$, $\barf(N,x_1+t,...,x_N+t) = \barf(N,x_1,...,x_N)$. Then
\begin{itemize}
\item[iv)] The variables $f_1(X),...,f_{n_2}(X)$ have the same distribution. The couples $(f_i(X),f_j(X))$, for $1 \leq i \neq j \leq n_2$, have the same distribution.
\end{itemize}
\end{lem}

\begin{lem} \label{lem:cvg:mean:fk}
Use the notation of Lemma \ref{lem:dist:fk}, and consider $n_2$ functions $f_1,...,f_{n_2}$ that satisfy the conditions of Lemma \ref{lem:dist:fk}. Assume that there exist fixed even natural numbers $q,l$ and a finite constant $C_{sup}$ (independent of $n$ and $X$) so that $\EE\left( f_i^2(X) | N_i = k \right) \leq C_{sup} (1 + k^q + k^{q+l}/\Delta^l)$. Then, if $ \Delta \to_{n \to \infty} + \infty$ and $\Delta = O(n^{1/(2q+5)})$,
\[
var \left( \frac{1}{n_2} \sum_{i=1}^{n_2} f_i(X) \right) \to_{n \to \infty} 0.
\]
\end{lem}

\begin{lem} \label{lem:moment:binomial}
Let $N$ follow the binomial distribution $B(n,1/n_2)$, with $n/n_2 = \Delta \to_{n \to \infty} + \infty$. Then, for any $k \in \NN$, there exists a finite constant $C_{sup}$, independent of $n$, so that 
\[
\EE \left( N^k \right) \leq C_{sup} \Delta^k.
\]
\end{lem}

\begin{lem} \label{lem:sum:d:Ci:Cj}
Let $n_2$, $\Delta$ and $C_1,...,C_{n_2}$ be as in Definition \ref{def:blockwise}. Assume that $\Delta$ is lower bounded, as a function of $n$. Then, there exists a finite constant $C_{sup}$ so that for any $n$, $i \in \{1,...,n_2\}$,
\[
\sum_{j=1}^{n_2} \frac{1}{1+d(C_i,C_j)^{d+1}} \leq C_{sup}.
\]
\end{lem}

\begin{lem} \label{lem:norm:A:b}
Let $A$ be a real $m_1 \times m_2$ matrix and $b$ be a $m_2$-dimensional real column vector. Then
\[
||Ab||^2 \leq m_1 m_2 \left( \max_{i,j} A_{i,j}^2 \right) ||b||^2.
\]
\end{lem}

\section*{Acknowledgements}
The research leading to this paper was partly carried out when the author was affiliated to the University of Vienna. The author thanks Josselin Garnier, Benedikt P\"otscher and Luc Pronzato for constructive discussions.

\section*{Supplementary material}
In the supplementary material, we give the proof of the lemmas stated in Section \ref{subsection:technical:results}.

\bibliographystyle{apalike}
\bibliography{Biblio}

\end{document}